\documentclass[a4paper,11pt,reqno]{amsart}

\usepackage[utf8]{inputenc}
\RequirePackage[l2tabu, orthodox]{nag}
\usepackage[T1]{fontenc}
\usepackage{lmodern}
\usepackage{amsthm,amssymb,amsmath}
\usepackage{latexsym}
\usepackage{graphicx}
\usepackage{subfigure}
\usepackage{tikz}
\usetikzlibrary{arrows}
\usetikzlibrary{calc}
\usepackage[linktocpage=true]{hyperref}
\usepackage[left=3.5cm, right=3.5cm, top=3cm, bottom=3cm]{geometry}
\usepackage{color}
\usepackage{here}
\usepackage{mathrsfs}
\usepackage{pgf,tikz}
\usetikzlibrary{decorations.pathreplacing, shapes.multipart, arrows, matrix, shapes}
\usetikzlibrary{patterns}
\usepackage{caption}
\usepackage[english]{babel}
\usepackage[all] {xy}
\usepackage{enumitem}
\usepackage{cancel}
\usepackage{comment}
\usepackage{verbatim}

\makeatletter
\newcommand\footnoteref[1]{\protected@xdef\@thefnmark{\ref{#1}}\@footnotemark}
\makeatother

\hypersetup{ colorlinks   = true, 
	urlcolor  = blue, 
	linkcolor    = blue, 
	citecolor   = blue 
}

\date{}

\def\beq{\begin{equation}}
	\def\eeq{\end{equation}}

\def\tW{\widetilde W}
\def\det{\textrm{det}\ }

\newcommand\restr[2]{{ 
  \left.\kern-\nulldelimiterspace 
  #1 
  \vphantom{\big|} 
  \right|_{#2} 
  }}

\newcommand{\Z}{{\mathbb Z}}
\newcommand{\R}{{\mathbb R}}

\newcommand{\N}{{\mathbb N}}

\newcommand{\TT}{{\mathbb{T}^2}}
\newcommand{\TTf}{{\mathbb{T}_f^2}}
\newcommand{\loc}{{\operatorname{loc}}}

\newcommand{\cA}{{\mathcal A}}
\newcommand{\cC}{{\mathcal C}}

\newcommand{\cB}{{\mathcal B}}

\newcommand{\cL}{{\mathcal L}}
\newcommand{\cM}{{\mathcal M}}
\newcommand{\cN}{{\mathcal N}}
\newcommand{\cP}{{\mathcal P}}

\newcommand{\cW}{{\mathcal W}}

\newcommand{\cR}{{\mathcal R}}

\newcommand{\eps}{\varepsilon}
\newcommand{\card}{\operatorname{card}}

\newcommand{\eqdef}{\stackrel{\scriptscriptstyle\textrm def.}{=}}

\newcommand{\Leb}{\textrm{Leb}}

\newcommand{\tix}{\tilde{x}}
\newcommand{\ta}{\tilde{a}}
\newcommand{\tb}{\tilde{b}}
\newcommand{\tif}{\tilde{f}}
\newcommand{\tm}{\tilde{\mu}}

\newcommand{\w}{\widetilde{\cW}}
\makeatletter
\newcommand{\tpitchfork}{%
	\vbox{
		\baselineskip\z@skip
		\lineskip-.52ex
		\lineskiplimit\maxdimen
		\m@th
		\ialign{##\crcr\hidewidth\smash{$-$}\hidewidth\crcr$\pitchfork$\crcr}
	}%
}
\makeatother

\newtheorem{theorem}{Theorem}[section]
\newtheorem{corollary}{Corollary}
\newtheorem{maintheorem}{Theorem}

\newtheorem{lemma}[theorem]{Lemma}
\newtheorem{proposition}{Proposition}

\newtheorem{problem}{Problem}
\theoremstyle{definition}
\newtheorem{definition}[theorem]{Definition}
\newtheorem{remark}{Remark}
\newtheorem*{notation}{Notation}
\newtheorem{example}{Example}

\begin{document}

\title[$\mathbf{\textit{U}}$-Gibbs rigidity for PH endomorphisms on surfaces]{$\mathbf{\textit{U}}$-Gibbs measure rigidity for partially hyperbolic endomorphisms on surfaces}

\author[Marisa Cantarino and Bruno Santiago]{Marisa Cantarino and Bruno Santiago}
\address{}
\email{}

\date{\today}

\begin{abstract} 
We prove that, for a $C^2$ partially hyperbolic endomorphism of the 2-torus which is strongly transitive, given an ergodic $u$-Gibbs measure that has positive center Lyapunov exponent and has full support, then either the map is special (has only one unstable direction per point), or the measure is the unique absolutely continuous invariant measure. We can apply this result in many settings, in particular obtaining uniqueness of $u$-Gibbs measures for every non-special perturbation of irreducible linear expanding maps of the torus with simple spectrum. This gives new open sets of partially hyperbolic systems displaying a unique $u$-Gibbs measure.
\end{abstract}

\maketitle

\tableofcontents

\section{Introduction}
\label{sec:intro}

Measure rigidity is expressed by the following general idea: when analyzing some set of invariant measures for a family of dynamical systems, one gets uniqueness or a very restrict subset of measures by imposing some condition on the family. The most basic example is given by rotations of the circle: the restriction to irrational rotations makes the set of invariant measures turn into a singleton, for in the irrational case the system is uniquely ergodic, and all rotations preserve Lebesgue measure. Measure rigidity results are fundamental in many applications of dynamics, including problems in number theory, see for instance \cite{EKL} or the related survey \cite{venkatesh} and the references therein.

Along the last decades, measure rigidity results have been explored by several authors. In homogeneous dynamics some impressive accomplishments have been obtained, such as Furstenberg's proof of the unique ergodicity of the horocycle flow in constant negative curvature \cite{Furst}, which was deeply generalized in the series of works by Marina Ratner in the 90's \cite{Rat1}, \cite{Rat2}, \cite{Rat3}, \cite{Rat4}. She proved measure classification results for measures invariant under unipotent flows (the action of a unipotent subgroup). After that, in another landmark series of works, Yves Benoist and Jean-François Quint proved \cite{BenoistQuintI,BenoistQuintII,BenoistQuintIII} measure rigidity results for stationary measures of a Zariski dense random walk on homogeneous spaces.

These works were the inspirations in Teichmüller dynamics for the remarkable work of Alex Eskin and Maryam Mirzakhani \cite{EskinMirzakhani} about group actions on moduli spaces. In \cite{EskinMirzakhani}, they developed the now called ``factorization method'' to prove measure rigidity for ergodic measures that are invariant under the action of the upper triangular subgroup of $\operatorname{SL}_2(\Z)$. This method was after used back in the homogeneous setting by Eskin and Elon Lindenstrauss \cite{EskinLind}, where they generalized the results of Benoist--Quint.

In smooth dynamics, Aaron Brown and Federico Rodriguez-Hertz \cite{BRH} used many ideas from \cite{EskinMirzakhani} and \cite{BenoistQuintI} to classify stationary measures for random iterations of $C^2$ surface diffeomorphisms. This result has been generalized to higher dimensions by Brown--Rodriguez-Hertz in collaboration with Simion Filip and Eskin in the recent preprint \cite{befh}. It is worth to remark that the concept of uniform expansion plays a prominent role in this generalization, see \cite{chung} where the author relates this concept with \cite{BRH}.

In \cite{Obata} Davi Obata translated the result and the ideas of \cite{BRH} to the setting of 4-dimensional partially hyperbolic diffeomorphisms of skew-product type. In Obata's work, classification of stationary measures is translated as \emph{classification of $u$-Gibbs measures}. These are measures which are smooth along unstable manifolds. Recall that stationary measures can be regarded as ``smooth'' along \emph{symbolic} unstable manifolds. The result of \cite{Obata} was used in \cite{CPO} to obtain open sets of partially hyperbolic Anosov systems displaying uniqueness of $u$-Gibbs measures. 

For $C^{\infty}$ partially hyperbolic systems with one-dimensional center, in \cite{Katz} Asaf Katz developed the machinery of the factorization method of \cite{EskinMirzakhani} and proved a measure rigidity result for $u$-Gibbs measures satisfying a \emph{quantitative non-integrability (QNI)} condition. Using Katz's result, Artur Avila, Sylvain Crovisier, Alex Eskin, Rafael Potrie, Amie Wilkinson and Zhiyan Zhang, in an ongoing project \cite{ACEPWZ}, announced that, for $C^{\infty}$ partially hyperbolic Anosov diffeomorphisms with expanding center in dimension 3, either one has that the stable and the unstable bundle jointly integrate or the system displays a unique $u$-Gibbs measure.  

A different approach was taken in the work \cite{ALOS}, by Obata with Sébastien Alvarez, Martin Leguil and Bruno Santiago. Assuming regularity of stable holonomies, they took ideas from \cite{BRH,EskinMirzakhani,EskinLind} to show that for $C^2$ partially hyperbolic Anosov diffeomorphisms, with expanding center and close to volume preserving, in dimension three, either the stable and the unstable bundle jointly integrate or there exists a unique fully supported $u$-Gibbs measure. In particular, the result of \cite{ALOS} does not depend on the QNI condition of Katz. 

In this paper we aim to explore the connections between partially hyperbolic three dimensional diffeomorphisms and two dimensional partially hyperbolic \emph{endomorphisms} in order translate the ideas of \cite{ALOS} to the \emph{non-invertible} setting. We obtain measure rigidity of fully supported $u$-Gibbs measures for non-special partially hyperbolic and strongly transitive two dimensional endomorphisms. It is worth remarking that, contrarily to \cite{ALOS}, we do not assume uniform expansion along the center, and thus our result is also a generalization of \cite{ALOS}. Also, Katz in \cite{Katz} do not assume unifom expansion on the center bundle, nonetheless his arguments depend havily on high regularity of the dynamics, while our paper only assumes $C^2$. Moreover, we obtain that the measure is in fact absolutely continuous and physical, and thus our result dialogues with \cite{Tsujii}: instead of a generic transversality condition, we assume a mild qualitative transversality (being non-special) and positivity of center exponent in order to promote a $u$-Gibbs meas0ure to be physical. Notice that the full support assumption is automatically satisfied once one knows that every unstable curve has a dense orbit. Thus, our main result reduces the measure rigidity problem to a topological problem. Also, by a result announced in \cite{ACEPWZ} in many situations we known that this topological condition is satisfied. In particular, we obtain open sets of partially hyperbolic endomorphisms displaying a unique $u$-Gibbs measure.  

\subsection{Statement of the main result}
\label{sec:intro-statement}
We give here quick definitions of the objects appearing in our main result. For a comprehensive treatment, see Section~\ref{sec:ph-exp}. Denote by $\operatorname{End}^r(\TT)$ the set $C^r$ local diffeomorphisms $f:\TT\to\TT$. We say that an element $f\in\operatorname{End}^1(\TT)$ is \emph{partially hyperbolic} if there exist $\sigma>1$, $\ell\in\N$ and a continuous family $x\in\TT\mapsto\cC^u_x$ of \emph{unstable cones}: for each $x\in\TT$, $\cC^u_x$ is a closed cone in $T_x\TT$ such that 
\begin{enumerate}
    \item[(i)] $Df^{\ell}(x)\cC^u_x\subset\operatorname{Int}(\cC^u_{f^{\ell}(x)})\cup\{0\}$
    \item[(ii)] $\|Df^{\ell}(x)v\|\geq\sigma\|v\|$, for every $v\in\cC^u_x$.
\end{enumerate}

An endomorphism can be seen as a kind of dynamical system for which the future is deterministic but the past is random. In this way, for a given point $x\in\TT$ there exist uncountably many different sequences $\tix=\{x_i\}_{i\in\Z}$ satisfying $x_0=x$ and $f(x_i)=x_{i+1}$ for all $i \in \Z$. When $f\in\operatorname{End}^1(\TT)$ is partially hyperbolic, for each such an \emph{orbit} $\tix=\{x_i\}_{i\in\Z}$ there exist two line bundles $E^c(\tix)\oplus E^u(\tix)$, defined along the sequence, and invariant under the differential $Df$. The bundle $E^u(\tix)$ is contained in the unstable cone $\cC^u_x$ while the \emph{center bundle} $E^c(\tix)$ is transverse to it. In particular, $E^u$ is uniformly expanded and $E^c$ is \emph{dominated} by $E^u$: vectors in $E^c$, whenever they are expanded they do so (exponentially) slower than the vectors in fast unstable bundle $E^u$. 

An important fact is that the center bundle is entirely deterministic, for it does not depend on the past, i.e. $E^c(\tix)=E^c(\ta)$ whenever $a_0=x_0$. The unstable bundle, on the contrary, in general changes with the past orbit. If, for a given partially hyperbolic $f\in\operatorname{End}^1(\TT)$, the unstable bundle satisfies $E^u(\tix)=E^u(\ta)$ whenever $a_0=x_0$, we say that $f$ is a\emph{special} endomorphism. We observe that being non-special is a $C^1$ open and $C^r$ dense condition (see \cite{CostaMicena2022} for $r=1$ and a local argument and \cite{he2017accessibility} for a general argument in the volume preserving case) among partially hyperbolic $f\in\operatorname{End}^r(\TT)$.  

Given an ergodic $f$-invariant measure $\mu$, we define its \emph{center exponent} by
\[
L^c(\mu)\eqdef\int\log\|Df(x)|_{E^c}\|d\mu(x).
\]
Notice that, since the center bundle is deterministic, the above is well defined. 

When $f$ is non-special the bundle $E^u$ does not integrate into an $f$-invariant foliation. Nevertheless, we can construct unstable manifolds using the \emph{natural extension} (see Section~\ref{sec:inv-lim}): given an orbit $\tix=\{x_i\}_{i\in\Z}$ one can apply the graph transform method along the sequence $\{x_i\}_{i<0}$ to construct an unstable manifold $\cW^u(\tix)$ attached to the orbit $\tix$ (see \cite{QXZ2009}). We say that an ergodic $f$-invariant probability measure $\mu$ is a $u$-\emph{Gibbs measure} if the conditionals of $\mu$ along unstable manifolds are absolutely continuous with respect to the inner Lebesgue measure of the leaves (see Section~\ref{sec:inv-lim-ph}). 

Finally, we say that $f\in\operatorname{End}^1(\TT)$ is \textit{strongly transitive} if for every open set $U\subset\TT$ there exists an integer $N>0$ such that {$\cup_{n = 0}^{N}f^n(U)=\TT$}. We can now state our main result.

\begin{maintheorem}
	\label{teo:main.rigidez}
Let $f$ be a $C^2$ partially hyperbolic strongly transitive endomorphism of $\TT$. Let $\mu$ be an ergodic $u$-Gibbs measure of $f$ with positive center exponent. Assume that $f$ is not special and that $\operatorname{supp}(\mu)=\TT$. Then, $\mu$ is the unique absolutely continuous invariant measure for $f$. 
\end{maintheorem}

We do not know if a map $f$ satisfying the assumptions of our main result is \emph{mostly expanding} in the sense of \cite{AndVas}, i.e., if the existence of \emph{one fully supported} $u$-Gibbs measure with positive center exponent would imply that every $u$-Gibbs measure have positive center exponent. We also don't know whether  the results of \cite{AndVas} hold true in the non-invertible setting, although we believe that both answers are yes. Regardless, our result can be seen also as giving a new criteria for existence and uniqueness of absolutely continuous invariant probability measures for partially hyperbolic endomorphisms in $\TT$ (see Corollary~\ref{cor:main-dynamicallymininmal} below), going beyond the measure rigidity.

We also remark that one cannot drop the full support and transitivity assumptions simultaneously. Indeed, consider the product of $f: \mathbb{S}^1 \to \mathbb{S}^1$ a north pole-south pole diffeomorphism and $g: \mathbb{S}^1 \to \mathbb{S}^1$ a uniformly expanding map with stronger expansion than $f$. This system is partially hyperbolic, not transitive, and the Lebesgue measure on the circle $\{N\} \times \mathbb{S}^1$ given by $\delta_N \times Leb_{\mathbb{S}^1}$ has expanding center exponent and it is $u$-Gibbs. We can perturb this system to a non-special one which has the same features (as in Section \ref{subsec:provasexe3}), being also not transitive. This phenomenon can not occur on the uniformly expanding case. Nonetheless, even a uniformly expanding (therefore strongly transitive) system can have u-Gibbs measures that do not have full support, as we explore in Section \ref{sec:ex}.

\subsection{Some applications}
\label{sec:ex}

The most important particular case of Theorem~\ref{teo:main.rigidez} concerns the case in which every unstable curve has a dense forward orbit. More precisely, we say that a partially hyperbolic element $f\in\operatorname{End}^1(\TT)$ has \emph{dynamically minimal unstable leaves} if, for every unstable manifold $\gamma=\cW^u(\tix)$, one has $\overline{\cup_{n\in\N}f^n(\gamma)}=\TT$. Given an ergodic, $f$-invariant $u$-Gibbs measure $\mu$, for almost every $x\in\operatorname{supp}(\mu)$ there exists at least one orbit $\tix=\{x_i\}_{i\in\Z}$, with $x_0=x$, so that the associated unstable manifold $\cW^u(\tix)$ is contained in the support of $\mu$. In particular, if $f$ has dynamically minimal unstable leaves then every $u$-Gibbs measure is fully supported. Therefore, we deduce from our main result that

\begin{corollary}
    \label{cor:main-dynamicallymininmal}
Let $f\in\operatorname{End}^2(\TT)$ be partially hyperbolic, strongly transitive with dynamically minimal unstable leaves. Assume that $f$ is non-special. Then, if there exists a $u$-Gibbs measure $\mu$ with positive center exponent, then $\mu$ is the unique absolutely continuous invariant measure of $f$.
\end{corollary}

This Corollary reduces the measure rigidity problem of $u$-Gibbs measures with positive center exponent to the topological problem of proving dynamical minimality of unstable leaves. 

\subsubsection{Perturbations of an irreducible linear expanding map with simple spectrum} An interesting case in which Corollary~\ref{cor:main-dynamicallymininmal} can be applied concerns perturbation of the induced map $f_B:\TT\to\TT$ where $B\in\operatorname{SL}_2(\Z)$ has two irrational eigenvalues $\sigma^u>\sigma^c>1$ (see Example~\ref{ex:2}). This map is partially hyperbolic, strongly transitive and special. In fact, $f_B$ is \emph{uniformly expanding} in the sense that $\min_{x\in\TT}\|Df(x)|_{E^c}\|>1$. It is a well known fact that this is an open condition which implies that every perturbation $f$ of $f_B$ is topologically conjugate: there exists a continuous change of coordinates $h\in\operatorname{Homeo}(\TT)$ under which $f$ and $f_B$ are indistinguishable, i.e. $h\circ f=f_B\circ h$.    
A corollary of the work announced by \cite{ACEPWZ} is the following result.

\begin{theorem}[\cite{ACEPWZ}]
Consider $f\in\operatorname{End}^2(\TT)$ a partially hyperbolic and uniformly expanding map which is topologically conjugate to $f_B$. Then, $f$ has dynamically minimal unstable leaves.
\end{theorem}

\begin{remark}
The above result can be proved using the $s$-transversality argument from \cite{ACW}. Indeed, a sketch of the proof goes as follows: if $f$ is special then by an argument from \cite{yu} the conjugacy with $f_B$ preserves the unstable foliation and the minimality follows. If $f$ is not special, then one shows that $s$-transversality always hold. This argument is called Rafael Potrie's trick in a talk by Amie Wilkinson, available online\footnote{\url{https://www.youtube.com/watch?v=J672NR5dwzc}}. Once $s$-transversality is established, \cite{ACW} shows that dynamical minimality must hold.
\end{remark}

Combining this with Corollary~\ref{cor:main-dynamicallymininmal} we obtain the following.

\begin{corollary}
    \label{cor:main-irredutivel}
Let $f\in\operatorname{End}^2(\TT)$ be a partially hyperbolic and uniformly expanding map which is topologically conjugate to $f_B$. If $f$ is not special, then $f$ admits a unique $u$-Gibbs measure which is the unique absolutely continuous invariant measure.   
\end{corollary}

In particular, this gives open sets of partially hyperbolic endomorphisms on $\TT$ displaying a unique $u$-Gibbs measure. Although this is a clean and direct application, we believe that the combination of the results and, most of all, the techniques developed in \cite{ACEPWZ} with Corollary~\ref{cor:main-dynamicallymininmal} will have even more consequences, going beyond the uniformly expanding scenario. An evidence of this guess is that in the case of three dimensional partially hyperbolic diffeomorphisms, some results giving uniqueness of $u$-Gibbs states will be given in \cite{ACEPWZ}, notably in the smooth setting where \cite{Katz} can be used. 

\subsubsection{The reducible case} Considering now a matrix $A\in\operatorname{SL}_2(\Z)$ with two distinct integer eigenvalues, both larger than $1$, then the main obstruction to the application of our main result lies in the lack of knowledge about the behavior of unstable leaves. For in this case a phenomenon not present in the invertible setting appears: the existence of \emph{compact} unstable leaves. Indeed, we demonstrate in Example~\ref{ex:3} that one can build a non-special perturbation of the induced map $f_A$ which presents compact unstable leaves with non-dense orbit. This example, despite of being uniformly expanding, partially hyperbolic, {strongly transitive} and not special, displays infinitely $u$-Gibbs measures. These claims are proved in Section~\ref{subsec:provasexe3}. 

More generally, we do not know what to expect as an answer to the following problem.

\begin{problem}
\label{prob:redutivel}
Let $f\in\operatorname{End}^2(\TT)$ be a partially hyperbolic map isotopic to a linear expanding map $f_A$, where the matrix $A$ has integer eigenvalues $\sigma^u\geq\sigma^c>1$. Do we have that every unstable curve is either compact or has a dense orbit?
\end{problem}

Notice that the work of Martin Andersson and Wagner Ranter \cite{AW} shows that, if we assume further than the assumptions of Problem~\ref{prob:redutivel} that 
\[
\det Df(x)>\sigma^u,\:\:\:\textrm{for every}\:\:\:x\in\TT,
\]
then $f$ is locally eventually onto (for all $U \subseteq M$ non-empty open set there is $N \in \N$ such that $f^N(U) = M$), and in particular strongly transitive. Even under this stronger assumption, we do not know the answer to Problem~\ref{prob:redutivel}.

\subsection{An ideia of the Eskin--Mirzakhani scheme}
Let us give an overview of the argument we employ to prove Theorem~\ref{teo:main.rigidez}. Consider a pair $(f,\mu)$ satisfying the assumptions of Theorem~\ref{teo:main.rigidez}. First, one reduces the problem of demonstrating the absolute continuity of the measure to a very precise measure rigidity problem. In our setting, we do this by showing that the measure has the \emph{SRB property} on the inverse limit space (see Definition~\ref{def:srb}). Since we are not in a homogeneous setting we need to use \emph{normal form coordinates} in order to linearize the dynamics along \emph{center unstable manifolds} (see Section~\ref{sec:normal}). This permits, among other things, to give a precise meaning to the statement that the $u$-Gibbs property is equivalent to ``\textit{the measure is invariant by translations along unstable manifolds}''. Thus if we quotient the measure along unstable manifolds (see Section~\ref{sec:leafwise}) we obtain a family of measures on the one-dimensional center manifolds, which by normal forms are mapped onto $\R$. Therefore we obtain a family of measures $\hat{\nu}^c_{\tix}$ on the real line, attached to almost every point $\tix$ on the inverse limit, so that our $u$-Gibbs measure $\mu$ has the SRB property (and thus is absolutely continuous) if, and only if, the measures $\hat{\nu}^c_{\tix}$ are invariant by translations (Lemma \ref{lem:temqueserlebesgue}).  Moreover, this family inherits the \emph{starting invariance of our $u$-Gibbs measure $\mu$}: they (virtually) do not change when we move the base point $\tix$ to some $\ta \in \w^u(\tix)$. Moreover, from $x$ to $f(x)$ the measure is multiplied by $\|Df(x)|_{E^c}\|$ (see Lemma~\ref{lem:basicmoves} for a precise statement).

With proper adaptations, necessary in each context, this reduction of the measure rigidity problem to the problem of obtaining some kind of extra invariance along the direction of weak expansion is present in all previous works \cite{ALOS,EskinMirzakhani,BRH,BenoistQuintI}. In this part of the proof, our contribution, comparing with \cite{ALOS} and \cite{Katz}, is to observe that it works also when one only assumes low regularity and non-uniform expansion along the center. Recall that the maps considered in \cite{ALOS} are Anosov an thus present uniform expansion along the center. Also, one can compare with \cite{Katz} where the high smoothness is used from the start to obtain normal forms, contrarily to our work (and \cite{ALOS}) where we only need the map to be $C^2$.  

Once the reduction is completed, the factorization method of Eskin--Mirzakhani can roughly be described, following the scheme in Figure \ref{fig:EMscheme}, as follows (see Section~\ref{sec:proof-main-lemma2} where we implement the scheme in precise terms).

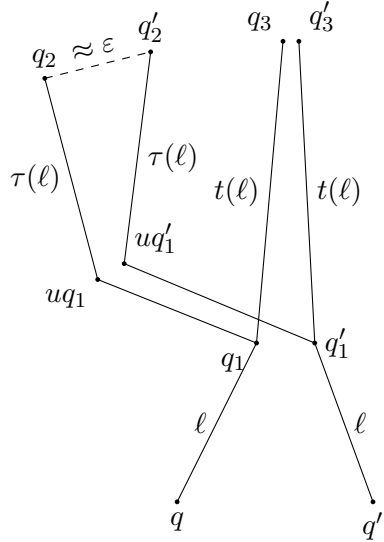
\begin{figure}[ht]

\begin{tikzpicture}[scale=.7]

\coordinate (q) at (-1.7,0); 
\coordinate (q_1) at (-0.2,3);
\coordinate (q_2) at (-4.2,8);
\coordinate (q_3) at (0.3,8.7); 
\coordinate (q') at (2,0);
\coordinate (q'_1) at (0.9,3);
\coordinate (q'_2) at (-2.2,8.5);
\coordinate (q'_3) at (0.6,8.7);
\coordinate (uq_1) at (-3.2,4.2);
\coordinate (u'q'_1) at (-2.7,4.5);

\draw (q) -- (q_1) node[pos=0.5,left]{$\ell$};
\draw (q') -- (q'_1) node[pos=0.5,right]{$\ell$};
\draw (q_1) -- (q_3) node[pos=0.5,left]{$t(\ell)$};
\draw (q'_1) -- (q'_3) node[pos=0.5,right]{$t(\ell)$};
\draw (uq_1) -- (q_2) node[pos=0.5,left]{$\tau(\ell)$};
\draw (q'_2) -- (u'q'_1) node[pos=0.5,right]{$\tau(\ell)$};
\draw (q_1) -- (uq_1);
\draw (q'_1) -- (u'q'_1);

\draw[dashed] (q_2) -- (q'_2) node[pos=0.5,above, sloped]{$\approx \varepsilon$};

\filldraw (q) circle (1pt) node[below]{$q$};
\filldraw (q') circle (1pt) node[below]{$q'$};
\filldraw (q_1) circle (1pt) node[below left]{$q_1$};
\filldraw (q'_1) circle (1pt) node[right]{$q'_1$};
\filldraw (q_2) circle (1pt) node[above]{$q_2$};
\filldraw (q'_2) circle (1pt) node[above]{$q'_2$};
\filldraw (q_3) circle (1pt) node[above left]{$q_3$};
\filldraw (q'_3) circle (1pt) node[above right]{$q'_3$};
\filldraw (uq_1) circle (1pt) node[below left]{$uq_1$};
\filldraw (u'q'_1) circle (1pt) node[above right]{$uq'_1$};
\end{tikzpicture}
\caption{\label{fig:EMscheme} Illustration of Eskin--Mirzakhani scheme.}
\end{figure}

\begin{itemize}
   \item On a ``large set'', pick points $q$ and $q'$ in the support of the measure such that $q$ and $q'$ are in the same stable leaf (which, in our case, means the fibers of the inverse limit space) with controlled distance.
    \item The ``randomness'' of the unstable direction along stable leaves ensures that $E^u(q)$ is uniformly transverse to $E^u(q')$. Thus, one has a uniform control on the center distance between $q$ and $q'$.
    \item Iterating forward by some time $\ell$ and moving to points $uq_1\in\w^u(q_1)$ and $uq_1^{\prime}\in\w^u(q^{\prime}_1)$ (thus, ``changing the future'') one gets an exponential estimate for the center distance between $uq_1$ and $uq_1^{\prime}$ of the form $e^{-c\ell}$, for some $c=c(q_1,q_1^{\prime},uq_1,uq_1^{\prime})>0$. 
    \item Iterating forward by the new future for some time $\tau(\ell)$ we find points $q_2$ and $q_2^{\prime}$ which are \emph{almost} on the same center-unstable leaf and whose center distance is approximately $\eps$.
    \item Using the \emph{basic moves} of quotient measures (see Lemma~\ref{lem:basicmoves}), one can find a time $t(\ell)$ so that the points $q_3=f^{t(\ell)}(q_1)$ and $q_3^{\prime}=f^{t(\ell)}(q_1^{\prime})$ satisfy 
    \[
    \hat{\nu}^c_{q_2}\propto B_*\hat{\nu}^c_{q_3}\:\:\:\textrm{and}\:\:\:\hat{\nu}^c_{q_2^{\prime}}\propto \hat{B}_*\hat{\nu}^c_{q_3^{\prime}},
    \]
    where $B$ and $\hat{B}$ are linear maps of the real line with uniform slope (independent of $\ell$), and the notation $\propto$ means that the measures are proportional (see Section \ref{sec:normal-basicmoves}). Indeed, notice that from $uq_1$ to $q_2$ the measure changes by a linear factor of slope $\|Df^{\tau(\ell)}(uq_1)|_{E^c}\|$, thus it suffices to find $t=t(\ell)$ so that $\|Df^t(q_1)|_{E^c}\|$ is almost of the same size than $\|Df^{\tau(\ell)}(uq_1)|_{E^c}\|$, and a similar procedure has to be done in the ``$q_1^{\prime}$ side'' of the picture.  
    \item Since $q_3$ is exponentially close to $q_3^{\prime}$, if all points on the picture belong to a compact \emph{Lusin set} where the family $\hat{\nu}^c_{\tix}$ varies continuously we deduce that 
    \[
    \hat{\nu}^c_{q_2}\propto\widetilde{B}_*\hat{\nu}^c_{q_2^{\prime}}\:\:\:\textrm{up to an error}\:\:\:\delta(\ell)\to_{\ell\to+\infty} 0.
    \]
    where $\widetilde{B}:\R\to\R$ is a linear map, with uniformly bounded slope
    \item Thus, by taking limits, we find two points $q_{\infty}$ and $q_{\infty}^{\prime}$ in the same center leaf, distant apart of about $\eps$, so that their corresponding quotient measures change by a linear factor with uniformly bounded slope. This implies that $\hat{\nu}^c_{q_{\infty}}$ has the necessary additional invariance (see Lemma~\ref{lem:main-lemma}). 
\end{itemize}

In the last point of the argument, the fact that we have $\dim E^c=1$ plays a crucial role, since we need the affine structures given by normal forms. This is one of the key difficulties for generalizing our result to higher dimensions. 

As in \cite{ALOS}, we actually implement a modification of the above sketch where the points $uq_1$ and $uq_1^{\prime}$ are replaced by small unstable curves intersecting the good Lusin sets (see Lemma~\ref{lem:matching}). Moreover, as we already pointed out, the main technical difference of our argument in comparison with \cite{ALOS} is that we work with non-uniform expansion along the center. This makes the \emph{matching argument} (see Proposition 11.8 of \cite{ALOS}) much more complicated, for we need to be extremely careful with several constants which are no longer uniform (depending only on $f$) but rather they are uniform on Lusin sets. For this reason, in Lemma~\ref{lem:matching} we work with two different Lusin sets. It is also worth to remark that our argument makes more clear the meaning of ``factorization'' in our context: notice that the \emph{stopping times} $\tau(\ell)$ and $t(\ell)$ depend on the exponential rate $c=c(q_1,q_1^{\prime},uq_1,uq_1^{\prime})$. In our case, we can prove that actually $c$ depends only on $q_1$ and thus, in order to have \emph{uniform estimations} (like the \emph{synchronization of stopping times} in Proposition~\ref{prop:sync})    we only need that the points on the left side of the picture lie in a good Lusin set $\cL_1$. This allows us to use $\cL_1$ to control the estimates and a larger set $\cL_2\supset\cL_1$ to require continuity properties.   

\subsection{Structure of the paper}
In Section~\ref{sec:ph-exp} we introduce more comprehensively the dynamical systems we work with: partially hyperbolic endomorphisms in dimension two. We also develop a number of important tools from Pesin Theory, notably Lyapunov norms and distortion estimates. In Section~\ref{sec:normal} we discuss the generalization of the results from \cite{ALOS} on normal forms and the construction of quotient measures $\hat{\nu}^c_{\tix}$ to our setting. The proof of our main result starts at Section~\ref{sec:transv}, where we give a different and simpler proof of the $0-1$ law of \cite{ALOS,BRH}, taking advantage of the symbolic structure of stable leaves. Sections~\ref{sec:proof-main-lemma} and \ref{sec:matching} are devoted to the development of the concept of $Y$-configurations from \cite{EskinMirzakhani,EskinLind} (also used in \cite{Katz,ALOS}) in our setting. The main technical point of our argument (Lemma~\ref{lem:matching}) is proved at the end of Section~\ref{sec:matching}. Finally, in Section~\ref{sec:proof-main-lemma2} we implement the Eskin--Mirzakhani scheme as in \cite{ALOS}, using the conclusion of Lemma~\ref{lem:matching}.

\subsection*{Acknowledgments} B.S. thanks Rafael Potrie who suggested the problem in October 2022 at IMPA, and also thanks Sylvain Crovisier for important remarks. We both thank Sébastien Alvarez for useful discussions. M.C. thanks Cristina Lizana for useful conversations at an earlier stage of this work. B.S. was supported by \emph{Conselho Nacional de Densenvolvimento Cinetífico e Tecnológico (CNPQ)} via the grants \emph{Bolsa PQ 313363/2021-8} and \emph{Bolsa PDE 401318/2022-2} and by \emph{Fundação de Amparao a Pesquisa do Estado do Rio de Janeiro (FAPERJ)} via the grants \emph{JCNE E-26/201.411/2021} and \emph{JPF E-26/210.344/2022}. B.S. was partially supported by \emph{Coordenação de Aperfeiçoamento de Pessoal de Nível Superior (CAPES) - Finance code 001}. M.C. was supported by \emph{Fundação de Amparao a Pesquisa do Estado do Rio de Janeiro (FAPERJ)} via the grant \emph{JPF E-26/210.344/2022} and the Australian Research Council (ARC).

\section{Partially hyperbolic endomorphisms on surfaces}
\label{sec:ph-exp}

In this section we give complete definitions and state some properties of the kind of dynamical systems we consider in this work. We present in detail the inverse limit space and how one can translate the theory of three dimensional partially hyperbolic diffeomorphisms to the non-invertible setting. We present a number of examples and study their $u$-Gibbs measures. At the end, we develop the theory of Lyapunov norms in our setting which is a convenient way of measuring the amount of expansion along the center direction. This tool plays a significant role in our argument.

\subsection{Inverse Limit}
\label{sec:inv-lim}

One of the key features of our work is to explore the connections between partially hyperbolic diffeomorphisms in dimension three, with a decomposition $E^s\oplus E^c\oplus E^u$, and partially hyperbolic endomorphisms in dimension two. In this section we describe the object which allows to properly formulate a dictionary between these classes of systems.

\subsubsection{Abstract construction}
Consider $(X, d)$ a compact metric space and a map $f: X \to X$ that is continuous and surjective. Then $f$ has an invertible extension $\tilde{f}: X_f \to X_f$, that is, there exists $\pi: X_f \to X$ continuous and surjective such that $\pi \circ \tilde{f} = f \circ \pi$. This extension is called \textit{inverse limit} or \textit{natural extension} and it is defined as follows.
\begin{itemize}
    \item $X_f = \{\Tilde{x} = (x_k) \in X^\mathbb{Z}: x_{k+1} = f(x_k) \mbox{, } \forall k \in \mathbb{Z}\}$,
    \item $(\Tilde{f}(\Tilde{x}))_k = (x_{k+1})$ $\forall k \in \mathbb{Z}$ and $\forall \Tilde{x} \in \Tilde{X}$,
    \item $\pi: X_f \to X$ is the projection on the 0th coordinate.
\end{itemize}

Given $\tix,\tilde{y}\in X_f$ we define $n(\tix,\tilde{y}) \eqdef \max \{n\in\Z\; : \; x_i=y_i,\:\textrm{for all}\:|i|<n\}$. Then, by defining
\[
d(\tix,\tilde{y})\eqdef 2^{-n(\tix,\tilde{y})},
\]
we obtain a metric in $X_f$ which makes it a compact metric space. 
\begin{remark}
    \label{lem:ultramen}
The space $(X_f,d)$ is in fact an \emph{ultrametric space}, in the sense that it holds the improved triangle inequality: $d(\tix,\tilde{y})\leq\max\{d(\tix,\tilde{z}),d(\tilde{z},\tilde{y})\}$, for every $\tix,\tilde{y},\tilde{z}\in X_f$. 
\end{remark}

Notice that the shift map $\tilde{f}$ and the projection $\pi$ are both continuous. If $f$ is invertible, then $\pi$ is a homeomorphism between $X_f$ and $X$. We consider in this work only the case in which $f$ is not invertible.

One can prove that the metric $d$ generates the same topology as the \textit{open cylinders} given by
$$[A_{-k}, \cdots A_0] \eqdef \{\tilde{x} \in X_f : x_i \in A_i \mbox{ for } i \in \{-k,...0\} \},$$
where $A_i \subseteq X$ are open sets. The same cylinders generate the Borel $\sigma$-algebra for $(X_f, \tilde{d})$. 

All measures considered in this work are over the Borel $\sigma$-algebra for the given space and, given a space $X$, we denote by $\cM(X)$ the probability measures on $X$, endowed with the weak$^*$ topology. Given a discrete time dynamical system $f:X\to X$ we denote by $\cM_f(X)$ the subset of $f$-invariant probability measures. Finally, we denote by $\cM^{\operatorname{erg}}_f(X)$ the ergodic elements in $\cM_f(X)$.  

It is easy to check that if $\tilde{\mu}$ is a $\tilde{f}$-invariant measure on $X_f$, then $\pi_*\tilde{\mu}$ is a $f$-invariant measure on $X$. Since $\pi$ is surjective, for all such $\mu$, there is a $\tilde{\mu}$ that projects on it. Actually, $\pi_*: \mathcal{M}_{\tilde{f}}(X_f) \to \mathcal{M}_f(X)$ is a bijection, as proved for instance in \cite[Proposition I.3.1]{QXZ2009}.

\begin{proposition}
    \label{prop:qxz-mu}
    Let $(X,d)$ be a compact metric space and $f: X \to X$ continuous. For any $f$-invariant probability measure $\mu$ on $X$, there is a unique $\tilde{f}$-invariant probability measure $\tilde{\mu}$ on $X_f$ such that $\pi_* \tilde{\mu} = \mu$.
\end{proposition}

The uniqueness on the above result is given by the fact that $\tilde{\mu}$ is defined by the values it attributes to cylinders. 

The topological structure of $(X_f, \tilde{d})$ is quite interesting if $f$ is not invertible. If $X$ is connected and $f$ is a self-covering map with positive degree, then, for each $x \in X$, $\pi^{-1}(x)$ is compact, totally disconnected and has no isolated points, thus it is a Cantor set.

\begin{remark}
Additionally, if $X$ is a manifold (it suffices that $X$ is locally simply connected), then $(X_f, X, \pi, \Sigma)$ is a fiber bundle, where $\Sigma$ is a Cantor set \cite[Theorem 6.5.1]{AH1994}. Therefore, even if $X$ is a manifold and $f$ is differentiable, $X_f$ does not have a manifold structure. Nevertheless, we can pullback the differential structure from $X$ to $X_f$ and define $D\tilde{f}_{\tilde{x}}: T_{\tilde{x}}X_f \to T_{\tilde{f}(\tilde{x})}X_f$ as $D\tilde{f}_{\tilde{x}} = Df_x$. This derivative is uniformly bounded as a function of $\tilde{x}$. Using this tangent structure, we can see that each path connected component of $X_f$ is diffeomorphic to the universal covering of $X$ and that $\tilde{f}$ is invertible and as regular as $f$ restricted to a given path connected component. Another topological feature of $X_f$ is given by the following result.
\end{remark}

\begin{notation}
	Along this text, we use the notation $$\Sigma(\tilde{x}) \eqdef \pi^{-1}(\pi(\tilde{x}))$$ to refer to the fiber of $\tilde{x}$ on $(X_f, X, \pi, \Sigma)$.
\end{notation}

\begin{proposition}
    \label{prop:supp-mu}
    For any $f$-invariant probability measure $\mu$ on $X$, we have $\operatorname{supp}(\tilde{\mu})=X_f$ if and only if $\operatorname{supp}(\mu)=X$.
\end{proposition}

\begin{proof}
    If $\operatorname{supp}(\tilde{\mu})=X_f$, then $\operatorname{supp}(\mu)=X$ is a direct consequence of $\mu = \pi_* \tilde{\mu}$.

    If $\operatorname{supp}(\mu)=X$ and we suppose that $\operatorname{supp}(\tilde{\mu}) \neq X_f$, then there is an open set $A \subseteq X_f$ such that $\tilde{\mu}(A) = 0$. Since $\tilde{\mu}$ is $\tilde{f}$-invariant, this implies that $\tilde{\mu}(\tilde{f}^k(A)) = 0$ for all $k \in \mathbb{Z}$. But the topological structure of the inverse limit space implies that there is $k < 0$ such that $\tilde{f}^k(A)$ contains $\pi^{-1}(U)$ for some open $U \in X$, a contradiction.
\end{proof}

\subsubsection{Topological dynamics on $X_f$}

Recall that a continuous map $f: X \to X$ is said to be {strongly transitive} if for all $U \subseteq X$ non-empty open set there is $N \in \N$ such that {$\cup_{n=0}^N f^n(U) = X$}. Every uniformly expanding map on a compact manifold is {strongly transitive}, thus it is transitive. The following simple result and its corollary are fundamental in allowing us to have ``quantitative'' information about points which have ``many'' different unstable directions (Lemma \ref{l.contracaomaravilhosa}).

\begin{lemma}
  \label{lem:preorbitasdensas}
Let $f:X\to X$ be a {strongly transitive} transformation. Then, for each $x\in X$ the set $\{f^{-n}(x);n\in\N\}$ is dense in $X$.
\end{lemma}
\begin{proof}
Let $U\subset X$ be any open set. Then, there exists $N>0$ such that $\cup_{k=0}^N f^k(U)=X$. In particular, there exist some $y\in U$ and $k \leq N$ such that $f^k(y)=x$, which means that $f^{-k}(x)\cap U\neq\emptyset$. 
\end{proof}

\begin{corollary}
    \label{cor:perorbitascdensas}
Let $f:X\to X$ be a {strongly transitive} transformation. Then, for each $\tix\in X_f$ the set $\{\tif^{-n}(\Sigma(\tix));n\in\N\}$ is dense in $X_f$.    
\end{corollary}

\subsection{Partial hyperbolicity on the inverse limit}
Consider $\operatorname{End}^r(\TT)$ the set of $C^r$ local diffeomorphisms on $\TT$. Let $f\in\operatorname{End}^1(\TT)$ be a partially hyperbolic endomorphism. We denote by $\TTf$ the associated inverse limit space as constructed above. The cone field condition of partial hyperbolicity for $f$ is equivalent to the following \cite{QXZ2009,HirschPughShub}.

\begin{lemma}
    \label{lem:defideph}
There exists a continuous non-trivial splitting  $T_{\tix}\TTf=E^c(\tix)\oplus E^u(\tix)$ on the inverse limit space, for some choice of background Riemannian metric on $\TT$, such that the continuous functions
\[
\lambda^{*}_{\tix}\eqdef\|D\tilde{f}(\tix)|_{E^{*}}\|,\:\:\:\textrm{for}\:\: * \in \{c,u\},
\]
satisfy the following inequalities:
\begin{enumerate}
\item$\lambda^c_{\tix}<\lambda^u_{\tix}$, for every $\tix\in\TTf$;
\item  $\lambda^u_{\tix}>1$, for  every $\tix\in\TTf$.
\end{enumerate}
\end{lemma}

\begin{notation}
    To simplify calculations involving derivatives along trajectories, we use the following notation
    \[
    \lambda^{*}_{\tix}(n)\eqdef\|D\tif^n(\tix)|_{E^{*}}\|,
    \]
    and 
    \[
    d^{n}_{\tix}\eqdef\frac{\lambda^u_{\tix_{-n}}(n)}{\lambda^c_{\tix_{-n}}(n)},
    \]
    for every $n\in\Z$
\end{notation}

Moreover, as a consequence of \cite{CostaMicena2022}, we have that the center bundle $E^c$ descends to an $f$-invariant bundle on $\TT$.

\begin{lemma}
    \label{lem:centralcte}
Let $f\in\operatorname{End}^1(\TT)$ be partially hyperbolic. Then, for every $\tilde{y}\in\Sigma(\tix)$ one has $E^c(\tilde{y})=E^c(\tix)$.
\end{lemma}

\begin{remark}
	\label{rem:uniq-c}
	In fact, if $f: X \to X$ is a local diffeomorphism with a $Df$-invariant dominated splitting $T_{\tilde{x}}X = E(\tilde{x}) \oplus F(\tilde{x})$, with $E \prec F$, then $E(x)$ does not depend on $\tilde{x} \in \pi^{-1}(x)$. A proof of this fact for partially hyperbolic endomorphisms is given on \cite[Lemma 2.5]{CostaMicena2022}.
\end{remark}

The unstable bundle $E^u$, on the other hand, in general is not constant along fibers of the inverse limit space. 

\begin{definition}
 We say that a partially hyperbolic $f\in\operatorname{End}^1(\TT)$ is \emph{special} if, for every $\tix\in\TTf$ and every $\tilde{y}\in\Sigma(\tix)$ one has $E^u(\tix)=E^u(\tilde{y})$.   
\end{definition}

\subsubsection{Some examples} To get a more concrete view of the variety of behaviors one can encounter among partially hyperbolic elements $f\in\operatorname{End}^1(\TT)$, we discuss below some simple examples.

\begin{example}
    \label{ex:1}
    
    The simplest example is the reducible one given by the product of two expanding maps $3x \pmod{1}$ and $2x \pmod{1}$ in the circle $\mathbb{S}^1 = \R/\Z$, the first expansion dominating the second one. Equivalently, consider $f_A: \mathbb{T}^2 \to \mathbb{T}^2$ as the transformation induced by the matrix $A = \left(\begin{matrix} 
3 & 0\\
0 & 2
\end{matrix}\right)$.
This example is special, since it is linear, and each unstable/center leaf is a circle. It is worth to remark that it presents a behavior that \emph{is not present} on invertible partially hyperbolic Anosov maps on the 3-torus, for in this case the integer matrix $A$ should have $\vert \det A \vert = 1$ and the system is not reducible.
\end{example}

\begin{example}
    \label{ex:2}
Consider $f_B: \mathbb{T}^2 \to \mathbb{T}^2$ as the function induced by the matrix $B = \left(\begin{matrix} 
3 & 1\\
1 & 2
\end{matrix}\right)$, with eigenvalues $\dfrac{5+ \sqrt{5}}{2} > \dfrac{5- \sqrt{5}}{2} > 1$. We have that $f_B$ is special and each unstable/center leaf is a dense line tangent to the corresponding eigenvector of $B$ on the tangent space.
\end{example}

Both simple examples presented above are special. However, most partially hyperbolic endomorphisms are not special. More precisely, the set of non-special partially hyperbolic endomorphisms on a manifold $X$ is $C^1$ open and dense \cite[Theorem B]{CostaMicena2022}. We remark that their result requires additional assumptions because they are considering systems with three non-trivial directions in the splitting and they are proving that both the center and the unstable directions are not unique for a given point. Nonetheless, the result holds in particular in dimension two as follows.

\begin{theorem}[\cite{CostaMicena2022}]
    Consider $PHE(\TT)$ as the set of all partially hyperbolic endomorphisms $f: \TT \to \TT$ with $Df$-invariant splitting $T_{\tilde{x}}M = E^c(x) \oplus E^u(\tilde{x})$. Then there is a $C^1$ open and dense subset $\mathcal{U} \subseteq PHE(\TT)$ such that every map in $U$ is non-special.
\end{theorem}

The proof of denseness is achieved by, starting with a special partially hyperbolic endomorphism, perturbing it to generate new directions, as we do for completeness with the following example.

\begin{example}
    \label{ex:3}

Recall that $f_A$ from Example \ref{ex:1} is a special partially hyperbolic uniformly expanding endomorphism with the strong unstable direction $E^u_A$ being horizontal at every point, and the center direction $E^c_A$ being vertical at every point. 

 Consider $p = (0, 0)$ the unique fixed point of $f_A$ and $q=(2/3,1/2) \in f_A^{-1}(\{p\})$. Since $f_A$ is a self covering map, there is $\tau > 0$ such that, if $f_A(y) = f_A(z)$ and $y \neq z$, then $d(y,z) > \tau$. Take $\delta \in (0, \tau/2)$ and consider $a \in \mathbb{R}$ such that the square $R = [-a,a]^2 \in \mathbb{R}^2$, when projected under $\exp_q: T_q\TT \to \TT$, is contained in $B_{\delta}(q)$. Making an abuse of notation by identifying $R$ and $\exp_q(R)$ we define
 $$\varphi(x,y) = \begin{cases}
 \left(x, y + \varepsilon \; a \; \psi_1\left( \dfrac{x}{a} \right) \psi_2\left( \dfrac{y}{a} \right) \right) & \mbox{ if } (x,y) \in R\\
 \operatorname{Id}    & \mbox{ otherwise, } 
 \end{cases}$$
 where $\varepsilon > 0$ is as small as we want and $\psi_1: \mathbb{R} \to [-2,2]$ and $\psi_2: \mathbb{R} \to [0,1]$ are $C^\infty$ functions satisfying
 \begin{itemize}
     \item $\operatorname{supp} \psi_i \subseteq [-1,1]$;
     \item $\psi_1(1/2) = 1$ and $\psi_1(-1/2) = -1$;
     \item $\psi_1(-1) = \psi_1(0) = \psi_1(1) = 0$;
     \item $\restr{\psi_1}{[-1/2,1/2]}$ is linear;
     \item $\restr{\psi_2}{[-1/2,1/2]} \equiv 1$.
 \end{itemize}

 We have that
 $$D\varphi_{(x,y)} = \begin{pmatrix}
1 & 0 \\
\varepsilon \; \psi_1'\left( \dfrac{x}{a} \right) \psi_2\left( \dfrac{y}{a} \right)  & 1 + \varepsilon \; \psi_1\left( \dfrac{x}{a} \right) \psi_2'\left( \dfrac{y}{a} \right)
\end{pmatrix},$$
in $R$, which is easily checked to be invertible for all $(x,y)$, provided that $\varepsilon$ is sufficiently small. 

Additionally, $\varphi$ is $C^\infty$ and satisfies $\varphi(q) = q$, $\restr{\varphi}{\overline{R^\complement}} = \operatorname{\operatorname{Id}}$, $\varphi(R) = R$ and
 $$D\varphi_{q} \cdot E^u_A(q) = \begin{pmatrix}
1 & 0 \\
2 \varepsilon & 1 
\end{pmatrix}\begin{pmatrix}
\alpha \\
0
\end{pmatrix} = \begin{pmatrix}
\alpha \\
2 \varepsilon \alpha
\end{pmatrix} \neq E^u_A(p),$$
meaning that $\varphi$ does not preserve the unstable directions of $f_A$.
 
 Let $f\eqdef f_A\circ\varphi$. Since $f$ is $C^1$ close to $f_A$, we have that $f$ is a partially hyperbolic uniformly expanding endomorphism. Additionally, since $(1,0)$ is an eigenvector for $D\varphi_{(x,y)}$, $f$ has the same central manifolds as $f_A$. The same argument as in \cite{CostaMicena2022} gives us that $f$ is non-special as follows.

 By the definition of $\delta$, we have that any point $x\in\TT$ has a past orbit $\tix\in\Sigma(x)$ such that $x_{-k}\notin R$ for every $k\in\N$. Thus, along this orbit we have that $f=f_A$, which implies that $E^u(\tix)=E^u_A$.	In particular, there exists $\tilde{q}\in\Sigma(q)$ such that $E^u(\tilde{q})$ is horizontal. Now, by invariance this implies that for $\tilde{p}= \tilde{f}(\tilde{q}) = (...q_{-1}q_0,p,p,p,...)$ we have that
 \[
 E^u(\tilde{p})=Df(q)E^u(\tilde{q})\neq E^u_A.
 \]
 However, for $\hat{p}=(...p,p,p,...)$ one has that $f=f_A$ all along this fixed orbit and thus $E^u(\hat{p})=E^u_A$. This proves that $p$ has more than one unstable direction, and thus $f$ is not u-special.

\end{example}

Similarly, we can do a perturbation of Example \ref{ex:2}.

\begin{example}
    \label{ex:4}
    We can repeat the construction from the previous example with $q = \left( \dfrac{2}{5}, \dfrac{4}{5} \right)$, and by using $E^c_B(q)$ and $E^u_B(q)$ as coordinates on $T_q\TT$ to define $\varphi$.
    Then, for all sufficiently small $\varepsilon > 0$, there is a non-special partially hyperbolic uniformly expanding endomorphism $g$ that is $C^1$-close to $f_B$ from Example \ref{ex:2} and having the same center leaves.
\end{example}

\subsubsection{Invariant foliations}
Let $f\in\operatorname{End}^r(\TT)$ be partially hyperbolic and let $\tif:\TTf\to\TTf$ denote its lift to the inverse limit space. By viewing the direction of the fibers as a stable direction and the ``manifold direction'' as a center-unstable direction, we can see the dynamics on the inverse limit space as partially hyperbolic as follows.

\begin{definition}
    \label{defi:variedadecentroinstavel}
Given $\tix\in\TTf$ we define its center-unstable manifold $\w^{cu}(\tix)$ as being the path-connected component of $\TTf$ containing $\tix$.    
\end{definition}

As we mentioned above, one can show that each center-unstable manifold $\w^{cu}(\tix)$ is homeomorphic to the universal cover $\R^2$ of $\TT$. Notice that $\tif$ acts permuting the components. Thus, $\{\w^{cu}(\tix)\}_{\tix\in\TTf}$ is an $\tif$-invariant lamination of the space $\TTf$. This lamination is analogous to the center unstable foliation of a three dimensional partially hyperbolic diffeomorphism. Indeed, notice that the fibers $\Sigma(\tix)$ are contracted by the action of $\tif$.

\begin{remark}
    If $\tilde{y}\in\Sigma(\tix)$ then the definition of the metric in $\TTf$ immediately implies $d(\tif^n(\tix),\tif^n(\tilde{y}))\leq 2^{-n}$.
\end{remark}

Thus the fibers $\Sigma$ of the space $\TTf$ play the same role as local stable manifolds in three dimensional partially hyperbolic diffeomorphisms, and they are ``transverse'' to the center-unstable manifolds $\w^{cu}$. Clearly, there is a notorious difference for here the fibers are totally disconnected compact spaces, and not embedded curves as in the case of diffeomorphisms. 

Nonetheless, the graph transform method applies to the dynamics of $\tif$ on $\TTf$ and one can show the following (see \cite{HirschPughShub,QXZ2009}).

\begin{proposition}
    \label{prop:varieadeinstavel}
Given $f\in\operatorname{End}^r(\TT)$ partially hyperbolic, there exists a one-dimensional unstable lamination $\w^u$ of the space $\TTf$ by $\tif$-invariant immersed (on $\w^{cu}(\tix)$) $C^r$ curves, varying continuously, which are everywhere tangent to $E^u$. 
\end{proposition}

The existence of a center foliation is a more delicate issue. In our setting, we have the following consequence of \cite[Theorem B]{HallHammerlindl22}.

\begin{proposition}
    \label{prop:notannulus}
Let $f\in\operatorname{End}^r(\TT)$ be a strongly transitive partially hyperbolic map. Then, there exists a one-dimensional (center) lamination $\w^c$ of the space $\TTf$ by $C^r$ immersed curves, varying continuously, which are everywhere tangent to $E^c$. Moreover, this foliation descends to an $f$-invariant foliation of $\TT$.
\end{proposition}

In this work, we use consistently the following notation: $\w^{*}(\tix)$ , $* \in \{u, c, cu \}$, is the leaf through the point $\tix\in\TTf$, while $\cW^{*}(\tix)\eqdef\pi(\w^u(\tix))\subset\TT$.

\subsubsection{The unstable foliation of Example~\ref{ex:3}}
\label{subsec:provasexe3}
The unstable foliation of Example~\ref{ex:1} descends to a foliation by (horizontal) compact circles, while the center foliation is the foliation by vertical circles. In Example~\ref{ex:3}, as the map is not special the unstable foliation on $\TTf$ does not descends to a foliation of $\TT$, while the center foliation remains the same as that of $f_A$.

Moreover, we can choose the parameter $a>0$ small enough so that there exists some (meager) compact set $K\subset\mathbb{S}^1$, which is forward invariant under $x\mapsto 2x\mod 1$, and $(\mathbb{S}^1\times K)\cap R=\emptyset$. Thus, after perturbation the set $\Gamma=\mathbb{S}^1\times K$ remains invariant under the perturbed map $f$, as $f=f_A$ restricted to $\Gamma$. In particular, every horizontal unstable curve through a point $x\in\Gamma$ has non-dense forward orbit. Moreover, we have that there are points $\tix\in \TTf$ such that the set $\cup_{i\in\Z}\cW^u(x_i)$ is not dense in $\TT$. Indeed, for all $x \in \Gamma$, there is $\tix\in\Sigma(x)$ such that $x_{-k}\in \Gamma$ for every $k\in\N$, which implies that $E^u_f(\tix) = E^u_A$. And the uniform expansion of $f$ allows us to make this choice of past orbit coherently along the compact curves in $\Gamma$, in the sense that $\cW^u(x_i) \subseteq \Gamma$ for all $i \leq 0$.

The behavior of the foliation in this example is still not well understood. We believe that every unstable curve which is not horizontal at some place should have a dense orbit, but this claim has not yet been confirmed. 

\subsection{$u$-Gibbs measures for endomorphisms}
\label{sec:inv-lim-ph}

We are now able to define u-Gibbs measures on the non-invertible setting. Informally speaking, a u-Gibbs measure is a measure that is ``well distributed'' with respect to an unstable foliation. Since, in general, there is no unstable foliation on the ambient manifold, the definition passes through the unstable leaves of the inverse limit.

To make this concept of a ``well distributed'' measure with respect to a foliation precise, we first introduce how we decompose the measure with respect to a ``nice partition'', with the following definition and theorem stated specifically for our context.

Let $f\in\operatorname{End}^r(\TT)$ be a partially hyperbolic endomorphism. 

\begin{definition}
    \label{def:measurable-part}
    A partition $\mathcal{P}$ of $\TTf$ is a \emph{measurable} (or \emph{countably generated}) partition with respect to $\tilde{\mu}$ if there is $M_0 \subseteq \TTf$ with $\tilde{\mu}(M_0) = 1$ and a countable family $\{A_i\}_{i \in \N}$ of measurable sets such that, given $P \in \mathcal{P}$, there is $\{P_i\}_{i \in \N}$ with $P_i \in \{A_i, A_i^\complement\}$ such that $P = \bigcap\limits_{i \in \N} P_i$ restricted to $M_0$.
\end{definition}

In other words, every element of the partition can be generated with the intersections of a countable family of measurable sets or their complements. For such partitions the following result holds.

\begin{theorem}[Rokhlin disintegration]
    If $\mathcal{P}$ is a measurable partition for $\TTf$ and $\tilde{\mu}$, then $\tilde{\mu}$ has a disintegration on a family of conditional measures $\{\tilde{\mu}_P\}_{P \in \mathcal{P}}$, that is, for all $\phi: \TTf \to \R$ continuous
    \begin{enumerate}
        \item $P \mapsto \int \phi d\tilde{\mu}_P$ is measurable;
        \item $\tilde{\mu}_P(P) = 1$ for $\hat{\mu}$-almost every $P \in \mathcal{P}$;
        \item $\int \phi d\tilde{\mu} = \int_\mathcal{P} \int_P \phi  d\tilde{\mu}_P d\hat{\mu} $.
    \end{enumerate}    Here, $\hat{\mu}$ denotes the \emph{transverse} measure given by $\hat{\mu}(\mathcal{Q}) = (p_{\mathcal{P}})_*\mu(\mathcal{Q}) = \mu(p_{\mathcal{P}}^{-1}(\mathcal{Q}))$, where $p_{\mathcal{P}}: \TTf \to \mathcal{P}$ is the projection assigning for each point $x \in\TTf$ the element of $\mathcal{P}$ containing $x$.
\end{theorem}

\begin{remark}
    \label{rem:desintegranafibra}
The partition $\{\Sigma(\tix)\}_{\tix\in\TTf}$ into fibers, being a decomposition of the space into a disjoint union of compact sets, is easily seen to be an example of a measurable partition. Given a probability measure $\mu\in\cM(\TT)$ we denote by $\{\mu_x\}_{x\in\TT}$ the conditional measures of its lift $\tm$  with respect to this partition.
\end{remark}

A special case of measurable partition for which the above theorem applies is given by $u$-subordinate partitions.

\begin{definition}
	\label{def:u-subord}
	A measurable partition $\mathcal{P}$ of $\TTf$ is \emph{u-subordinate} with respect to $\mu$ if, for $\tilde{\mu}$-almost every $\tilde{x} \in \TTf$, the atom $\mathcal{P}(\tilde{x})$ satisfies
	\begin{enumerate}
		\item $\restr{\pi}{\mathcal{P}(\tilde{x})}$ is a bijection onto its image;
		\item there is a $C^1$ embedded submanifold $U_{\tilde{x}} \subseteq \TT$ with dimension equal to $\dim E^u$, such that $U_{\tilde{x}} \subseteq \cW^u(\tilde{x})$, $\pi(\mathcal{P}(\tilde{x})) \subseteq U_{\tilde{x}}$ and $\pi(\mathcal{P}(\tilde{x}))$ contains a open neighborhood of $\pi(\tilde{x})$ in $U_{\tilde{x}}$.
	\end{enumerate}
\end{definition}

The proof of the existence of u-subordinate partitions for non-invertible maps is given in \cite[\S IX.2.2]{QXZ2009}.

\begin{definition}
	\label{def:u-gibbs}
	An $f$-invariant measure $\mu$ is said to be \emph{$u$-Gibbs} if, for every u-subordinate partition $\mathcal{P}$ for $\mu$, we have 
	$$\pi_*\tilde{\mu}^{\mathcal{P}}_{\tilde{x}} \ll \Leb^u_{\tilde{x}}$$
	for $\tilde{\mu}$-almost every $\tilde{x} \in\TTf$, where $\{\tilde{\mu}^{\mathcal{P}}_{\tilde{x}}\}$ is the disintegration of $\tilde{\mu}$ with respect to $\mathcal{P}$ and $\Leb^u_{\tilde{x}}$ is the Lebesgue measure on $U_{\tilde{x}}$.
\end{definition}

\begin{remark}
    \label{rem:supp-u-sat}
Given $\mu$ a u-Gibbs measure, if $\tilde{\mu}$ is the unique lift to the inverse limit space, then it follows directly from the definition that $\tix \in \operatorname{supp} \tilde{\mu}$ implies $\w^u(\tilde{x}) \subseteq \operatorname{supp} \tilde{\mu}$. In other words, the support of $\tilde{\mu}$ is $u$-saturated.
\end{remark}

\subsubsection{u-Gibbs measures for the examples}
\label{sec:u-gibbs-ex}

Example \ref{ex:1} has many u-Gibbs measures: for any $\nu$ a ($2x \pmod 1$)-invariant measure on $\mathbb{S}^1$, $\operatorname{Leb}_{\mathbb{S}^1} \times \nu$ is a u-Gibbs measure. Since the doubling map $2x \pmod 1$ is semi-conjugated to a full shift in two symbols, we know this system has infinitely many u-Gibbs measures. This is a particular case because the system is reducible.

In contrast, Example \ref{ex:2} has only $\operatorname{Leb}_{\mathbb{T}^2}$ as a u-Gibbs measure. A brief sketch of this argument goes as follows: up to a constant, the density of $\tm$ along some unstable leaf $\w^u(\tix)$ is given by
\[
\lim_{n\to-\infty}\frac{\lambda^u_{\tilde{y}}(n)}{\lambda^u_{\tilde{x}}(n)}.
\]
Since the map is linear, this shows that the density is constant, and thus the conditional measure is invariant by translations along leaves. The invariance of conditional measures by translations along leaves implies invariance of the measure $\mu$ under the unity flow along leaves. However, the unity flow along leaves is a Diophantine flow, which is uniquely ergodic. One deduces that $\mu=\Leb_{\TT}$.

Both Examples \ref{ex:1} and \ref{ex:2} are special, and one might wonder how the situation changes after perturbation. Example~\ref{ex:3} gives us a non-special endomorphism, close to $f_A$, which still have infinitely many $u$-Gibbs measures. Indeed, as we discussed above in Section~\ref{subsec:provasexe3}, the example can be build with a forward invariant set $\Gamma=\mathbb{S}^1\times K$ which is made of horizontal compact unstable leaves. Moreover, $K$ can be chosen so that it contains infinitely many $x\mapsto 2x\mod 1$ periodic orbits. Given $x_1,...,x_k$ such a periodic orbit, the measure
\[
\frac{1}{k}\sum_{i=1}^k\delta_{x_i}\times\Leb_{\mathbb{S}^1}
\]
is a $u$-Gibbs measure for $f$. 

On the other hand, Example~\ref{ex:4} is also non-special, but since we can apply Corollary~\ref{cor:main-dynamicallymininmal}, we deduce that the map has only one $u$-Gibbs measure which is the unique absolutely continuous invariant measure. 

\subsubsection{Uniform density estimate}
For a Lebesgue measurable set $A\subset\widetilde{\cW}^u(\tix)$ we denote, for simplicity, $|A|\eqdef\Leb^u_{\tix}(A)$ the Lebesgue measure of $A$. Given a $u$-Gibbs measure, the conditional measures are ``comparable'' with respect to Lebesgue, but only up to a measurable function. To be precise, the density $\dfrac{\tilde{\mu}_{\tilde{x}}}{\Leb^u_{\tix}}$ is a measurable function. By restricting to compact sets of large measure, we get a uniform comparison, as made clear by the next result. 

\begin{lemma}
\label{lem:densidadeuniforme}
Let $f\in\operatorname{End}^r(\TT)$ be partially hyperbolic and $\mu$ be an ergodic $u$-Gibbs measure for $f$. Consider $\{\tm^u_{\tix}\}$ a family of conditional measures with respect to a $u$-subordinate partition. Then, given $\delta>0$, there exists a compact Lusin set $\cL\subset\TTf$, with $\tm(\cL)>1-\delta$, for which there exists a constant $C=C(\cL)$ such that for every subset $A\subset\w^u_1(\tix)$, for $\tix\in\cL$, one has $C^{-1}|A|\leq\tm^u_{\tix}(A)\leq C|A|$.
\end{lemma}
\begin{proof}
From the definition of $u$-Gibbs measure, it follows that for $\tm$-almost every $\tix\in\TTf$ there exists a density function $\rho^u_{\tix}$ such that 
\[
\tm^u_{\tix}(A)=\int_A\rho^u_{\tix}(t)d\Leb^u_{\tix}(t).
\]
Moreover, up to a positive constant depending measurably on $\tix$, we have that the density $\rho^u_{\tix}(t)$, $t=\pi(\tilde{y})$, is given by
\[
\rho^u_{\tix}(t) = \lim_{n\to-\infty}\frac{\lambda^u_{\tilde{x}}(n)}{\lambda^u_{\tilde{y}}(n)}.
\]
By Lusin's theorem, given $\delta>0$, there exists a compact set $\cL$ with measure $\tm(\cL)>1-\delta$ such that the assignment $\tix\in\cL\mapsto\rho^u_{\tix}$ is continuous. In particular, there exists a constant $C$ such that $ C\geq\rho^u_{\tix}(t)\geq C^{-1}$, for every $\tix\in\cL$ and every $t=\pi(\tilde{y})$, where $\tilde{y}$ belongs to the atom containing $\tix$. 
\end{proof}

\begin{notation}
Since uniform estimates such as the one above are abundant in this paper, we use repeatedly the following notation: given $a,b,C>0$ we say that $a\asymp_Cb$ provided that $C^{-1}b\leq a\leq Cb$.
\end{notation}

\subsection{Pesin Theory in the partially hyperbolic setting}

\label{sec:pesin}
In this section we consider $f:\TT\to\TT$ a partially hyperbolic endomorphism and $\mu\in\cM_f^{\operatorname{erg}}(\TT)$ an ergodic invariant probability measure. 
\subsubsection{Center Lyapunov exponent}
The (pointwise) center Lyapunov exponent of $f$ is defined for all points in the $f$-invariant regular set $\Gamma$, where $\Gamma$ has full $\mu$-measure for any $f$-invariant measure $\mu$. It is given by
$$L^c(x) = \lim_{n \to \infty} \dfrac{1}{n} \log \Vert Df^n(x) v \Vert, \mbox{ where } v \in E^c(x).$$

The integrated center Lyapunov exponent with respect to $\mu$ is given by
\begin{align*}
    L^c(\mu) &= \int L^c(x) d\mu = \int \lim_{n \to \infty} \dfrac{1}{n} \log \Vert Df(x) v \Vert d\mu, \mbox{ where } v \in E^c(x)\\
    &= \int \log \vert \restr{Df}{E^c} \vert d\mu,
\end{align*}
using the fact that $E^c$ is one dimensional and Birkhoff ergodic theorem. Since $L^c(x)$ is $f$-invariant and since we work only with ergodic measures, then $L^c(x) =  L^c(\mu)$ for $\mu$-almost every point and \textit{non-uniform expansion} is given by $L^c(x) > 0$ for $\mu$-almost every point.

The differential structure of $f$ pulled back to $\TTf$ allows us to have an Oseledec's theorem for the inverse limit space and the Lyapunov exponents are the same as in the original manifold and they do not depend on the past orbit, that is, they are constant along $\Sigma(\tix)$, see \cite[Proposition I.3.5]{QXZ2009}. In Section \ref{sec:pesin} we explore Pesin theory further.

From now on, we consider a fixed ergodic $f$-invariant measure $\mu$. We also assume that the center Lyapunov exponent of $\mu$ satisfies $L^c(\mu)>0$. In this setting, we fix $0<\lambda<L^c(\mu)<\hat{\lambda}$. Then, there exists a measurable function $C:\TTf\to[1,+\infty)$ such that  
\[
C(\tix)^{-1}e^{\lambda n}\leq\lambda^c_{\tix}(n)\leq C(\tix)e^{\hat{\lambda} n},
\]
for every $n>0$ and 
\[
C(\tix)^{-1}e^{\lambda n}\geq\lambda^c_{\tix}(n)\geq C(\tix)e^{\hat{\lambda} n},
\]
for every $n<0$. Moreover, the function $C$ has slow exponential growth: 
\[
C(\tilde{f}^n(\tix))\leq e^{\varepsilon n} C(\tix),
\]
where $\varepsilon=(\hat{\lambda}-\lambda)/2$. 

\subsubsection{Pesin center manifolds}

We now briefly recall the properties of the Pesin center manifolds in our setting. There exists a pair of measurable functions $r:\TTf\to (0,1)$ and $\gamma:\TTf\to [1,+\infty)$ such that whenever $\tilde{y},\tilde{z}\in \w^c_{r(\tix)}(\tix)$ then 
\begin{equation}
\label{e.pesincentermanifold}
d(\tif^{-n}(\title{y}),\tif^{-n}(\tilde{z})\leq\gamma(\tix)e^{-n\lambda}d(\tilde{y},\tilde{z})
\end{equation}

\begin{notation}
    We use the notation $\w^c$ for the global center manifold given by the partially hyperbolic splitting and $\tW^c_{\loc}$ for the Pesin manifolds where we observe expansion, which in this case are local manifolds contained on $\w^c$ for almost every point, that is: $\tW^c_{\operatorname{loc}}(\tix)=\w^c_{r(\tix)}(\tix)$.
\end{notation}

The \emph{global Pesin manifolds}, which we denote by $\tW^c(\tix)$, are defined by the classical formula
\[
\tW^c(\tix)\eqdef\bigcup_{n=0}^{+\infty}\tif^n\left(\tW^c_{\loc}(\tix_{-n})\right).
\]
Thus $\tW^c(\tix)\subseteq\w^c(\tix)\subset\w^{cu}(\tix)$ is an immersed line. As, a priory, we assume no uniform expansion along $E^c$, the inclusion $\tW^c(\tix)\subseteq\w^c(\tix)$ could be strict. We define the \emph{global Pesin center unstable manifolds} by 
\[
\tW^{cu}(\tix)=\bigcup_{\tilde{y}\in \tW^c(\tix)}\w^u(\tix).
\]
As in \cite{ALOS} Lemma 2.4, we have the following

\begin{lemma}
    \label{lem:coerencia}
For every $\tilde{y}\in \tW^c(\tix)$ it holds that $\tW^c(\tix)\cap\w^u(\tilde{y})=\{\tilde{y}\}$.  
\end{lemma}

\subsubsection{The SRB property}

Consider now a measurable partition $\cP^{cu}$ subordinate to the Pesin center unstable foliation $\tW^{cu}(\tix)$. This means a measurable partition whose atoms are entirely contained in a Pesin center-unstable manifold and project bijectively onto measurable subsets of $\TT$ which contain a ball. More precisely

\begin{definition}
	\label{def:subord}
	A measurable partition $\mathcal{P}$ of $\TTf$ is \emph{cu-subordinate} with respect to $\mu$ if, for $\tilde{\mu}$-almost every $\tilde{x} \in \TTf$, the atom $\mathcal{P}(\tilde{x})$ satisfies
	\begin{enumerate}
		\item $\restr{\pi}{\mathcal{P}(\tilde{x})}$ is a bijection onto its image;
		\item there is an open set $U_{\tilde{x}} \subseteq \TT$ such that  $\pi(\mathcal{P}(\tilde{x}))\supseteq U_{\tix}$.
	\end{enumerate}
\end{definition}

The construction of subordinate partitions in \cite{LedStr} applies in our case. 

\begin{definition}
\label{def:srb}
    We say that $\mu$ is an SRB measure if there exists a cu-subordinate partition $\cP$ such that for $\tm$-almost every $\tix$ one has $\pi_{*}\tm^{\cP}_{\tix}<<\Leb_{\TT}$.
\end{definition}

\begin{lemma}
    \label{lem:srbendo}
Every SRB measure is absolutely continuous with respect to $\Leb_{\TT}$.
\end{lemma}
\begin{proof}
Let us assume by contradiction that $\mu$ is not absolutely continuous with respect to $\Leb_{\TT}$. Then, there exists a Borel measurable set $A\subset\TT$ with $\mu(A)>0$ but $\Leb_{\TT}(A)=0$. Then, $\widetilde{A}\eqdef\pi^{-1}(A)$ satisfies $\tm(\widetilde{A})>0$. By the regularity of $\mu$ we may assume that $A$ is compact and contained in an open set $U=B_r(x)$, for some $x\in A$, such that $\pi|_{\pi^{-1}(U)\cap\w^{cu}(\tilde{y})}$ is a bijection for every $\tilde{y}\in\pi^{-1}(x)$. We may also choose the point $x$ so that for $\mu_x$ almost every $\tilde{y}\in\pi^{-1}(x)$ the measure $\pi_*\tm^{\cP}_{\tilde{y}}$ is absolutely continuous with respect to $\Leb_{\TT}$. In particular, we have that
\[
\tm^{\cP}_{\tilde{y}}(\widetilde{A})=0,
\]
for $\mu_x$ almost every $\tilde{y}\in\pi^{-1}(x)$. However, this implies that 
\[
\mu(A)=\tm(\widetilde{A})=\int\tm^{\cP}_{\tilde{y}}(\widetilde{A})d\mu_x(\tilde{y})=0,
\]
which gives a contradiction.
\end{proof}

\subsubsection{Distortion estimates}

Given $\delta>0$, we can choose a compact set $\cL\subset\TTf$ with $\tilde{\mu}(\cL)>1-\delta$ and such that the functions $C,r,\gamma$ are continuous when restricted to $\cL$. For this section, we fix $\delta>0$ and $\cL$. 

The following result is a classical consequence of uniform expansion along unstable manifolds or center Pesin manifolds.

\begin{lemma}
\label{lem:bounded-dist}
Let $\varphi: \TTf \to \R$ be a H\"older continuous function. There exists constants $C_0>0$ and $C=C(\cL)>0$ such that
\begin{enumerate}
	\item  if $\tilde{y} \in \w^u_1(\tix)$ then for every $n>0$, 
	 $$\left\vert \sum_{l=0}^{n-1} \varphi(\tilde{f}^{-l}(\tix)) - \varphi(\tilde{f}^{-l}(\tilde{y}))  \right\vert \leq C_0;$$
	 
	\item if $\tilde{f}^{l}(\tilde{y}) \in \w^u_{1}(\tilde{f}^{l}(\tix))$ for $l \in \{0, \cdots, n-1 \}$ then  $$\left\vert \sum_{l=0}^{n-1} \varphi(\tilde{f}^l(\tix)) - \varphi(\tilde{f}^l(\tilde{y}))  \right\vert \leq C_0;$$
    \item if $\tilde{y} \in \tW^c_{r(\tix)}(\tix)$ and $\tix \in \cL$, then for every $n > 0$,
        $$\left\vert \sum_{l=0}^{n-1} \varphi(\tilde{f}^{-l}(\tix)) - \varphi(\tilde{f}^{-l}(\tilde{y}))  \right\vert \leq C.$$
    \item if $\tilde{f}^{l}(\tilde{y}) \in \tW^c_{r(\tilde{f}^{l}(\tix))}(\tilde{f}^{l}(\tix))$ for $l \in \{0, \cdots, n-1 \}$ and $\tilde{f}^{n}(\tix) \in \cL$ then 
        $$\left\vert \sum_{l=0}^{n-1} \varphi(\tilde{f}^l(\tix)) - \varphi(\tilde{f}^l(\tilde{y}))  \right\vert \leq C.$$
\end{enumerate}
\end{lemma}

\begin{proof}
    We write the proof for (2) and the center leaf, the others being analogous. If $\tix \in \cL$ and $\tilde{y} \in \tW^c_{r(\tix)}(\tix)$, then
$$d(\tilde{f}^{-n}(\tilde{x}), \tilde{f}^{-n}(\tilde{y})) \leq \gamma_0 e^{-n\lambda} d(\tilde{x}, \tilde{y})$$
for all $n \in \mathbb{N}$, where $\gamma_0=\max \gamma|_{\cL}$.

We have that
$$\left\vert \sum_{l=0}^{n-1} \varphi(\tilde{f}^{-l}(\tix)) - \varphi(\tilde{f}^{-l}(\tilde{y}))  \right\vert \leq K \sum_{l=0}^{n-1} d(\tilde{f}^{-l}(\tix), \tilde{f}^{-l}(\tilde{y}))^{\alpha},$$
where $\varphi$ is $\alpha$-H\"older with constant $K$. This is less or equal to
\[
K \gamma_0 d(\tilde{x}, \tilde{y}) \sum_{l=0}^{n-1} e^{-l\lambda} \leq K \gamma_0 \operatorname{diam}(\TTf) \sum_{l=0}^{\infty} e^{-l\lambda} = C.\qedhere
\] 
\end{proof}

We immediately have the following result.

\begin{lemma}
    \label{lem:distor-unstable}
There exists a constant $C=C(f)>1$ such that if $\tif^n(\tilde{y})\in \cW^u_{1}(\tif^n(\tix))$ then for every $\ell\leq n$
\[
\frac{\lambda^u_{\tix}(\ell)}{\lambda^u_{\tilde{y}}(\ell)}\asymp_{C}1.
\]
\end{lemma}

Also, as a direct consequence of our choice of $\mathcal{L}$, we have similar estimates for the center direction.

\begin{lemma}
    \label{lem:distorlusin}
Given a Lusin set $\cL$ there exists a constant $C=C(\cL)>1$ such that if $\tif^n(\tix)\in\cL$ and $\tif^n(\tilde{y})\in \tW^c_{r(\tif^n(\tix))}(\tif^n(\tix))$ then for every $\ell\leq n$
\[
\frac{\lambda^c_{\tix}(\ell)}{\lambda^c_{\tilde{y}}(\ell)}\asymp_{C}1.
\]
\end{lemma}

\subsection{The Lyapunov norm}

Given $v\in E^c(\tix)\setminus\{0\}$ and $0< \lambda < L^c(\mu)$ we define
\[
\|v\|_{L,\tix}\eqdef \left(\sum_{k\leq 0}\|D\tif^k(\tix)v\|^2e^{-2k\lambda}\right)^{1/2},
\]
which is called the \emph{Lyapunov norm of $v$}. Observe that all terms in the series are non-negative, but in principle the series could diverge. The lemma below guarantees that this is not the case by providing a measurable upper bound with respect to the background Riemannian norm. 

\begin{lemma}
    \label{lem:lyapunovriemmann}
There exists a measurable function $\xi:\TTf\to [1,+\infty)$ such that for $\tilde{\mu}$-almost every $\tix\in\TTf$ and every $v\in E^c(\tix)$ it holds
$\|v\|\leq\|v\|_{L,\tix}\leq\xi(\tix)\|v\|$.
\end{lemma}
\begin{proof}
Denote, for simplicity, $\chi\eqdef L^c(\mu)$. Choose $0<\eps<\chi-\lambda$. Then, there exists a measurable function $\hat{C}:\TTf\to[1,+\infty)$ such that 
\[
\hat{C}(\tix)^{-1}e^{k(\chi-\eps)}\geq\lambda^c_{\tix}(k)\geq \hat{C}(\tix)e^{k(\chi+\varepsilon)},\:\:\textrm{for all}\:\:k<0.
\]
This implies that 
\[
\|D\tif^k(\tix)v\|^2=\lambda^c_{\tix}(k)^2\|v\|^2\leq \hat{C}(\tix)^{-2}e^{2k(\chi-\eps)}\|v\|^2.
\]
Thus defining $\xi(\tix)\eqdef \left(\hat{C}(\tix)^{-2}\sum_{k\leq 0}e^{2k(\chi-\eps-\lambda)}\right)^{1/2}$ we obtain
\[
\|v\|^2_{L,\tix}=\sum_{k\leq 0}\|D\tif^k(\tix)v\|^2e^{-2k\lambda}\leq\xi(\tix)^2\|v\|^2. \qedhere
\]
\end{proof}

\subsubsection{Growth estimate for norm $\|.\|_{L,\tix}$}

When we measure the length of a vector $v\in E^c(\tix)$ using the norm $\|.\|_{L,\tix}$, we see expansion in a uniform fashion. The drawback is that the assignment $x\mapsto \|.\|_{L,\tix}$ is only measurable. In the sequel we use the notation $\tix_n=\tif^n(\tix)$.

\begin{lemma}
    \label{lem:crescelyapunovcresce}
There exists a constant $\sigma=\sigma(f)>1$ such that for $\tilde{\mu}$ a.e. $\tix\in\TTf$ and every $v\in E^c(\tix)$ it holds 
\[
e^{n\lambda}\|v\|_{L,\tix}\leq\|D\tif^n(\tix)v\|_{L,\tix_n}\leq e^{n\sigma}\|v\|_{L,\tix},\:\:\:\textrm{for every}\:\:n>0
\]
\end{lemma}


\begin{proof}
Take $v\in E^c(\tix)\setminus\{0\}$ and $n>0$. Then,
\[
\|D\tif^n(\tix)v\|^2_{L,\tix_n}=\sum_{k\leq 0}\|D\tif^k(\tif^n(\tix))D\tif^n(\tix)v\|^2e^{-2\lambda k}=\sum_{k\leq 0}\|D\tif^{n+k}(\tix)v\|^2e^{-2\lambda k}.
\]
By making the change $\ell\rightarrow k+n$ in the sum, we obtain
\begin{align*}
\|D\tif^n(\tix)v\|^2_{L,\tix_n}&=\sum_{\ell\leq 0}\|D\tif^{\ell}(\tix)v\|^2e^{-2\lambda(\ell-n)}+\sum_{\ell=1}^{n-1}\|D\tif^{\ell}(\tix)v\|^2e^{-2\lambda(\ell-n)}\nonumber \\
&\eqdef I+II.\nonumber
\end{align*}

The first term on the right-hand side above is equal to 
\[
I=e^{2\lambda n}\sum_{\ell\leq 0}\|D\tif^\ell(\tix)v\|^2e^{-2\lambda\ell}=e^{2\lambda n}\|v\|^2_{L,\tix}.
\]
As the second term is non-negative, this gives immediately the lower bound. For the upper bound, take $\xi>\max\{\log\|D\tif(\tix)\|;\tix\in\TTf\}>1$. Notice that $n<e^{n\xi}$ for every $n>0$. By definition, for every $\ell=1,...,n$ we have 
\[
\|D\tif^\ell(\tix)v\|^2\leq e^{2\xi\ell}\|v\|^2\leq e^{2\xi n}\|v\|^2.
\]
This allows us to estimate the second term II above by
\begin{align*}
  II&\leq\|v\|^2\sum_{\ell=1}^{n-1}e^{2\xi n}e^{-2\lambda(\ell-n)}\leq\|v\|^2\sum_{\ell=1}^{n-1}e^{4\xi n}e^{-2\lambda \ell}\nonumber\\
    &\leq e^{4\xi n}\|v\|^2\sum_{\ell=1}^{n-1}e^{-2\lambda\ell}\leq ne^{4\xi n}\|v\|^2\nonumber\\
    &\leq e^{5\xi n}\|v\|^2\leq e^{5\xi n}\|v\|_{L,\tix}^2\nonumber. 
\end{align*}
    
Thus
\[
I+II\leq(e^{2\lambda n}+e^{5\xi n})\|v\|_{L,\tix}\leq e^{7\xi n}\|v\|^2_{L,\tix}.
\]
This establishes the upper bound and completes the proof. 
\end{proof}

\subsubsection{$u$-saturation of the well behaved set for $\|.\|_{L,\tix}$}

\begin{lemma}
\label{lem:saturalyapunov}
    If $\tix$ is $\tilde{\mu}$-regular and if $\tix^u\in\tilde{\cW}^u(\tix)$, then $0<\|v\|_{L,\tix^u}<+\infty$ for every $v\in E^c(\tix^u)$. The same holds for every $\tilde{y}\in\tW^c_{r(\tix)}(\tix)$.
\end{lemma}
\begin{proof}
As $\tix\mapsto\log\|D\tif(\tix)|_{E^c}\|$ is Hölder continuous it follows from Lemma~\ref{lem:bounded-dist} that 
\[
\frac{\lambda^c_{\tix}(k)}{\lambda^c_{\tix^u}(k)}\asymp_{C}1,
\]
for some constant $C>1$ (which depends only on $f$ and $d(\tix,\tix^u)$). Since $E^c$ is one-dimensional we can therefore write 
\begin{align*}
        \|v\|^2_{L,\tix^u}&=\sum_{k\leq 0}\|D\tif^k(\tix^u)v\|^2e^{-2\lambda k}\nonumber\\
    &=\|v\|^2\sum_{k\leq 0}\lambda^c_{\tix^u}(k)^2e^{-2\lambda k}\nonumber\\
    &\asymp_{C^2}\|v\|^2\sum_{k\leq 0}\lambda^c_{\tix}(k)^2e^{-2\lambda k}=\|v\|_{L,\tix},
\end{align*}
which implies the result. The proof for $\tilde{y}\in\w^c_{r(\tix)}(\tix)$ is identical. Indeed, the argument for the distortion Lemma~\ref{lem:bounded-dist} shows that 
\[
\frac{\lambda^c_{\tix}(-n)}{\lambda^c_{\tilde{y}}(-n)}\asymp_C 1,
\]
for every $n>0$, where $C=C(\tix)>1$. It follows as before that $\|v\|^2_{L,\tilde{y}}\asymp_{C^2}\|v\|_{L,\tix}^2$.   
\end{proof}

\subsubsection{Distortion estimates}

We consider now a Lusin set $\cL\subset\TTf$ of measure $\tilde{\mu}(\cL)>1-\delta$.

\begin{notation}
 Given $\tix\in\TTf$ and $n\in\N$, we define $\hat{\lambda}^c_{\tix}(n) \eqdef \frac{\|D\tif^n(\tix)v\|_{L,\tix_n}}{\|v\|_{L,\tix}}$, for  $v\in E^c(\tix)\setminus\{0\}$.
 It is easy to check that
 \begin{equation}
     \label{eq:cocycle}
     \hat{\lambda}^c_{\tix}(m+n) = \hat{\lambda}^c_{\tilde{f}^n(\tilde{x})}(m) \hat{\lambda}^c_{\tix}(n). 
 \end{equation}
\end{notation}
\begin{remark}
\label{rem:lyapunov}
Observe from Lemma~\ref{lem:crescelyapunovcresce} the following estimate: 
\[
e^{n\lambda}\leq\hat{\lambda}^c_{\tix}(n)\leq e^{n\sigma},
\]
for every $n\in\Z$ and for a constant $0<\lambda<L^c(\mu)$, which is fixed, and for $\sigma=\sigma(f)$ the constant given in Lemma~\ref{lem:crescelyapunovcresce}. This estimation only reflects the fact that when measuring growth with the Lyapunov norm we see immediate expansion along $E^c$. This will be of crucial importance later when we establish our quasi-isometric and synchronization estimates for stopping times. 
\end{remark}

We begin with a distortion result along Pesin center manifolds. 

\begin{lemma}
    \label{lem:distorlyp}
There exists a constant $C=C(\cL)>1$ such that, given $n>0$, if $\tix_n\in\cL$ and $\tilde{y}_n\in\tW^c_{r(\tix)}(\tix_n)$, then 
\[
\frac{\hat{\lambda}^c_{\tix}(n)}{\hat{\lambda}^c_{\tilde{y}}(n)}\asymp_C 1.
\]
\end{lemma}
\begin{proof}
Notice that we are under the same assumptions of Lemma~\ref{lem:distorlusin}. Let $C=C(\cL)$ be the constant given by that lemma. Then, given any non-zero vector $v\in E^c(\tix)$ we have that
\begin{align*}
        \hat{\lambda}^c_{\tix}(n)^2&=\frac{\|v\|^2}{\|v\|^2_{L,\tix}}\sum_{k\leq 0}\lambda^c_{\tix}(n+k)^2e^{-2\lambda k}\nonumber\\
    &\asymp_{C^2}\frac{\|v\|^2}{\|v\|^2_{L,\tix}}\sum_{k\leq 0}\lambda^c_{\tilde{y}}(n+k)^2e^{-2\lambda k}.\nonumber\\
    &=\frac{\|v\|^2_{L,\tilde{y}}}{\|v\|^2_{L,\tix}}\hat{\lambda}^c_{\tilde{y}}(n)^2.\nonumber
\end{align*}
The proof of Lemma~\ref{lem:saturalyapunov} shows that $\|v\|_{L,\tilde{y}}/\|v\|_{L,\tix}\asymp_{C} 1$. We deduce that 
\[
\hat{\lambda}^c_{\tix}(n)^2\asymp_{C^4}\hat{\lambda}^c_{\tilde{y}}(n)^2,
\]
which gives the desired result. 
\end{proof}

We turn now to the case of two points on the same fiber of $\TTf$. In this case we need to assume that both points belong to the Lusin set. Indeed, for the next lemma we demand another property of the Lusin set. 

\begin{lemma}
    \label{lem:distlusinnasfibras}
    Given a Lusin set $\cL\subset\TTf$ there exists a constant $C=C(\cL)>1$ such that whenever $\tix,\tilde{y}\in\cL$ with $\tilde{y}\in\Sigma(\tix)$ we have that 
    \[
    \frac{\hat{\lambda}^c_{\tix}(n)}{\hat{\lambda}^c_{\tilde{y}}(n)}\asymp_C1,\:\:\:\textrm{for all}\:\:\:n>0.
    \]
\end{lemma}
\begin{proof}
    The function $\xi$ of Lemma~\ref{lem:lyapunovriemmann} is uniformly bounded in $\cL$. Also, given any tangent vector $v\in E^c(\tix)$, there exists a unique (up to orientation) vector $\hat{v}\in E^c(\tilde{y})$ such that $\|v\|=\|\hat{v}\|$. 

    Now, write as in the proof of Lemma~\ref{lem:crescelyapunovcresce} that
    \[
    \|D\tif^n(\tix)v\|^2_{L,\tix_n}=e^{2\lambda n}\|v\|^2_{L,\tix}+\sum_{j=1}^n\|D\tif^j(\tix)v\|^2e^{-2\lambda(j-n)},
    \]
    for all $v\in E^c(\tix)$ and for all $n>0$. Obviously, the same formula holds with $\tix$ and $v$ replaced by $\tilde{y}$ and $\hat{v}$. Since $\tix,\tilde{y}\in\cL$ we have that 
    \[
    \|v\|^2_{L,\tix}\asymp_C\|v\|=\|\hat{v}\|\asymp_C\|\hat{v}\|^2_{L,\tilde{y}}
    \]
    The basic distortion lemma gives that $\|D\tilde{f}^j(\tix)v\|\asymp_C\|D\tif^j(\tilde{y})\hat{v}\|$. Up to enlarging the constant $C$ we get $\|D\tif^n(\tix)v\|_{L,\tix}\asymp_C\|D\tif^n(\tilde{y})\hat{v}\|_{L,\tilde{y}}$.  Therefore, up to enlarging $C$ once more, we have
    \[
    \hat{\lambda}^c_{\tix}(n)\asymp_C\frac{\|v\|_{L,\tix}}{\|\hat{v}\|_{L,\tilde{y}}}{\hat{\lambda}^c_{\tilde{y}}(n)}\asymp_C\hat{\lambda}^c_{\tilde{y}}(n). \qedhere
    \]
\end{proof}

The following distortion result concerns points in the same unstable dynamical ball. The proof follows from the same calculations we have done in this section, so we refrain from repeating the details.

\begin{lemma}
    \label{lem:distodynball}
There exists a constant $C=C(f)>1$ depending only on $f$ such that for every $\tilde{\mu}$ generic point $\tix\in\TTf$ and every $\tilde{a}\in\w^u_1(\tilde{x})$, if $\tilde{a}_n\in\w^u_1(\tix_n)$, for some $n\in\N$ then 
\[
\frac{\hat{\lambda}^c_{\tilde{a}}(j)}{\hat{\lambda}^c_{\tilde{x}}(j)}\asymp_C 1,\:\:\:\textrm{for every}\:\:j=1,\dots,n.
\]
\end{lemma}

\section{Normal forms and leaf-wise quotient measures}
\label{sec:normal}

This section collects a number of tools we use in our argument. To motivate these tools, recall that our goal is to prove that, under the hypotheses of Theorem \ref{teo:main.rigidez}, a given u-Gibbs measure is absolutely continuous with respect to $\operatorname{Leb}_{\TT}$. Since, by definition, a u-Gibbs measure is (in some sense) homogeneous along unstable leaves, we claim that it is enough to show that the measure is also homogeneous along center leaves. This claim is made precise by the following result, which we prove in Section \ref{sec:leafwise}.

\begin{lemma}
	\label{lem:temqueserlebesgue}
There is a family of locally finite Radon measures $\{\hat{\nu}^c_{\tix}\}_{\tix \in \mathbb{T}^2_f}$ on $\mathbb{R}$ such that, if $\hat{\nu}^c_{\tix}$ is equivalent to Lebesgue for $\tilde{\mu}$-almost every $\tix$, then $\mu$ is absolutely continuous with respect to Lebesgue.
\end{lemma}

This essentially translate our main problem into a one-dimensional question and it is the main goal of this section. In order to prove this lemma, we introduce a crucial tool: normal forms. 

All the arguments in this section follow closely \cite[Sections 6--7]{ALOS}, so we only give a brief account for the sake of completeness.

\subsection{One-dimensional normal forms}
The result below, which comes from \cite{KalininKatok}, provides us with invariant affine structures along Pesin center manifolds and unstable manifolds. 

\begin{proposition}
\label{prop:formanormal}
For $\tilde{\mu}$-almost every point $\tilde{x}\in \mathbb{T}^2_f$, there are diffeomorphisms $\mathcal{R}^u_{\tix}: \w^u(\tix) \to \R$ and $\cR^c_{\tix}:\tW^c(\tix)\to\R$, which are both as regular as $f$, such that 
\begin{enumerate}
	\item $\mathcal{R}^{*}_{\tix}(\tix)=0$ and $D\mathcal{R}^{*}_{\tix}(\tix)v^{*}=1$, $* \in \{c,u\}$, where $v^{*}\in E^{*}(\tix)$ is a unity vector.
	\item Let $\Lambda^{*}_{\tix}:\R\to\R$ denote the linear map $\Lambda^{*}_{\tix}(t)=\lambda^{*}_{\tix}t$. Considering $\Phi^{*}_{\tix}=(\mathcal{R}^{*}_{\tix})^{-1}$, we have
	\[
	\mathcal{R}^{*}_{\tilde{f}(\tix)}\circ\tilde{f}\circ\Phi^{*}_{\tix}=\Lambda^{*}_{\tix}.
	\]
        \item $\mathcal{R}^u_{\tix}$ depends continuously on $\tix$, and $\mathcal{R}^c_{\tix}$ depends continuously on $\tix$ on a Pesin set or along the fiber $\Sigma(\tix)$.
        \item For $\tilde{y} \in \w^u(\tix)$ or $\tilde{y}\in \tW^c(\tix)$ we have that
            $$\mathcal{R}^{*}_{\tix, \tilde{y}} \eqdef \mathcal{R}^{*}_{\tilde{y}} \circ (\mathcal{R}^{*}_{\tix})^ {-1}$$
        is affine and satisfies $\mathcal{R}^{*}_{\tix, \tilde{y}}(t) = \rho^{*}_{\tilde{y}}(\tilde{x})t +  \mathcal{R}^{*}_{\tilde{y}}(\tix)$, where $\rho^{*}_{\tilde{y}}$ is a continuous function.
\end{enumerate}
\end{proposition}

\begin{proof}
We follow the arguments of \cite{KalininKatok} using the differential structure of $f$ pulled back to the inverse limit. The proof is for ${*} = c$, with the case of the unstable direction being simpler. We define $\rho^c_{\tilde{x}}: \tW^c(\tix) \to (0, \infty)$ as

\begin{equation}
	\label{eq:rho}
	\rho^c_{\tilde{x}}(\tilde{y}) \eqdef \lim_{n \to \infty} \dfrac{\left\Vert \restr{D\tilde{f}^{-n}(\tilde{y})}{E^c(\tilde{y})}\right\Vert}{\left\Vert \restr{D\tilde{f}^{-n}(\tilde{x})}{E^c(\tilde{x})}\right\Vert} = \prod_{i=0}^\infty \dfrac{\lambda^c_{y_{-i}}}{\lambda^c_{x_{-i}}},
\end{equation}
where $\lambda^c_{\tilde{x}} = \Vert Df(x)|_{E^c_{\tilde{x}}} \Vert$ is our notation for the norm of the derivative restricted to the center direction. As in \cite{KalininKatok}, $\rho^c_{\tilde{x}}(\cdot)$ is well defined and $\alpha$-H\"older continuous on each Pesin manifold $\tW^c(\tix)$ if $f$ is $C^{1+\alpha}$, or $C^{k-1}$ if $f$ is $C^k$, $k \geq 2$. Additionally, for $\tilde{y}, \tilde{z} \in \tW^c(\tix)$,
\begin{equation}
    \label{eq:1cancela}
    \rho^c_{\tilde{x}}(\tilde{z}) = \rho^c_{\tilde{y}}(\tilde{z}) \rho^c_{\tilde{x}}(\tilde{y}).
\end{equation}

Consider $\mathcal{R}^c_{\tilde{x}}: \tW^c(\tix) \to \mathbb{R}$ defined as
\begin{equation}
	\label{eq:R}
	\mathcal{R}^c_{\tilde{x}}(\tilde{y}) \eqdef \int_{\tix}^{\tilde{y}} \rho^c_{\tilde{x}}(\tilde{z}) d\tilde{z}.
\end{equation}

All its properties follow as in \cite{KalininKatok}. To extend this function to the one-dimensional center leaf $\w^c(\tix)$, given $\tilde{y} \in \w^c(\tix)$, we take $n \in \mathbb{N}$ such that $\tif^{-n}(\tilde{y}) \in \tW^c(\tif^{-n}(\tix))$. Then we define
$$\mathcal{R}^c_{\tilde{x}}(\tilde{y}) = D\tilde{f}^n \circ \mathcal{R}^c_{\tilde{f}^{-n}(\tix)} \circ \tilde{f}^{-n}(\tilde{y}).$$

$\mathcal{R}^c_{\tilde{x}}(\cdot)$ is as regular as $f$. Its restriction to a fixed-size ball in the Pesin center leaf $\tW^c(\tix)$ depends continuously on $\tix$ on a Pesin set. Additionally, if $\tilde{y} \in \Sigma(\tilde{x})$, we have that $x_{-i} = y_{-i}$ for all $i \leq k$, with $k \in \mathbb{N}$, and $k \to \infty$ as $\tilde{d}(\tix, \tilde{y}) \to 0$. This implies that $\tilde{x} \mapsto \mathcal{R}^c_{\tilde{x}}(\cdot)$ is continuous on each fiber $\Sigma$.
\end{proof}

\begin{remark}
    \label{rem:R-dist}
    $\mathcal{R}^u_{\tix}$ defines in $\w^u(\tix)$ a distance $d^u_{\mathcal{R}}(\tix, \tilde{y}) = \mathcal{R}^u_{\tix} (\tilde{y})$ ($\tilde{y} \in \w^u(\tix)$) such that, for all $C > 0$ there is $N > 0$ such that $d^u(\tilde{x}, \tilde{y}) < C$ implies
    $$\dfrac{1}{N} d^u_{\mathcal{R}}(\tix, \tilde{y}) \leq d^u(\tilde{x}, \tilde{y}) \leq N d^{u}_{\mathcal{R}}(\tix, \tilde{y}),$$
    where $d^u$ is the Riemannian distance along the leaf. In a similar way, $\cR^c_{\tix}$ defines a measurably varying distance along the Pesin global manifold $\tW^c(\tix)$ which is uniformly comparable to the Riemannian length for $\tix$ in a Lusin set $\cL$.
\end{remark}

\subsection{Two dimensional normal forms on Pesin center-unstable manifolds}

Now we define a new two-dimensional structure for the center unstable Pesin manifold $\tW^{cu}(\tix)$. This structure is not affine in general, but it is affine for points in the same center leaf, as we made precise in Theorem \ref{teo:formanormal2}. This result follows \cite{ALOS} and use Proposition \ref{prop:formanormal} and the fact that $E^u(\tilde{y})$ is $C^1$ inside $\w^{cu}(\tix)$.

Consider $\tix \in \mathbb{T}^2_f$, $\tilde{y} = \Phi^c_{\tix}(s) \in W^{c}(\tix)$ and
\begin{equation}
    \label{eq:beta}
    \beta_x(s) \eqdef \rho^u_{\tilde{y}}(\tix) =  \rho^u_{\Phi^c_{\tix}(s)}(\tix).
\end{equation}

\begin{theorem}
    \label{teo:formanormal2}
    For $\tilde{\mu}$-almost every point $\tilde{x}\in \mathbb{T}^2_f$, the function $\Phi_{\tix}: \R^2 \to \tW^{cu}_{\tix}$, defined as
    \begin{equation}
        \label{eq:Phi}
        \Phi_{\tix}(t,s) \eqdef \Phi^u_{\tilde{y}}(\rho^u_{\tilde{y}}(\tix) t)= \Phi^u_{\Phi^c_{\tix}(s)}(\beta_{\tix}(s) t),
    \end{equation}
    where $\tilde{y} = \Phi^c_{\tix}(s)$, is a $C^1$ diffeomorphism. Moreover, for all $t,s \in \R$
    \begin{enumerate}
    \item $\tilde{f} \circ \Phi_{\tix}(t,s) = \Phi_{\tilde{f}(\tix)}(\lambda^u_{\tix}t,\lambda^c_{\tix}s)$;
        \item $\Phi_{\tix}(\R \times \{s\}) = \w^u(\tilde{y})$;
        \item $\Phi_{\tix}(\{0\} \times \R) = \tW^c(\tilde{x})$.
    \end{enumerate}
        Additionally, by defining $\mathcal{R}_{\tix} \eqdef (\Phi_{\tix})^{-1}$, we have that the map
        $$\mathcal{R}_{\tilde{x},\tilde{y}} \eqdef \mathcal{R}_{\tilde{y}} \circ \Phi_{\tilde{x}}: \R^2 \to \R^2$$ has the following properties
    \begin{enumerate}
    \setcounter{enumi}{3}
        \item It can be written as $\mathcal{R}_{\tilde{x},\tilde{y}}(t,s) = (h^1_{\tilde{x},\tilde{y}}(t,s), h^2_{\tilde{x},\tilde{y}}(s))$ (the second coordinate does not depend on $t$);
        \item if $\tilde{y} \in \tW^c(\tilde{x})$, then the first coordinate does not depend on $s$ and we have that
        \begin{align*}
            h^1_{\tilde{x},\tilde{y}}(t) &= \rho^u_{\tilde{y}}(\tilde{x})t,\\
            h^2_{\tilde{x},\tilde{y}}(s) &= \rho^c_{\tilde{y}}(\tilde{x})s + \mathcal{R}^c_{\tilde{y}}(\tilde{x});
        \end{align*}
        \item if $\tilde{x}' \in \w^u(\tilde{x})$, then 
        \begin{align*}
            h^1_{\tilde{x},\tilde{x}'}(t,s) &= \rho^u_{\tilde{x}}(\tilde{x}')t + a(\tilde{x}, \tilde{x}', s),\\
            h^2_{\tilde{x},\tilde{x}'}(s) &= \rho^c_{\tilde{x}}(\tilde{x}')s,
        \end{align*}
        where $s\mapsto a(\tilde{x}, \tilde{x}', s)$ is a $C^1$ function.
    \end{enumerate}
\end{theorem}

\begin{proof}
We follow closely the argument in \cite{ALOS}, pointing out only the main steps. The core argument is to show that $\Phi_{\tix}$ is a $C^1$ diffeomorphism satisfying (1). This has two main parts. The first (see Lemma 6.7 of \cite{ALOS}) is to show that $\beta_{\tix}$ is a $C^1$ function satisfying the reparametrization rule
\[
\beta_{\tif(x)}(\lambda^c_{\tix}s)=\frac{\lambda^u_{\Phi^c_{\tix}(s)}}{\lambda^u_{\tix}}\beta_{\tix}(s).
\]
For this, the backwards exponential contraction under the action of $\tif$ comes into play. The only difference from \cite{ALOS} is that as this contraction has uniform constants only inside a Lusin set, thus we obtain that the assignment $\tix\mapsto\beta_{\tix}$ is only continuous on a Lusin set. 

The second step (see Lemma 6.6 of \cite{ALOS}) is the proof that $(t,s)\mapsto\Phi^u_{\Phi^c_{\tix}(s)}(t)$ is a $C^1$ diffeomorphisms onto $\tW^{cu}(\tix)$. Again, the only difference from \cite{ALOS} is that the backward contraction has uniform constants only inside a Lusin set. The proof of all other announced properties, although long, is identical. 
\end{proof}

\subsection{Leaf-wise quotient measures }
\label{sec:leafwise}

Following previous works \cite{BenoistQuintI, BRH, Katz, ALOS}, we construct a family of locally finite Borel measures $\hat{\nu}^c_{\tix}$ on the real line which are one of the main objects of our paper. This construction starts with the center-unstable manifolds of the lifted map $\tilde{f}$.

Consider $\mathcal{P}^{cu}_0$ a measurable partition of $\mathbb{T}^2_f$ subordinated to $\tW^{cu}$, $\mathcal{P}^{cu}_n \eqdef \tilde{f}^n(\mathcal{P}^{cu}_0)$ and $\mathcal{Q}^{cu}_n({\tix}) \eqdef \mathcal{R}_{\tix}(\mathcal{P}^{cu}_n({\tix})) \subseteq \R^2$. We have, by Rokhlin disintegration theorem, that there are conditional probability measures $\{ \tilde{\mu}^{cu}_{n, \tix} \}_{\tix \in \mathbb{T}^2_f}$ supported at the atoms of $\mathcal{P}^{cu}_n$. We also have that $$\nu^{cu}_{n, \tix} \eqdef (\mathcal{R}_{\tilde{x}})_*\tilde{\mu}^{cu}_{n, \tix} = \tilde{\mu}^{cu}_{n, \tix} \circ \Phi_{\tix}$$ is a measure supported at $\mathcal{Q}^{cu}_n(\tix)$.

Note that, since the ergodic system $(\tif,\tm)$ is non-uniformly expanding for a $\tm$ typical point $\tix$, $\mathcal{P}^{cu}_n(\tix) \subseteq \mathcal{P}^{cu}_m(\tix)$ if $m > n$. Additionally, for any Borel set $A \subseteq \mathbb{T}^2_f$,
$$\tilde{\mu}^{cu}_{n, \tix}(A \cap \mathcal{P}^{cu}_{n}(\tix)) = \dfrac{\tilde{\mu}^{cu}_{m, \tix}(A \cap \mathcal{P}^{cu}_{n}(\tix))}{\tilde{\mu}^{cu}_{m, \tix}(\mathcal{P}^{cu}_{n}(\tix))} .$$

Consider now, for every $\tilde{y} \in \mathcal{P}^{cu}_n(\tix)$, 
$$\mathcal{P}^u_n(\tilde{y}) \eqdef \w^u(\tilde{y}) \cap \mathcal{P}^{cu}_n(\tix),$$
which gives a measurable partition subordinated to $\w^u$ satisfying $\mathcal{P}^{u}_n = \tilde{f}^n(\mathcal{P}^{u}_0)$. Again, by Rokhlin disintegration theorem, we can decompose any measure $\tilde{\mu}^{cu}_{n, \tix}$ supported in $\mathcal{P}^{cu}_n(\tix)$ in a system of conditional measures $\{ \tilde{\mu}^{u}_{n, \tilde{y}} \}$ supported at the atoms $\mathcal{P}^{u}_n(\tilde{y})$. We also have a quotient measure $\tilde{\mu}^{c}_{n, \tilde{x}}$ on the partition whose atoms are given by $\cP^c_n(\tix)\eqdef \tW^c(\tix)\cap\cP^{cu}_n(\tix)$, which is defined by the equation $$\tilde{\mu}^{c}_{n, \tilde{x}}(A) = \tilde{\mu}^{cu}_{n, \tilde{x}}\left(\bigcup\limits_{a \in A} \cP^u_n(a)\right),$$ for all $A \subseteq \mathcal{P}^{c}_n.$

For any Borelian $B \subseteq \mathcal{P}^{cu}_n(\tix)$ we have
$$\tilde{\mu}^{cu}_{n, \tilde{x}}(B) = \int \tilde{\mu}^{u}_{n, \tilde{y}}(B) d\tilde{\mu}^{c}_{n, \tilde{x}}(\tilde{y}).$$

As we did for the two dimensional conditional measure $\tilde{\mu}^{cu}_{n, \tilde{x}}$, we also push $\tilde{\mu}^{u}_{n, \tilde{x}}$ forward to $\R^2$ as
$$\nu^{u}_{n, s} = \nu^{u}_{n, \tilde{y}} \eqdef (\mathcal{R}_{\tilde{x}})_*\tilde{\mu}^{u}_{n, \tilde{y}} = \tilde{\mu}^{u}_{n, \tilde{y}} \circ \Phi_{\tix},$$
a measure supported at $(\R \times \{s\}) \cap \mathcal{Q}^{cu}_n(\tix)$, where $\tilde{y} = \Phi^c_{\tix}(s) = \Phi_{\tix}(0, s)$.

Since $\mathcal{R}_{\tilde{x}}$ takes conditional measures on $\tW^{cu}(\tilde{y})$ to conditional measures with respect to the corresponding partition on $\R^2$, we have that
$$\nu^{cu}_{n, \tilde{x}}(B) = \int \nu^{u}_{n, s}(B) d\nu^{c}_{n, \tilde{x}}(s),$$
where $d\nu^{c}_{n, \tilde{x}}(s) \eqdef d(\mathcal{R}_{\tilde{x}})_*\tilde{\mu}^{c}_{n, \tilde{x}}(0,s)$ and $B \subset \mathcal{Q}^{cu}_n(\tix)$ is a Borelian.

As in \cite{ALOS} Lemma 7.6, we have the following.

\begin{lemma}
    \label{lem:nu-horiz-eh-leb}
    For $\tilde{\mu}$-a. e. $\tilde{x} \in \TTf$, $n \in \N$ and $\nu^{cu}_{n, \tilde{x}}$-a.e. $(t,s) \in \mathcal{Q}^{cu}_n({\tix})$, we have that
    $$d\nu^{u}_{n, s} = \gamma^u_n(s) dLeb_{(\R \times \{s\}) \cap \mathcal{Q}^{cu}_n(\tix)},$$
    where the density $\gamma^u_n(s)$ does not depend on $t$.
\end{lemma}    

Consider $I_1 = [-1, 1] \subseteq \R$ and $I = [a,b]$ with $a < 0 < b$. For $\tilde{\mu}$-a. e. $\tilde{x} \in \TTf$ and $B \subseteq \R$ a Borel subset, then there is $n_0 = n_0(x, I, I_1, B) \in \N$ such that $n \geq n_0$ implies $I \times (B \cup I_1) \subseteq \mathcal{Q}^{cu}_n(\tix)$.

\begin{lemma}
    \label{lem:nu-converge}
    For $\tilde{\mu}$-a. e. $\tilde{x} \in \TTf$ and $B \subseteq \R$ a Borel bounded subset, there is $n_0 \in \N$ such that $m > n \geq n_0$ implies
    $$\dfrac{\nu^{cu}_{n, \tilde{x}}(I \times B)}{\nu^{cu}_{n, \tilde{x}}(I \times I_1)} = \dfrac{\nu^{cu}_{m, \tilde{x}}(I \times B)}{\nu^{cu}_{m, \tilde{x}}(I \times I_1)}.$$
\end{lemma}

\subsubsection{Proof of Lemma~\ref{lem:temqueserlebesgue}} Lemma~\ref{lem:nu-converge} allows us to define the \textit{leaf-wise quotient measures}, by considering $n \in \N$ sufficiently large, 
$$\hat{\nu}^c_{\tix}(B) \eqdef \dfrac{\nu^{cu}_{n, \tilde{x}}(I \times B)}{\nu^{cu}_{n, \tilde{x}}(I \times I_1)} = \dfrac{\tilde{\mu}^{cu}_{n, \tilde{x}}(\Phi_{\tix}(I \times B))}{\tilde{\mu}^{cu}_{n, \tilde{x}}(\Phi_{\tix}(I \times I_1))},$$
which gives us indeed a measure on $\R$. Notice that $\hat{\nu}^c_{\tix}$ coincide with $\nu^c_{n,x}$ on $Q^c_{n}(\tix)$ (up to normalization). Thus, if $\hat{\nu}^c_{\tix}$ is absolutely continuous so is $\nu^c_{n,\tix}$. By Lemma~\ref{lem:nu-horiz-eh-leb} this implies that $\nu^{cu}_{n,\tix}$ is absolutely continuous with respect to the Lebesgue measure of $\R^2$. Since $\Phi_{\tix}$ is $C^1$ this proves that $\tilde{\mu}^{cu}_{n,\tix}$ is absolutely continuous with respect to the Lebesgue measure of the Pesin center manifold $W^{cu}(\tix)$. Thus, Lemma~\ref{lem:temqueserlebesgue} follows from Lemma~\ref{lem:srbendo}. \qed

\subsection{Basic moves of leaf-wise measures} 
\label{sec:normal-basicmoves}

The lemma bellow describes how leaf-wise quotient measures change when we move the base point in three different ways. These properties play a crucial role in the argument. The proof is a direct consequence of: the definition of leaf-wise quotient measures, the properties of the normal forms given in Proposition \ref{prop:formanormal} and Theorem \ref{teo:formanormal2}, and Lemma \ref{lem:nu-horiz-eh-leb}. It follows exactly as in \cite{ALOS}, so we omit it here. We say that two locally finite Radon measures on $\R$ are proportional, and we denote it by $\mu\propto\nu$, if $\mu=c\nu$, for some $c>0$.

\begin{lemma}
	\label{lem:basicmoves}
For $\tilde{\mu}$ almost every $\tix \in \TTf$, the following holds:
\begin{enumerate}
	\item $\hat{\nu}^c_{\tilde{f}(\tix)} \propto (\Lambda^c_{\tix})_*\hat{\nu}^c_{\tix}$, where $\Lambda^c_{\tix}: \R \to \R$ is given by $\Lambda^c_{\tix}(s) = \lambda^c_{\tix}s$;
	\item if $\tilde{y} \in \tW^c(\tix)$, then 
	\[
	\hat{\nu}^c_{\tilde{y}} \propto (h^2_{\tilde{x}, \tilde{y}})_*\hat{\nu}^c_{\tix};
	\]
	\item  if $\tilde{x}^u\in\w^u(\tix)$, then 
	\[
	\hat{\nu}^c_{\tilde{x}^u} \propto (h^2_{\tilde{x}, \tilde{x}^u})_*\hat{\nu}^c_{\tix}.
	\]
\end{enumerate} 
\end{lemma}

\section{Uniform transversality along fibers}
\label{sec:transv}

From now on we start the proof of Theorem~\ref{teo:main.rigidez}. In this section our setting is slightly more flexible: we consider $f:\TT\to\TT$ a partially hyperbolic, {strongly transitive} endomorphism and $\mu$ an ergodic invariant probability measure (not necessarily a $u$-Gibbs measure). We consider $\tif:\TTf\to\TTf$ and $\tm$ the lifts of $f$ and $\mu$ to the inverse limit space. Let $\{\tm_x\}_{x\in\TT}$ denote the conditional measures of $\tm$ with respect to the fibers of $\TTf$. 

\subsection{The angle function}
\label{sec:transv-bad}

We can measure angles between vectors in $T_{\tix}\TTf$ and $T_{\tilde{y}}\TTf$ using the canonical projection $\pi:\TTf\to\TT$. In particular, this allows us to measure angles between different unstable directions along the fiber $\Sigma(\tix)$. 

For $\tix\in\TTf$ and any given $\tilde{y}\in\Sigma(\tix)$ we denote
\[
\alpha(\tix,\tilde{y})\eqdef\angle(E^u(\tix),E^u(\tilde{y})).
\]
Notice that the angle function is continuous in the following sense: if $\tilde{x}_n\to\tix$ and $\tilde{y}_n\to\tilde{y}$, with $\tilde{y}_n\in\Sigma(\tix_n)$ for every $n\in\N$, then $\lim\alpha(\tix_n,\tilde{y}_n)=\alpha(\tix,\tilde{y})$.

We consider the compact set $\cP(\tix)=\{\tilde{y}\in\Sigma(\tix);\alpha(\tix,\tilde{y})=0\}$ and its complement $\cN(\tix)=\Sigma(\tix)\setminus\cP(\tix)$. The invariance of the unstable bundle implies that  
\begin{equation}
	\label{eq:P-N-invartiant}
\Sigma(\tix_m)\cap\tilde{f}^m(\cP(\tix))=\cP(\tix_m) \:\:\:\textrm{ and }\:\:\:\:\Sigma(\tix_m)\cap\tilde{f}^m(\cN(\tix))=\cN(\tix_m),
\end{equation}
for every $m\in\Z$.
\begin{remark}
    Notice that $f$ is special if, and only if, $\cN(\tix)=\emptyset$ for every $\tix\in\TTf$.
\end{remark}

Our assumptions that $f$ is a {strongly transitive} transformation of $\TT$ implies the following.  

\begin{lemma}
	\label{l.especialcontagia}
Assume that $f$ is not special. Then, for every fiber $\Sigma(\tix)$, the set $\cN(\tix)$ is not empty.
\end{lemma}
\begin{proof}

Assume that for some $\tix\in\TTf$ the set $\cN(\tix)$ is empty. Take any fiber $\Sigma(\tilde{a})\subset\TTf$ and any point $\tilde{b}\in\Sigma(\tilde{a})$. Then, by Corollary~\ref{cor:perorbitascdensas} there exist sequences of points such that $\tilde{a}_k\to \tilde{a}$ and $\tilde{b}_k\to \tilde{b}$, with $\tb_k\in\Sigma(\ta_k)$, and, for each $k$, both points $\tilde{a}_k$ and $\tilde{b}_k$ belong to $\tif^{-k}(\Sigma(\tix))$. By our assumption, this implies that
\[
\alpha(\tif^k(\tilde{a}_k),\tif^k(\tilde{b}_k))=0,\:\:\:\textrm{for all}\:\:k\in\N.
\]
Since the derivative of $f$ is a linear isomorphism this implies that 
\[
\alpha(\tilde{a}_k,\tilde{b}_k)=0,\:\:\:\textrm{for all}\:\:k\in\N.
\]
By continuity of the assignment $(\tix,\tilde{y})\mapsto\alpha(\tix,\tilde{y})$ we deduce that 
\[
\alpha(\tilde{a},\tilde{b})=0.
\]
This proves that $\cN(\tilde{a})=\emptyset$, which implies that $f$ must be special and completes the proof.
\end{proof}

By pushing a bit further the above argument we can obtain an even stronger conclusion. 

\begin{lemma}
\label{l.contracaomaravilhosa}
Assume that $f$ is not special. Then, the set $\cN(\tix)$ is open and dense in the fiber $\Sigma(\tix)$, for every $\tix\in\Sigma(\tix).$
\end{lemma}
\begin{proof}
Openness is direct from the definition and the continuity of the function $(\tix,\tilde{y})\mapsto \alpha(\tix,\tilde{y})$. Assume the lemma is false. Then, for some $\tix\in\TTf$ the set $\cP(\tix)$ has non-empty interior. In particular, there exists a small cylinder $I\subset\cP(\tix)$. By backward expansion along fibers there must exist an integer $n>0$ such that 
\[
\Sigma(\tilde{f}^{-n}(\tix))\subset\tilde{f}^{-n}(I)\subset\tilde{f}^{-n}(\cP(\tix)).
\] 
By equation \eqref{eq:P-N-invartiant} we deduce that $\cP(\tilde{f}^{-n}(\tix))=\Sigma(\tilde{f}^{-n}(\tix))$, which contradicts Lemma~\ref{l.especialcontagia}.
\end{proof}

The above result implies the dichotomies presented in \cite{CostaMicena2022, MicenaTahzibi2016}, but for the strongly transitive case.
    
\subsection{The non-transversality set}

When $f$ is not special, the strongly transitive assumption implies that in every fiber there are infinitely many different unstable directions. In our case, this is made explicit by Lemma~\ref{l.contracaomaravilhosa}. We are now going to consider how the measure $\tm$ ``sees'' these different unstable directions. The natural strategy then is to consider the \emph{``bad''} set of points for which the measure do not ``see'' any transversality at all.  

\begin{definition}
The \emph{non-transversality set} is the set $\tilde{\mathbf{B}}=\{\tix\in\TTf : \tilde{\mu}_x(\cN(\tix))=0\}$.
\end{definition}

We have that $\tilde{\mathbf{B}}$ is $\tilde{f}$-invariant. Therefore, by ergodicity, $\tilde{\mu}(\tilde{\mathbf{B}})\in\{0,1\}$. Notice that $\tilde{\mathbf{B}}$ depends on the measure, although, for simplicity, our notation does not reflect that.

\begin{lemma}
	\label{l.fakenewslemma}
Assume that $f$ is not special and that $\tilde{\mu}(\tilde{\mathbf{B}})=1$. Then $\operatorname{supp}(\tilde{\mu})$ has empty interior. 
\end{lemma}

\begin{proof}
By Lemma~\ref{l.contracaomaravilhosa}, it suffices to show that for every $\tix\in\TTf
$ no point $\tilde{y}\in\cN(\tix)$ lies in $\operatorname{supp}(\tilde{\mu})$. Take $\tilde{y}\in\cN(\tix)$. By continuity, there exists a neighborhood $U$ of $x=\pi(\tix)$ such that for every $\tilde{z}\in \pi^{-1}(U)\cap\tilde{\cW}^{cu}(\tix)$ there exists a cylinder $I(\tilde{z})\subset\cN(\tilde{z})$. In particular, $\cup_{\tilde{z}\in \pi^{-1}(U)\cap\tilde{\cW}^{cu}(\tix)}I(\tilde{z})$ is a neighborhood of $\tilde{y}$. Since we assumed $\tilde{\mu}(\tilde{\mathbf{B}})=1$, we have that for almost every $\tilde{z}$ it holds $\tilde{\mu}_{\tilde{z}}(I(\tilde{z}))=0$. By Rokhlin's theorem we deduce that 
\[
\tilde{\mu}\left(\bigcup_{\tilde{z}\in \pi^{-1}(U)\cap\tilde{\cW}^{cu}(\tix)}I(\tilde{z})\right)=0,
\] 
which implies that $\tilde{y}\notin\operatorname{supp}(\tilde{\mu})$.
\end{proof}

\subsection{A 0 or 1 law}
\label{sec:0-ou-1}

Let us assume now that $\mu\in\cM^{\operatorname{erg}}_f(\TT)$ is a fully supported ergodic invariant measure. We deduce from Lemma~\ref{l.fakenewslemma} that for $\mu$-almost every $x\in\TT$ it holds $\tilde{\mu}_x(\cN(\tix))>0$. That is, under the full support assumption, the measure ``sees'' \emph{some} transversality at each fiber. However, in our setting this is a \emph{0 or 1 law}: once \emph{some} transversality is seen at almost every fiber we immediately have that the measure actually sees transversality \emph{almost everywhere} in almost every fiber. More precisely, the goal of this section is to prove the following.

\begin{lemma}
\label{lem:zeroum}
For $\mu$-almost every $x\in\TT$ it holds $\tilde{\mu}_x(\cN(\tix))=1$.
\end{lemma} 

In our case we can take advantage of the fractal structure of the fibers and give a point-density proof of the 0 or 1 law that is much simpler then the proof in \cite{ALOS}, which is based on a martingale convergence argument.

\begin{lemma}{\cite[Proposition 2.10]{miller2008existence}}
	\label{lem:miller}
Suppose that $(X,d)$ is a compact ultrametric space, $\mu$ is a probability measure on $X$, and $B$ is a Borel set. Then,
\[
\lim_{\varepsilon\to 0}\frac{\mu(B\cap B_{\varepsilon}(x))}{\mu(B_{\varepsilon}(x))}=1,
\]
for $\mu$ a.e. $x\in B$.
\end{lemma}

Recall that, by Remark~\ref{lem:ultramen}, each fiber $\Sigma(\tix)$ is a compact ultrametric space. In particular, the above lemma applies to the measure $\tm_x$. 

\subsubsection{Proof of Lemma~\ref{lem:zeroum}}
\label{sec:proof-0-ou-1}

Assume by contradiction the existence of a $\tilde{\mu}$ positive measure set $\tilde{A}$ such that $\tilde{\mu}_x(\cP(\tix))>0$ for every $\tix\in\tilde{A}$. Consider the measurable function 
\[
\tix\mapsto\tm_x(\cN(\tix)),
\] 
which is positive a.e. since $\tm(B)=0$. By Lusin's Theorem, for every $\delta>0$ there exists $c>0$ such that $\tm_x(\cN(\tix))\geq c$ if $\tix\in \tilde{K}$, for some compact set $\tilde{K}$ with $\tm_x(\tilde{K})>1-\delta$. Take $\delta$ small enough such that $\tm(\tilde{K}\cap \tilde{A})>0$ and take $\tix\in\tilde{K}\cap\tilde{A}.$ By Rokhlin's theorem and ergodicity, we can assume the existence of a point $\tilde{y}\in\cP(\tix)$ so that $\tilde{f}^{-n}(\tilde{y})\in\tilde{K}$ for infinitely many values of $n$. By Lemma~\ref{lem:miller} we can further assume that $\tilde{y}$ is a density point of $\cP(\tix)$ with respect to $\tm_x$. 

Observe that, in the metric given in Lemma~\ref{lem:miller}, balls are cylinders. Therefore, there exists some cylinder $\tilde{y}\in I(\tilde{y})\subset\Sigma(\tix)$  such that
\[
\frac{\tm_x(I(\tilde{y})\cap\cP(\tix))}{\tm_x(I(\tilde{y}))}>1-c.
\]
By equation \eqref{eq:P-N-invartiant}, we can take $n$ large enough such that the following properties hold:
\begin{itemize}
	\item $\tilde{f}^{-n}(I(\tilde{y}))\supseteq\Sigma(\tilde{y}_{-n})$;
	\item $\cP(\tilde{f}^{-n}(\tilde{y}))=\tilde{f}^{-n}(\cP(\tilde{y}) \cap I(\tilde{y}))\cap\Sigma(\tilde{y}_{-n})$;
	\item  $\tilde{y}_{-n}\eqdef\tif^{-n}(\tilde{y})\in\tilde{K}$.
\end{itemize}
The first two properties above imply that
\[
\tm_{\tilde{f}^{-n}(\tix)}(\cP(\tilde{y}_{-n})) = \frac{\tm_{\tilde{f}^{-n}(\tix)}(\tilde{f}^{-n}(I(\tilde{y})\cap\cP(\tix)))}{\tm_{\tilde{f}^{-n}(\tix)}(\tilde{f}^{-n}(I(\tilde{y})))} = \frac{\tm_x(I(\tilde{y}\cap\cP(\tix)))}{\tm_x(I(\tilde{y}))}>1-c,
\]
which, in turn, implies that $\tm_{\tilde{f}^{-n}(\tilde{y})}(\cN(\tilde{f}^{-n}(\tilde{y})))<c$. By the definition of $\tilde{K}$, this contradicts the third property above and completes the proof.\qed 

\subsection{An Egorov-type estimate of transversality}
We can now state our uniform transversality estimate along fibers. The proof is the same as the argument of Lemma 11.2 of \cite{ALOS}, so we refrain from repeating the details. 

\begin{proposition}
\label{prop:transversalidade}
Given a constant $\gamma>0$ and $K\subset\TT$ a compact set with $\mu(K)>1-\gamma$, there exist a constant $\beta=\beta(\gamma)$ and a measurable set $K'\subset K$, with $\tm(K')>1-2\sqrt{\gamma}$, such that for every $\tix\in K'$, there exists a point $\tilde{y}$ inside the fiber $\Sigma(\tix)$ so that $\alpha(\tix,\tilde{y})>\beta$. 
\end{proposition}

The above lemma is used in Section \ref{sec:matching} in the proof of Lemma \ref{lem:matching}, where we prove that we have ``abundance'' of the configurations we need to run our argument.

\section{Y-configurations: coupling and synchronization}
\label{sec:proof-main-lemma}

\subsection{Stopping times}
\label{sec:stopping-times}

Remember that the notation $\hat{\lambda}^c_{\tilde{x}}(n)$ stands for the Lyapunov norm of the n-th derivative restricted to the direction $E^c$, that is
$$\hat{\lambda}^c_{\tilde{x}}(n) \eqdef \Vert D\tif^n(\tix)|_{E^c} \Vert_{L,\tix_n}.$$

Consider $\tix \in \TTf$ and $\tix^u \in \w^u(\tix)$. Given a small number $\varepsilon>0$, we define the \emph{first stopping time} as the function $\tau:\Z\to\Z$ defined by
\begin{equation}
	\label{eq:tau}
 \tau(n)=\tau(\tix,\tix^u,n,\varepsilon) = \inf \left\{ k\in\N; \; \hat{\lambda}^c_{\tix^u}(k) \frac{\lambda^c_{\tix_{-n}}(n)}{\lambda^u_{\tix_{-n}}(n)} \geq \varepsilon \right\}. 	
\end{equation}
Given the value of $\tau(n)=\tau(\tix,\tix^u,n,\varepsilon)$ we define
\begin{equation}
    \label{eq:t}
t(n)=t(\tix,\tix^u,n,\varepsilon)=\inf \left\{k\in\N;\frac{\hat{\lambda}^c_{\tix}(n)}{\hat{\lambda}^c_{\tix^u}(\tau(n))} \geq 1 \right\}.    
\end{equation}
The function $t:\Z\to\Z$ is called the \emph{second stopping time}. We refer to the pair $(,\tau,t)$ as the \emph{stopping times}. These functions can be computed for $\tilde{\mu}$-almost every $\tix$ and for all $\tix^u \in \w^u(\tix)$ by Lemma \ref{lem:saturalyapunov}.

\begin{remark}
	\label{rem:black}
It follows from Remark~\ref{rem:lyapunov} that $\tau=\tau(\tix,\tix^u,n,\eps)$ and $t=t(\tix,\tix^u,\ell,\eps)$ it holds 
\[
\hat{\lambda}^c_{\tix^u}(\tau) \frac{\lambda^c_{\tix_{-n}}(n)}{\lambda^u_{\tix_{-n}}(n)}\asymp_{C}\eps,
\]
and also
\[
\hat{\lambda}^c_{\tix}(t)\asymp_{C}\hat{\lambda}^c_{\tix^u}(\tau),
\]
for a constant $C=C(f,\mu)$ depending only on the map $f$ and the measure $\mu$.
\end{remark}

These functions satisfy the following quasi-isometric estimate.

\begin{lemma}
\label{lem:quasiisometric}
There are constants $A >0$, $\Theta > 1$ (depending only on $(f,\mu)$) such that, for $\tilde{\mu}$-almost every $\tix \in \TTf$ and for all $\tix^u \in \w^u(\tix)$, and $m,n \in \N$, we have
\begin{itemize}
	\item $\Theta^{-1}|m - n|-A\leq|\tau(m)-\tau(n)|\leq \Theta|m - n|+A$,
	\item $\Theta^{-1}|m - n|-A\leq|t(m)-t(n)|\leq \Theta|m - n|+A$.
\end{itemize}
\end{lemma}

\begin{proof}
    For a $\tilde{\mu}$-regular point $\tix \in \TTf$, $\tix^u \in \w^u(\tix)$ and $\varepsilon > 0$, consider $\tau(n) = \tau(\tix, \tix^u, n, \varepsilon)$ and $t(n) = t(\tix, \tix^u, n, \varepsilon)$. From \eqref{eq:tau} we have

    $$\hat{\lambda}^c_{\tix^u}(\tau(n)) \frac{\lambda^c_{\tix_{-n}}(n)}{\lambda^u_{\tix_{-n}}(n)} \geq \varepsilon \mbox{, and } \hat{\lambda}^c_{\tix^u}(\tau(n)-1) \frac{\lambda^c_{\tix_{-n}}(n)}{\lambda^u_{\tix_{-n}}(n)} < \varepsilon.$$

    Using equation \eqref{eq:cocycle} and Lemma \ref{lem:crescelyapunovcresce}, the above inequalities imply that
    \begin{equation}
        \label{eq:tau-bounded1}
        \varepsilon \leq \hat{\lambda}^c_{\tix^u}(\tau(n)) \frac{\lambda^c_{\tix_{-n}}(n)}{\lambda^u_{\tix_{-n}}(n)} = \hat{\lambda}^c_{\tilde{f}^{\tau(n) - 1}(\tix^u)}(1) \hat{\lambda}^c_{\tix^u}(\tau(n) - 1) \frac{\lambda^c_{\tix_{-n}}(n)}{\lambda^u_{\tix_{-n}}(n)} < e^\sigma \varepsilon,
    \end{equation}
    for all $n \in \mathbb{N}$.

    By compactness and Lemma~\ref{lem:defideph}, one easily deduces that there are $\chi_1^d < \chi_2^d < 0$ such that $$e^{\chi_1^dn} < d^n_{\tix} \eqdef \frac{\lambda^c_{\tix_{-n}}(n)}{\lambda^u_{\tix_{-n}}(n)} < e^{\chi_2^dn}.$$
    
    By equation \eqref{eq:tau-bounded1} we have, given $m \in \mathbb{N}$, $m \geq n$, that
    $$\varepsilon e^{\chi_1^d(m-n)} \leq \hat{\lambda}^c_{\tix^u}(\tau(n)) \underbrace{d^{m-n}_{\tilde{f}^{n}(\tix)}d^{n}_{\tix}}_{d^{m}_{\tix}} < e^{\chi_2^d(m-n)} e^\sigma \varepsilon.$$
    
    Thus,
    \begin{equation}
        \label{eq:tau-bounded2}
        \varepsilon e^{\chi_1^d(m-n)} e^{\lambda(\tau(m) - \tau(n))} \leq \hat{\lambda}^c_{\tix^u}(\tau(m)) d^{n+m}_{\tix} < e^{\chi_2^d(m-n)} e^{\sigma(\tau(m) - \tau(n) + 1)} \varepsilon
    \end{equation}
    is given by using Lemma \ref{lem:crescelyapunovcresce} and the fact that
    $$\hat{\lambda}^c_{\tilde{f}^{\tau(n)}(\tix^u)}(\tau(m) - \tau(n)) = \dfrac{\hat{\lambda}^c_{\tix^u}(\tau(m))}{\hat{\lambda}^c_{\tix^u}(\tau(n))}.$$

    The left hand inequality on equation \eqref{eq:tau-bounded1} and the right hand side of equation \eqref{eq:tau-bounded2} imply that
    $$\varepsilon \leq e^{\chi_2^d(m-n)} e^{\sigma(\tau(m) - \tau(n) + 1)} \varepsilon,$$
    and, by using the remaining inequalities on said equations, we have also
    $$\varepsilon e^{\chi_1^d(m-n)} e^{\lambda(\tau(m) - \tau(n))} < e^\sigma \varepsilon.$$
    We may assume without loss of generality that $\eps$ is small enough so that $e^{\sigma}\eps<1$. Then, by taking the logarithm on the last two inequalities, we have
    \begin{align*}
        \chi_1^d(m-n) + \lambda(\tau(m) - \tau(n) - 1) &< 0 <  \chi_2^d(m-n) + \sigma(\tau(m) - \tau(n) + 1)\\
        \implies \dfrac{\chi_1^d}{\lambda} (m-n) - 1 &< - \tau(m) + \tau(n) < \dfrac{\chi_2^d}{\sigma} (m-n) + 1,
    \end{align*}
    which gives us the bound for $\tau$.

    A similar argument gives us the bound for $t$. Indeed, the definition of $t(n)$ implies that
    \begin{equation}
        \label{eq:t-bounded}
        1 \leq \dfrac{\hat{\lambda}^c_{\tix}(t(n))}{\hat{\lambda}^c_{\tix^u}(\tau(n))} < e^\sigma.
    \end{equation}
    for all $n \in \mathbb{N}$.
    
    Lemma \ref{lem:crescelyapunovcresce} gives us that
    $$e^{\lambda (t(m) - t(n)) - \sigma (\tau(m) - \tau(n))} < \dfrac{\hat{\lambda}^c_{\tilde{f}^{t(n)}(\tix)}(t(m) - t(n))}{\hat{\lambda}^c_{\tilde{f}^{\tau(n)}(\tix^u)}(\tau(m) - \tau(n))} < e^{\sigma (t(m) - t(n)) - \lambda (\tau(m) - \tau(n))},$$
    which in turn implies
    $$e^{\lambda (t(m) - t(n)) - \sigma (\tau(m) - \tau(n))} < \dfrac{\hat{\lambda}^c_{\tix}(t(m))}{\hat{\lambda}^c_{\tix^u}(\tau(m))} < e^{\sigma (t(m) - t(n) + 1) - \lambda (\tau(m) - \tau(n))},$$
    by equations \eqref{eq:cocycle} and \eqref{eq:t-bounded}.

    The inequalities from both the above equation and \eqref{eq:t-bounded} imply that
    $$e^{\lambda (t(m) - t(n)) - \sigma (\tau(m) - \tau(n) + 1)} < 1 < e^{\sigma (t(m) - t(n) + 1) - \lambda (\tau(m) - \tau(n))},$$
    and then, by taking the logarithm, we have
    \begin{align*}
        \lambda (t(m) - t(n)) - \sigma (\tau(m) - \tau(n) + 1) &< 0 < \sigma (t(m) - t(n) + 1) - \lambda (\tau(m) - \tau(n))\\
        \implies - \dfrac{\sigma}{\lambda} (\tau(m) - \tau(n) + 1) &< - t(m) + t(n) <  1 - \dfrac{\lambda}{\sigma} (\tau(m) - \tau(n)),
    \end{align*}
    which, using the bound for $\tau$, gives us the bound for $t$.
\end{proof}

Lemma~\ref{lem:quasiisometric} is a key point where our argument deviates from \cite{ALOS}. Here, we need to work with Lyapunov norms, so that the quasi-isometric estimates for stopping times \emph{do not depend} on a Lusin set. This is a crucial ingredient that we use in the Fubini-like argument of Eskin--Lindenstrauss (Lemma \ref{lem:claim64}).  

\subsection{Y-configurations}
\label{sec:Y-config}

We now introduce a dynamical configuration that will be a central tool in our argument. 

\begin{definition}[$Y$-configurations]
Given points $\tix\in\TTf$, $\tix^u\in\w^u_1(\tix)$ and numerical parameters $\ell\in\N$, $\eps>0$ and $T>1$, a \textit{$(\ell,T)$ $Y$-configuration} is the data of five points $(\tix_{-\ell},\tix,\tix^u,\tix^u_m,\tix_{\hat{m}})$, with the condition that the integers $m,\hat{m}$ satisfy
\[
|m-\tau(\tix,\tix^u,\ell,\eps)|<T\:\:\:\textrm{and}\:\:\:|\hat{m}-t(\tix,\tix^u,\ell,\eps)|<T.
\]
\end{definition}

The integers $m,\hat{m}$ are to be thought as the two upper (time) lengths of the configuration, while $\ell$ is the lower (time) length. Ideally, one would like have
$m=\tau(\tix,\tix^u,\ell,\eps)$ and $\hat{m}=t(\tix,\tix^u,\ell,\eps)$. However, to fully accomplish Eskin--Mirzakhani's strategy in our context we need the matching argument, because given two configurations the stopping times need not to be equal. To overcome this difficulty we use the synchronization idea of Eskin--Lindenstrauss: up to a time length bounded by $T>0$, we can find $m,\hat{m}$ such that both configurations are in \emph{good matching} (see Lemma~\ref{lem:matching}).

To emphasize the most relevant data in the definition, we denote a $Y$-configuration by $Y(\tix,\tix^u,\ell,\eps)$.

\begin{figure}[ht]
    \begin{tikzpicture}
    
    \coordinate (x) at (-1.5,.18); 
    \coordinate (xu) at (1.5,-.21);
    \coordinate (x-l) at (-3,-2);
    \coordinate (xm) at (-2,2); 
    \coordinate (xum) at (2,2);
    
    \draw[gray!80, thick,->,>=stealth, densely dotted] (x) -- (xm) node[pos=0.5,left]{$\tilde{f}^{\hat{m}}$};
    \draw[gray!80, thick,->,>=stealth, densely dotted] (xu) -- (xum) node[pos=0.5,right]{$\tilde{f}^{m}$};
    \draw[gray!80, thick,->,>=stealth, densely dotted] (x-l) -- (x) node[pos=0.5,right]{$\tilde{f}^{\ell}$};;
    \draw[red!80!black, thick] (-3,0) .. controls (-1.5,.25) and (-1.5,.25) .. (0,0) .. controls (1.5,-.25) and (1.5,-.25) .. (3,0);
    
    \filldraw (x) circle (1pt) node[below]{$\tix$};
    \filldraw (xu) circle (1pt) node[below]{$\tix^u$};
    \filldraw (x-l) circle (1pt) node[right]{$\tix_{-\ell}$};
    \filldraw (xm) circle (1pt) node[right]{$\tix_{\hat{m}}$};
    \filldraw (xum) circle (1pt) node[left]{$\tix^u_{m}$};
    \draw (3,0) node[right]{$\w^u_{1}(\tix)$};
    
    \end{tikzpicture}
    \caption{\label{fig:yconfig} The five points in the $Y$ configuration and the dynamical relation between them.}
\end{figure}
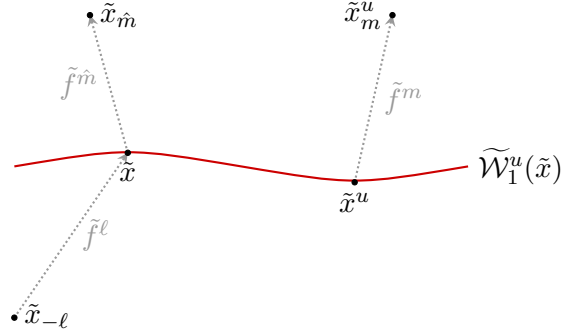

\begin{definition}[Quadrilaterals]
 Given $\beta>1$ and $\ell\in\N$, we say that four points $\tix,\tix^u,\tilde{y},\tilde{y}^u$ generate a $(\beta,\ell)$-quadrilateral if $\tix_{-\ell}\in\Sigma(\tilde{y}_{-\ell})$, $\tix^u\in\w^u(\tix)$ and, moreover:
\begin{enumerate}
    \item $\alpha(\tix_{-\ell},\tilde{y}_{-\ell})>\beta^{-1}$,
    \item $\beta^{-1}<d^u(\tix,\tix^u)<1$,
    \item $\tilde{y}^u=H^{cs}_{\tix,\tilde{y}}(\tix^u)$.
\end{enumerate}
\end{definition}

\subsubsection{Center deviation along coupled configurations}

When the points $\tix,\tix^u,\tilde{y},\tilde{y}^u$ generate a quadrilateral, we say that the associated configurations $Y(\tix,\tix^u,\ell,\eps)$ and $Y(\tilde{y},\tilde{y}^u,\ell,\eps)$ are \emph{coupled}. This notion is devised so that one has a very precise control (up to uniform constants) on the displacement along central leaves after iteration by the stopping times. 

The result below follows a similar strategy as Lemma 9.2 in \cite{ALOS}, but we need to be much more careful in our argument due to the lack of uniform expansion on the center. Notice that we also need to require more assumptions regarding which points belong to $\cL$.  In particular, a subtle point is that we discuss a property of configurations whose upper time lengths $m$, $\hat{m}$ are the first and second stopping times, respectively.

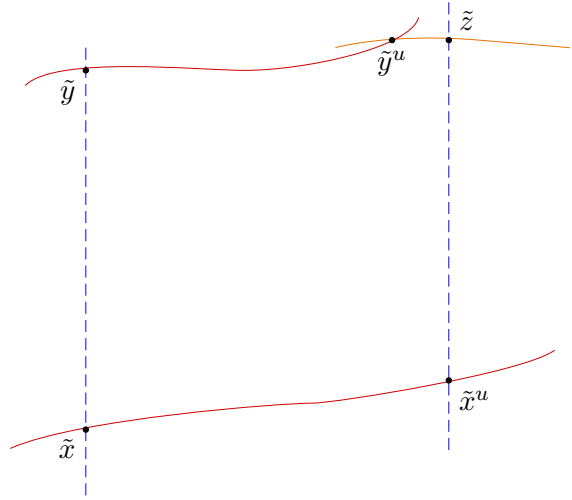
\begin{figure}[ht]
    \centering
    \begin{tikzpicture}[y=1cm, x=1cm]
    
    \path[draw=blue,line cap=butt,line join=miter,line width=0.0cm,miter limit=4.0,dash pattern=on 0.16cm off 0.08cm] (2.2, 8.4) -- (2.2, 2.4);
    \path[draw=blue,line cap=butt,line join=miter,line width=0.0cm,miter limit=4.0,dash pattern=on 0.16cm off 0.08cm] (7.0, 9.0) -- (7.0, 3.0);
    \path[draw=orange!90!black,line cap=butt,line join=miter,line width=0.0cm] (5.5, 8.4)..  controls (5.5, 8.4) and (6.2, 8.6) .. (7.4, 8.5).. controls (8.6, 8.4) and (8.6, 8.4) .. (8.6, 8.4);
    \path[draw=red!80!black,line cap=butt,line join=miter,line width=0.0cm] (1.2, 3.1).. controls (2.1, 3.5) and (4.8, 3.7) .. (5.2, 3.7).. controls (5.5, 3.7) and (8.0, 4.1) .. (8.4, 4.4);
    \path[draw=red!80!black,line cap=butt,line join=miter,line width=0.0cm] (1.4, 7.9).. controls (1.8, 8.3) and (3.7, 8.1) .. (4.3, 8.1).. controls (5.1, 8.1) and (6.5, 8.4) .. (6.6, 8.8);
    
    \filldraw (2.2,8.1) circle (1pt) node[below left]{$\tilde{y}$};
    \filldraw (6.25,8.5) circle (1pt) node[below]{$\tilde{y}^u$};
    \filldraw (2.2,3.35) circle (1pt) node[below left]{$\tilde{x}$};
    \filldraw (7,4) circle (1pt) node[below right]{$\tilde{x}^u$};
    \filldraw (7,8.5) circle (1pt) node[above right]{$\tilde{z}$};
  
    \end{tikzpicture}

    \caption{Coupled configurations on the inverse limit space $\TTf$.}
    \label{fig:acopladas2}
\end{figure}

\begin{lemma}
    \label{lem:driftestimate}
Given $\cL\subset\TTf$ a compact Lusin set, $\beta>1$ there exist $\hat{\beta}=\hat{\beta}(\beta,\cL)>1$ and $\ell_0=\ell_0(\cL,\beta,f)$ such that, if $\ell\geq\ell_0(\cL)$ and $\eps>0$ are such that there exists some length $\ell$ $Y$-configuration 
$Y(\tix,\tix^u,\ell,\eps)$ which is coupled with another $Y$-configuration $Y(\tilde{y},\tilde{y}^u,\ell,\eps)$ so that $\tix,\tix^u,\tilde{y},\tilde{y}^u$ generate a $(\beta,\ell)$ 
quadrilateral and such that $\tix^u,\tix^u_\tau\in\cL$, where $\tau=\tau(\tix,\tix^u,\ell,\eps)$ is the first stopping time, then 
\[
\eps\hat{\beta}^{-1}\leq|\cR^c_{\tix^u_\tau}(\tilde{y}^u_\tau)|\leq\eps\hat{\beta}.
\]
\end{lemma}
\begin{proof}
From the definition of a quadrilateral we know that $d^u(\tix,\tix^u)<1$. By the distortion estimate (Lemma \ref{lem:distodynball}), if we denote by $J$ the segment contained in $\w^u(\tix)$ with end points $\tix$ and $\tix^u$, we have, for every $\tilde{z}\in J$, that
\begin{equation}
    \label{eq:distunstable}
    \lambda^u_{\tilde{z}_{-\ell}}(j)\asymp_C \lambda^u_{\tilde{x}_{-\ell}}(j),
\end{equation}
for every $j=1,...,\ell$, where $C=C(f)>1$ is a constant depending only on $f$. One deduces that $d^u(\tix_{-\ell},\tix^u_{-\ell})\asymp_C\left(\lambda^u_{\tilde{x}_{-\ell}}(\ell)\right)^{-1}$. Thus, the unstable segment $f^{-\ell}(J)$ has length $\asymp_C\left(\lambda^u_{\tilde{x}_{-\ell}}(\ell)\right)^{-1}$. Since $\alpha(\tix_{-\ell},\tilde{y}_{-\ell})>\beta^{-1}$, a simple geometrical argument, using the continuity of the unstable bundle $E^u$ on the space $\TTf$ yields $d^c(\tix^u_{-\ell},\tilde{y}^u_{-\ell})\asymp_C\left(\lambda^u_{\tilde{x}_{-\ell}}(\ell)\right)^{-1}$, for every $\ell\geq\ell_0(\beta)$ where the new constant $C=C(f,\beta)$ now depends on $\beta$. 

We claim that 
\[
|\cR^c_{\tix^u}(\tilde{y}^u)|\asymp_C\frac{\lambda^c_{\tix_{-\ell}}(\ell)}{\lambda^u_{\tilde{x}_{-\ell}}(\ell)},
\]
for some constant $C=C(\cL,f,\beta)$. To prove this claim, first notice that the segment of center manifold joining the points $\tix^u$ and $\tilde{y}^u$ is contained inside the local Pesin manifold $\tW^c_{\loc}(\tix^u)$. Indeed, from the distortion Lemma~\ref{lem:bounded-dist} we deduce that 
\[
\operatorname{length}\left(\tif^{-\ell}(\tW^c_{\loc}(\tix^u))\right)\asymp_C\lambda^c_{\tix^u_{-\ell}}(\ell)^{-1}\asymp_C\lambda^c_{\tix_{-\ell}}(\ell)^{-1},
\]
for some constant $C=C(\cL,\tif,\beta)$, depending also on $\cL$. Since $d^c(\tix^u_{-\ell},\tilde{y}^u_{-\ell})\asymp_C\left(\lambda^u_{\tilde{x}_{-\ell}}(\ell)\right)^{-1}$, the domination of the partial hyperbolic splitting implies that $d^c(\tix^u_{-\ell},\tilde{y}^u_{-\ell})$ is exponentially smaller than the length of $\tif^{-\ell}(\tW^c_{\loc}(\tix^u))$. Therefore, for $\ell$ sufficiently large (depending on the constant $C(\cL,\beta,f))$, we have that the entire center segment with endpoints $\tix^u$ and $\tilde{y}^u$ is contained in the local Pesin manifold $\tW^c_{\loc}(\tix^u)$. The distortion Lemma~\ref{lem:bounded-dist} then yields
\[
d^c(\tix^u,\tilde{y}^u)\asymp_C\lambda^c_{\tix^u_{-\ell}}(\ell)d^c(\tix^u_{-\ell},\tilde{y}^u_{-\ell})\asymp_C\lambda^c_{\tix{-\ell}}(\ell)\left(\lambda^u_{\tilde{x}_{-\ell}}(\ell)\right)^{-1}.
\]
Now the claim follows, for on the Lusin set $\cL$ the quantity $\cR^c_{\tix^u}(\tilde{y}^u)$ is uniformly comparable to $d^c(\tix^u,\tilde{y}^u)$.

Since the dynamics acts as a linear expansion on normal forms, this gives us that 
\[
|\cR^c_{\tix^u_\tau}(\tilde{y}^u_\tau)|\asymp_C\frac{\lambda^c_{\tix^u}(\tau)\times\lambda^c_{\tix_{-\ell}}(\ell)}{\lambda^u_{\tix_{-\ell}}(\ell)}.
\]
Since $\tix^u_\tau\in\cL$, we have that $\lambda^c_{\tix^u}(\tau)$ and $\hat{\lambda}^c_{\tix^u}(\tau)$ are comparable by a constant depending on $\cL$. The lemma now follows from Remark~\ref{rem:black}.
\end{proof}

Our argument above also yields the following distortion estimate. 

\begin{lemma}
    \label{lem:distquadrilateros}
Under the same assumptions as in Lemma~\ref{lem:driftestimate}, there exists a constant $C=C(\cL,\beta,f)>1$ such that, for every $\eps>0$ sufficiently small,
\begin{enumerate}
    \item $\dfrac{\hat{\lambda}^c_{\tix^u}(\tau)}{\hat{\lambda}^c_{\tilde{y}^u}(\tau)}\asymp_C 1$ and
    \item   $\dfrac{\lambda^c_{\tix^u}(\tau(\ell))}{\lambda^c_{\tilde{y}^u}(\tau(\ell))}\asymp_C 1$.
\end{enumerate}
\end{lemma}
\begin{proof}
    The first estimate follows since $\hat{\beta}\eps<\inf\{r(\tix);\tix\in\cL\}$, if $\eps$ is small enough, and one can apply Lemma~\ref{lem:distorlyp}. For the second estimate, we can change from the Lyapunov norm at time $m$ to the background norm, since they are comparable by a constant which depends on $\cL$, and the conclusion follows.
\end{proof}

In the next subsection we are going to prove a fundamental property about coupled configurations: their stopping times cannot differ too much. 

\subsection{Synchronization of stopping times}

The following result plays a substantial role in our application of the factorization method, for it allows us to use the same stopping time for both Y-configurations once they are coupled. Before stating the result we need a definition.

\begin{definition}
Given $Y(\tix,\tix^u,\eps,\ell)$ a $Y$-configuration, the \emph{unstable dynamical ball} associated with the configuration is the set 
\[
J(\tix^u)\eqdef\tif^{-\tau}(\w^u_1(\tix^u_{\tau})),
\]
where $\tau=\tau(\tix,\tix^u,\eps,\ell)$ is the stopping time. 
\end{definition}

\begin{proposition}[Synchronization]
    \label{prop:sync}
Given $\cL\subset\TTf$ a compact Lusin set and $\beta>1$, there exists $T=T(\cL,\beta,f,\mu)>1$ such that if $(\tix,\tix^u,\tilde{y},\tilde{y}^u)$ is a $(\beta,\ell)$-quadrilateral such that, $\tix,\tix^u,\tix^u_{\tau(\ell)},\tilde{y}\in\cL$, where $\tau(\ell)=\tau(\tix,\tix^u,\ell,\eps)$, then for every $\ta\in J(\tix^u)$ and every $\tb\in J(\tilde{y}^u)$ we have that 
\begin{align}
	\left|\tau(\tix,\ta,\eps,\ell)-\tau(\tilde{y},\tb,\eps,\ell)\right|&<T\nonumber \mbox{ and}\\
	\left|t(\tix,\ta,\eps,\ell)-t(\tilde{y},\tb,\eps,\ell)\right|&<T\nonumber.
\end{align}	
\end{proposition}

Although the proof has some technicalities the idea is rather simple: at one hand, by Lemma~\ref{lem:distquadrilateros} we have a bounded distortion estimate for the center derivative \emph{at the first stopping time $\tau(\tix,\tix^u,\ell,\eps)$}. On the other hand, the exponentially small quantities $d^{\ell}_{\tix}$ and $d^{\ell}_{\tilde{y}}$ are comparable since they come from a Hölder cocycle and $\tix$ and $\tilde{y}$ lie on the same fiber (hence become exponentially close in the future), by Lemma~\ref{lem:bounded-dist}. Since the first stopping time is the time required to increase from $d^{\ell}_{\tix}$ to $\eps$ and since the second stopping time is the time required to increase from $\hat{\lambda}^c_{\tix}$ to $\hat{\lambda}^c_{\tix}(\tau)$, Proposition~\ref{prop:sync} must hold.    

\subsubsection{Proof of Proposition~\ref{prop:sync}}
We begin by comparing the stopping times at $\tix^u$ and $\tilde{y}^u$. This is the key point in our implementation of the factorization method: we need only to assume good properties (continuity inside the Lusin set) for the $Y$-configuration on $\tix$. In our case, the idea is that we are able to ``transfer'' some of these properties to the $Y$-configuration on $\tilde{y}$.  

\begin{lemma}
\label{lem:fatorizacao}
Given $\cL\subset\TTf$ a compact Lusin set and $\beta>1$, there exists $T_0=T_0(\cL,\beta,f,\mu)>1$ such that, if $(\tix,\tix^u,\tilde{y},\tilde{y}^u)$ is a $(\beta,\ell)$-quadrilateral satisfying that $\tix,\tix^u,\tix^u_{\tau(\ell)}\in\cL$, where $\tau(\ell)=\tau(\tix,\tix^u,\ell,\eps)$, then we have that 
\begin{align}
    \left|\tau(\tix,\tix^u,\eps,\ell)-\tau(\tilde{y},\tilde{y}^u,\eps,\ell)\right|&<T_0\nonumber \mbox{ and}\\
    \left|t(\tix,\tix^u,\eps,\ell)-t(\tilde{y},\tilde{y}^u,\eps,\ell)\right|&<T_0\nonumber.
\end{align}
\end{lemma}
\begin{proof}
Let us denote  $\tau=\tau(\tix,\tix^u,\eps,\ell)$ and $\hat{\tau}=\tau(\tilde{y},\tilde{y}^u,\eps,\ell)$. Assume first that $\tau>\hat{\tau}$. 
Then, from the definition we get
\[
d^\ell_{\tilde{y}}\hat{\lambda}^c_{\tilde{y}^u}(\hat{\tau})\geq\eps.
\]
On the other hand, we can write
$$
    \frac{d^\ell_{\tix}\hat{\lambda}^c_{\tix^u}(\hat{\tau}+k)}{d^\ell_{\tilde{y}}\hat{\lambda}^c_{\tilde{y}^u}(\hat{\tau})} =\frac{d^\ell_{\tix}}{d^\ell_{\tilde{y}}}\times\frac{\hat{\lambda}^c_{\tix^u}(\hat{\tau})}{\hat{\lambda}^c_{\tilde{y}^u}(\hat{\tau})}\times\hat{\lambda}^c_{\tix^u_{\hat{\tau}}}(k)\nonumber \geq Ce^{k\lambda}\nonumber,
$$
for some constant $C=C(\cL,\beta,f)$. Indeed, for the first factor on the right-hand side of the equality we can apply the distortion lemma for the Hölder cocycle $\ta\mapsto d^{\ell}_{\ta}$ and points on the same fiber and lower bound it by a constant depending only on $f$. The second factor is bounded by a constant depending on $\cL$ and $\beta$ due to Lemma~\ref{lem:distquadrilateros}. The third factor is lower bounded by the exponential due to Lemma~\ref{lem:crescelyapunovcresce}.  

Combining these two estimates we deduce that, whenever $k$ is large enough so that $Ce^{k\lambda}\geq 1$, we are forced to have ${d^\ell_{\tix}\hat{\lambda}^c_{\tix^u}(\hat{\tau}+k)}\geq\eps$, which proves that $\tau$ cannot be larger than $\hat{\tau}+k$.   

Observe that, in the above argument, the roles of $\tix$, $\tilde{y}$ and $\tix^u$, $\tilde{y}^u$
are reversible. In fact, as we assumed that $\tilde{y}\in\cL$, we have that $\tilde{y}$ is, in particular, a $\mu$-generic point and so is every $\tilde{y}^u\in\w^u(\tilde{y})$.  Thus, Lemma~\ref{lem:crescelyapunovcresce} applies to $\tilde{y}^u$ as well. As a consequence, we can perform the same argument if we assume $\hat{\tau}>\tau$, and thus $|\tau-\hat{\tau}|$ is bounded by a constant depending on $\cL,\beta,f$. 

We now treat the function $t$. We define $t$ and $\hat{t}$ in a similar fashion. Assume $t>\hat{t}$. We have, from the definition, that
\[
\frac{\hat{\lambda}^c_{\tilde{y}^u}(\hat{\tau})}{\hat{\lambda}^c_{\tilde{y}}(\hat{t})}\geq 1.
\]
On the other hand, we can bound
\begin{align}
    \frac{\hat{\lambda}^c_{\tilde{x}}(\hat{t}+k)}{\hat{\lambda}^c_{\tilde{x}^u}(\tau)}\times\frac{\hat{\lambda}^c_{\tilde{y}^u}(\hat{\tau})}{\hat{\lambda}^c_{\tilde{y}}(\hat{t})}&=\frac{\hat{\lambda}^c_{\tilde{y}^u}(\tau)}{\hat{\lambda}^c_{\tilde{x}^u}(\tau)}\times\frac{\hat{\lambda}^c_{\tilde{x}}(\hat{t})}{\hat{\lambda}^c_{\tilde{y}}(\hat{t})}\times\hat{\lambda}^c_{\tilde{y}^u}(\hat{\tau}-\tau)\times\hat{\lambda}^c_{\tix_{\hat{t}}}(k)\nonumber\\
    &\geq Ce^{k\lambda}\nonumber,
\end{align}
for some constant $C=C(\cL,\beta,f)$. Indeed, the first two factors are bounded by the distortion estimates for the Lyapunov norm, while for the third factor we consider the lower bound in Lemma~\ref{lem:lyapunovriemmann} and the fact that $|\hat{\tau}-\tau|$ is bounded by a constant depending on $\cL,\beta,f$. The case $\hat{t}>t$ is treated in the same way as before.
\end{proof}

We now consider the oscillations of the first stopping time function inside dynamical balls. More precisely, we have the following

\begin{lemma}
	\label{lem:syncdynball}
There exists a constant $C=C(f,\mu)$ depending only on the map $f$ and the measure $\mu$ such that for $\tm$ almost every $\tix\in\TTf$, for every $\tix^u\in\widetilde{\cW}^u_2(\tix)$, given $\ta\in J(\tix^u)$ it holds
\[
\left|\tau(\tix,\tix^u,\eps,\ell)-\tau(\tix,\ta,\eps,\ell)\right|\leq C.
\]
\end{lemma}
\begin{proof}
We first describe the choice of the constant $C=C(f,\mu)$. Consider the constant $0<\lambda<L^c(\mu)$, used in Lemma~\ref{lem:crescelyapunovcresce} and in Remark~\ref{rem:lyapunov}. Given this choice (which is already fixed) we consider $C_1(f,\mu)$ the constant given in Remark~\ref{rem:lyapunov}. Now, let $C_2=C_2(f,\mu)$ be large enough so that if $k>C_2(f,\mu)$ then
\[
e^{k\lambda}C_1(f,\mu)^{-1}>1.
\]
We shall prove that the Lemma holds with $C=C_2$. Indeed, let us denote $\tau\eqdef\tau(\tix,\tix^u,\eps,\ell)$ and $\tau_a\eqdef\tau(\tix,\ta,\eps,\ell)$. 

Assume that $\tau_a>\tau$. Then, by Remarks~\ref{rem:lyapunov} and \ref{rem:black} if $k>C_2(f\mu)$ then 
\begin{align*}
	d^{\ell}_{\tix}\hat{\lambda}^c_{\ta}(\tau+k)&=d^{\ell}_{\tix}\times\frac{\hat{\lambda}^c_{\ta}(\tau)}{\hat{\lambda}^c_{\tix^u}(\tau)}\times\hat{\lambda}^c_{\ta_{\tau}}(k)\times\hat{\lambda}^c_{\tix^u}(\tau)\\
	&\geq C_1(f,\mu)^{-1}\eps\hat{\lambda}^c_{\ta_{\tau}}(k)\\
	&\geq C_1(f,\mu)^{-1}\eps e^{k\lambda}\\
	&>\eps.
\end{align*}
This proves that $\tau_a-\tau\leq C_2(f,\mu)$. If $\tau>\tau_a$ one reverses the roles of $\tix^u$ and $\ta$ in the above argument to conclude that $\tau-\tau_a\leq C_2(f,\mu)$. This establishes our claim and completes the proof.  
\end{proof}

\begin{proof}[Proof of Proposition~\ref{prop:sync}]
First notice that it suffices to give the proof for the first stopping time for the estimate on the second stopping times follows using Remark~\ref{rem:black}. Now, we apply Lemmas~\ref{lem:syncdynball} and \ref{lem:fatorizacao} obtaining
\begin{align*}
\left|\tau(\tix,\ta,\eps,\ell)-\tau(\tix,\tb,\eps,\ell)\right|&\leq\left|\tau(\tix,\ta,\eps,\ell)-\tau(\tix,\tix^u,\eps,\ell)\right|+\left|\tau(\tix,\tix^u,\eps,\ell)-\tau(\tilde{y},\tilde{y}^u,\eps,\ell)\right|\\
&+\left|\tau(\tilde{y},\tilde{y}^u,\eps,\ell)-\tau(\tix,\tb,\eps,\ell)\right|\\
&\leq 2C(f,\mu)+T_0(\cL,\beta),
\end{align*}
concluding.
 \end{proof}

\subsection{The geometry of unstable dynamical balls}

We now combine several applications of the fundamental distortion Lemma~\ref{lem:bounded-dist}

\begin{lemma}
	\label{lem:distormagic}
	Given $\cL\subset\TTf$ a compact Lusin set and $\beta>1$, there exists a constant $L=L(\cL,\beta,T)>1$ such that, for every $\eps>0$ small enough, whenever $\ell\in\N$ is large enough and $\eps$ is small enough, if $(\tix,\tix^u,\tilde{y},\tilde{y}^u)$ is a $(\beta,\ell)$ quadrilateral such that $\tix,\tix^u,\tix^u_{\tau(\ell)}\in\cL$, we have for every $a\in J(\tix^u)$ and $b\in J(\tilde{y}^u)$ that
	\[
	\lambda^{*}_a(j)\asymp_L{\lambda^{*}_b(j)},
	\]
	for every $0\leq j\leq \max\{\tau(\tix,\tix^u,\eps,\ell),\tau(\tilde{y},\tilde{y}^u,\eps,\ell)\}$ and $* \in \{c,u\}$.
\end{lemma}
\begin{proof}
Consider the point $\tilde{z}^u\in\Sigma(\tix^u)$ such that $\tilde{y}^u$ and $\tilde{z}^u$ belong to the same center manifold on $\TTf$. From Lemma~\ref{lem:distquadrilateros} it follows that for $\eps>0$ small enough (depending on the Lusin set $\cL$) we have that $y^u\in W^c_{\loc}(x^u)$. Therefore, in the equality
\[
\frac{\lambda^*_{\tb}(j)}{\lambda^*_{\ta}(j)}=\frac{\lambda^*_{\tb}(j)}{\lambda^*_{\tilde{y}^u}(j)}\times\frac{\lambda^*_{\tilde{y}^u}(j)}{\lambda^*_{\tilde{z}^u}(j)}\times\frac{\lambda^*_{\tilde{z}^u}(j)}{\lambda^*_{\tilde{x}^u}(j)}\times\frac{\lambda^*_{\tilde{x}^u}(j)}{\lambda^*_{\tilde{a}}(j)}
\] 
we can bound the second term by a constant $C(\cL)$, due to Lemma~\ref{lem:bounded-dist}, while the third term is bounded by a constant $C(f)$, also by Lemma~\ref{lem:bounded-dist}. By definition of the stopping time, by Lemmas~\ref{lem:distor-unstable},\ref{lem:distorlusin} and Lemma~\ref{lem:syncdynball}  the first and the fourth term are bounded by a constant $C(\cL)$. 
\end{proof}

Lemma~\ref{lem:distormagic}, together with the distortion estimates given on Lemma \ref{lem:distor-unstable}, implies the following result. The proof follows as in \cite[Corollary 10.6]{ALOS}.

\begin{corollary}
    \label{cor:tamanhocomparavel}
    Given $\cL\subset\TTf$ a compact Lusin set and $\beta>1$ there exists a constant $L=L(\cL,\beta,T)>1$ such that whenever $\ell\in\N$ is large enough and $(\tix,\tix^u,\tilde{y},\tilde{y}^u)$ is a $(\beta,\ell)$ quadrilateral such that $\tix,\tix^u,\tix^u_{\tau(\ell)}\in\cL$ we have that 
    \[
    \frac{\left|f^j(J(\tix^u))\right|}{\left|f^j(J(\tilde{y}^u))\right|}\asymp_L 1,
    \]
    for every $0\leq j\leq \max\{\tau(\tix,\tix^u,\eps,\ell),\tau(\tilde{y},\tilde{y}^u,\eps,\ell)\}+T$.
\end{corollary}

\section{Y-configurations: the matching argument}
\label{sec:matching}

In this section we are going to prove the main technical result needed to establish Theorem~\ref{teo:main.rigidez}.  The goal is to find pairs of coupled $Y$ configurations $Y(\tix,\tix^u,\ell,\eps)$ and $Y(\tilde{y},\tilde{y}^u,\ell,\eps)$ such that both dynamical balls $J(\tix^u)$ and $J(\tilde{y}^u)$ intersect some Lusin set. 

The main idea of this section (which comes from \cite{EskinMirzakhani}) is to transfer information from the the first Y-configuration $Y(\tix,\tix^u,\eps,\ell)$ to the second one, $Y(\tilde{y},\tilde{y}^u,\eps,\ell)$, so that we have more flexibility to find ``good points'' in $J(\tilde{y}^u)$. 

\subsection{The Fubini argument of Eskin--Lindenstrauss}

Consider $\delta>0$ and $K\subset\TTf$ a compact subset of measure $\mu(K)>1-\delta$.  The result below is a simple consequence of Birkhoff's ergodic theorem and Egorov's theorem of measure theory. See Lemma 11.4 of \cite{ALOS} for a proof.

\begin{lemma}
	\label{lem:lema114}
There exists a compact set $K^\circ\subset K$ still of measure $\mu(K^\circ)>1-\delta$ and an integer $T=T(\delta)>1$ such that, whenever $n\geq T$, one has
\begin{equation}
	\label{eq:rec}
	\card\{j\in[0,n];\tif^j(\tix)\in K\}>(1-\delta)n,
\end{equation}
for every $\tix\in K^\circ$. 	
\end{lemma}
 
We let $\xi^u$ be a measurable partition of $\TTf$ subordinate to the lamination $\w^u$.  Given a pair $(\tix,\tix^u)$ such that $\tix^u\in\xi^u(\tix)$, we define the subset of integers
\[
E(\tix,\tix^u)\eqdef\{\ell\in\N;\tif^{\tau(\ell)}(\tix^u)\in K^\circ,\:\tif^{t(\ell)}(\tix)\in K^\circ\}.
\]

Given $\tix\in\TTf$ and an integer $\ell\in\N$, let us consider the set 
\begin{equation}
	\label{eq.q}
Q^u(\tix,\ell)\eqdef\{\tix^u\in\xi^u(\tix);\tix^u\in K^\circ,\:\ell\in E(\tix,\tix^u)\}.
\end{equation}  

Considering a continuous function $\eta:[0,1]\to[0,\infty)$ with $\eta(0)=0$ and an integer $\ell\in\N$, we define the set 
\[
K(\ell)\eqdef\{\tix\in K^\circ;\:\mu^u_{\tix}\left(Q^u(\tix,\ell)\right)>1-\eta(\delta)\}.
\]
This is the set of points $\tix\in\TTf$ for which one has ``a large set to choose $\tix^u\in\xi^u(\tix)$ from'' so that the length-$\ell$ $Y$-configuration $Y(\tix,\tix^u,\ell,\eps)$ has all its four upper points being ``highly recurrent'' to the compact set $K$. This high recurrence is used in the end of Lemma~\ref{lem:matching} to obtain simultaneous returns to $K$. 

The following result has the same proof as Proposition 11.6 in \cite{ALOS}, so we refrain from repeating the details. 

\begin{lemma}
    \label{lem:claim64}
    There exists a continuous function $\eta:[0,1]\to[0,\infty)$, with $\eta(0)=0$, (which depends only on $f$) so that the set $D(\delta,K)\eqdef\{\ell\in\N;\tm(K(\ell))>1-\eta(\delta)\}$ satisfies 
    \[
    \card D(\delta,K)\cap [0,n]>(1-\eta(\delta))n,
    \]
    for every $n>T(\delta)$. 
\end{lemma}

\subsection{The Eskin--Mirzakhnani combinatorial lemma}

The (combinatorial) result below is largely inspired by Lemma 12.8 of \cite{EskinMirzakhani} and, though elementary, it is one of the key ideas of our matching argument. 

\begin{lemma}
    \label{lem:eskinmirza}
    Let $\Delta>1$,  $m\in\N$, and $\eta_1,\eta_2\in(0,1)$ be numerical constants satisfying the following assumptions:
    \begin{itemize}
        \item $1-\eta_1>0.5$ and
        \item $\Delta m\eta_2<0.5$.
    \end{itemize}
    Assume that $(X,\nu,\cB)$ and $(Y,\hat{\nu},\hat{\cB})$ are probability spaces, each endowed with a finite collection $\{J_1,...,J_k\}\subset\cB $ and $\{\hat{J}_1,\dots,\hat{J}_k\}\subset\hat{\cB}$ of measurable sets so that the following three properties are satisfied:
    \begin{enumerate}
        \item $X=J_1 \cup \dots \cup J_k$ and $Y=\hat{J}_1\cup\dots\cup\hat{J}_k$,
        \item If $y\in Y$ then $\card\{i\in[1,k];y\in \hat{J}_i\}\leq m$,
        \item $\nu(J_i)\asymp_{\Delta}\hat{\nu}(\hat{J}_k)$.
    \end{enumerate}
    Suppose that there exist measurable sets $Q \subset X$ and $\hat{Q} \subset Y$ so that
    \[
    \nu(Q)>1-\eta_1\:\:\:\textrm{and}\:\:\:\hat{\nu}(\hat{Q})>1-\eta_2.
    \]
    Then, there exists $i\in[1,k]$ so that $J_i\cap Q\neq\emptyset$ and $\hat{J}_i\cap\hat{Q}\neq\emptyset$. 
\end{lemma}
\begin{proof}
    Let $I(X)\eqdef\{i\in [1,k]; J_i\cap Q\neq\emptyset\}$  and $I(Y)\eqdef\{i\in[1,k]; \hat{J}_i\cap\hat{Q}\neq\emptyset\}$.  Assume that $I(X)\cap I(Y)=\emptyset$. Then, 
    \begin{align*}
        0.5&<1-\eta_1<\nu(Q)\leq\sum_{i\in I(X)}\nu(Q\cap J_i)\leq\sum_{i\in I(X)}\nu(J_i)\nonumber\\
        &\leq\Delta\sum_{i\in I(X)}\hat{\nu}(\hat{J}_i)\leq\Delta\sum_{i\notin I(Y)}\hat{\nu}(\hat{J}_i)\nonumber\\
        &\leq\Delta m\hat{\nu}\big(\cup_{i\notin I(Y)}\hat{\nu}(\hat{J}_i)\big)\leq\Delta m\hat{\nu}(Y\setminus\hat{Q})\nonumber\\
        &<\Delta m\eta_2<0.5,
    \end{align*}
which is a contradiction.
\end{proof}
\subsection{Good matching of $Y$-configurations}

We consider, given $\delta>0$, a compact Lusin set $\cL\subset\TTf$ with $\tm(\cL)>1-\delta$. Given numbers $\beta>1$, $T>1$ and a compact subset $K$ of measure $\tm(K)>1-\delta$, we say that a pair $Y(\tix,\tix^u,\eps,\ell)$ and $Y(\tilde{y},\tilde{y}^u,\eps,\ell)$ of $(\ell,T)$-$Y$-configurations are in \emph{good matching} relatively to $(K,\cL,T,\beta)$ if:
\begin{enumerate}
    \item $(\tix,\tix^u,\tilde{y},\tilde{y}^u)$ is a $(\beta,\ell)$ quadrilateral,
    \item $\tix,\tilde{y}\in K\cap\cL$,
    \item there exist $\tilde{a}\in J(\tix^u)\cap K\cap\cL$, $\tilde{b}\in J(\tilde{y}^u)\cap K\cap\cL$ and $m\in(\tau(\ell)-T,\tau(\ell)-T)$ such that $\tilde{a}_m,\tilde{b}_m\in K\cap\cL$, and
    \item there exists $\widehat{m}\in(t(\ell)-T,t(\ell)+T)$ such that $\tix_{\widehat{m}},\tilde{y}_{\widehat{m}}\in K\cap\cL$. 
\end{enumerate}

The main result of this section which is also our main ingredient to prove Theorem~\ref{teo:main.rigidez} is the following. For the statement we recall the following definition: we say that a subset $D\subset\N$ of integers has \emph{positive density} if 
\[
\liminf_{n\to\infty}\frac{1}{n}\card(D\cap[1,n])>0.
\]
In particular, notice that every positive density subset $D\subset\N$ is an infinity set.
\begin{lemma}
    \label{lem:matching}
    There exist constants $\delta>0$, $\beta>1$, $T>1$, a Lusin set $\cL$ with measure $\tm(\cL)>1-\delta$ such that: for every compact set $K$ with $\tm(K)>1-2\delta$ there exists a subset $D\subset\N$ with positive density such that  for every $\ell\in D$ large enough, and every $\eps>0$ small enough, one can find a pair $Y(\tix,\tix^u,\eps,\ell)$ and $Y(\tilde{y},\tilde{y}^u,\eps,\ell)$ of $(\ell,T)$-$Y$-configurations which are in good matching relatively to $(K,\cL,T,\beta)$. 
\end{lemma}

The proof of Lemma~\ref{lem:matching} is rather involved. For the sake of clarity we shall present it in several steps. This result is parallel to Proposition 11.8 in \cite{ALOS}. The main difference in our case is that we consider \emph{non-uniform expansion along the center direction}, and thus we need to be much more careful with the choice of constants.

Let us give now a brief sketch of the argument: we choose $\cL_1\subset\cL_2$ two Lusin sets along with constants $\delta_2<<\delta_1$ so that $\tm(\cL_i)>1-\delta_i$, $i=1,2$. With the set $\cL_1$ we apply Proposition~\ref{prop:transversalidade} and obtain a constant $\beta_1>1$. With this constant we apply the synchronization of stopping times (Proposition~\ref{prop:sync}) and find a constant $T_1>1$. Once these parameters are fixed we can produce many coupled $(\ell,T_1)-Y-$ configurations. However, the definition of \emph{good matching} requires the six points $\tix,\tilde{y},\ta,\tb,\ta_m,\tb_m,\tix_{\hat{m}}$ and $\tilde{y}_{\hat{m}}$ lying inside the Lusin set. So, our strategy is to use the Lusin set $\cL_1$ to give uniform estimations allowing us to search for a configuration in good matching relatively to the \emph{larger} Lusin set $\cL_2$. 

The rest of this section is devoted to the formal proof of Lemma~\ref{lem:matching}.

\subsubsection{First step: choice of $\delta_1$}
As we announced before, we choose $\cL_1$ a Lusin set along with a constant $\delta_1$ so that $\tm(\cL_1)>1-\delta_1$. We shall first choose and fix $\delta_1$. 

Consider the function $\eta:[0,1]\to[0,\infty)$ given in Lemma~\ref{lem:claim64}, and let us denote $\eta_1\eqdef\eta(\delta_1)$. We consider $\gamma=\delta_1$ in Proposition~\ref{prop:transversalidade} and let $\beta_1=\beta(\delta_1)$ to be the constant given by that proposition.  We consider the constants $T_1\eqdef T(\cL_1,\beta_1)$ be the constant given in the Synchronization Proposition~\ref{prop:sync}. We take $\delta_1$ sufficiently small so that $1-\eta_1>0.6$ and
\begin{equation}
	\label{eq:etaum}
\frac{\eta_1}{1-\eta_1}<0.5.
\end{equation}
With this choice of $\delta_1$ we fix a Lusin set $\cL_1$ with measure $\tm(\cL_1)>1-\delta_1$. 
\subsubsection{Second step: choice of $\delta_2$} 
We shall now choose $\delta_2<<\delta_1$. For this we will need a series of lemmas giving numbers depending on the value of $\delta_1$, previously fixed. 
\begin{lemma}[Overlap control of unstable dynamical balls]
\label{lem:overlapcontrol}
There exists a constant $m=m(\cL_1)$ with the following property. Assume that $(\tix,\tix^u,\tilde{y},\tilde{y}^u)$ is a $(\beta_1,\ell)$-quadrilateral, with $\tix,\tilde{y}\in\cL_1$. Given a compact set $Q\subset\widetilde{\cW}^u_1(\tix)\cap\cL_1$ with $\tm(Q)>1-\eta_1$, given any finite cover $X=\cup_{i=1}^kJ_i$ of 
\[
Q\cap\left(\widetilde{\cW}^u_1(\tix)\setminus\widetilde{\cW}^u_{1/\beta_1}(\tix)\right),
\]
by unstable dynamical balls $J_i=J^u(\tix^u_i)$ such that for every $\tilde{z}\in\widetilde{\cW}^u_1(\tix)$
\[
\card\{i;\tilde{z}\in J_i\}\leq 2,
\]
defining $\tilde{y}^u_i=H^{cs}_{\tix,\tilde{y}}(\tix^u_i)$ and $\hat{J}_i\eqdef J^u(\tilde{y}^u_i)$, then for every $\tilde{b}\in\widetilde{\cW}^u_1(\tilde{y})$
\[
\card\{i;\tilde{b}\in\hat{J}_i\}\leq m.
\]
\end{lemma}
\begin{proof}
Take a point $\tb\in\widetilde{\cW}^u(\tilde{y})$ and denote by 
\[
\hat{J}_{i_1},\dots,\hat{J}_{i_m}
\]
all the intervals within the collection $\{\hat{J}_1,\dots,\hat{J}_k\}$ that contains the point $\tb$. Consider $J\eqdef\cup_{j=1}^mJ_{i_j}$ and $\hat{J}\eqdef\cup_{j=1}^m\hat{J}_{i_j}$. We shall bound from above the length of $\hat{J}$ and from below the length of $J$. The estimate on $m$ will then follow from the Hölder continuity of center holonomies.

For simplicity, we also denote
\[
\tau_{i_j}\eqdef\tau(\tix,\tix^u_{i_j},\eps,\ell)\:\:\textrm{and}\:\:\:\hat{\tau}_{i_j}\eqdef\tau(\tilde{y},\tilde{y}^u_{i_j},\eps,\ell).
\]
Let $\mathcal{T}\eqdef\{\tau_{i_1},\dots,\tau_{i_m},\hat{\tau}_{i_1},\dots,\hat{\tau}_{i_m}\}$ and define $\tau^*\eqdef\min\mathcal{T}$. For the bound on $\hat{J}$ one observe that $\tif^{\tau^*}(\tb)\in\tif^{\tau^*}(\hat{J}_{i_j})$, for every $j=1,\dots,m$. Since $|\tif^{\tau^*}(\hat{J}_{i_j})|\leq 2$ it follows that $|\hat{J}|\leq 4$. 

For the bound on $J$, since no more than two intervals $J_{i_j}$ overlap at a time we can estimate
\[
\left|\tif^{\tau^*}(J_{i_j})\right|=\left|\bigcup_{j=1}^m\tif^{\tau^*}(J_{i_j})\right|\geq\frac{1}{2}\sum_{j=1}^m\left|\tif^{\tau^*}(J_{i_j})\right|
\]
However, one has, by definition of a dynamical ball, that $\left|\tif^{\tau_{i_j}}(J_{i_j})\right|=2$. Also, by Proposition~\ref{prop:sync} we have that $|\tau_{i_j}-\tau(\tilde{y},\tb,\eps,\ell)|<T(\cL_1,\beta_1)$ and also $|\hat{\tau}_{i_j}-\tau(\tilde{y},\tb,\eps,\ell)|<T(\cL_1,\beta_1)$, for every $j=1,...,m$. Thus, by the triangle inequality we have that for every $\tau\in\mathcal{T}$ it holds
\[
\tau-\tau^*<2T(\cL_1,\beta_1).
\] 
These two facts together imply that 
\[
\left|\tif^{\tau^*}(J_{i_j})\right|\geq 2e^{-\xi(f)T_1},
\]
where $\xi(f)=\max\{\|Df(x)v\|/\|v\|;x\in\TT,v\in T_x\TT\}$ and $T_1=T_1(\cL_1,\beta_1)$ is a constant depend on the Lusin set $\cL_1$. Thus,
\[
\left|\tif^{\tau^*}(J_{i_j})\right|\geq 2me^{-\xi(f)T_1}.
\]
Since $J$ is a segment of unstable manifold, there must exist a pair of points $\ta,\ta^{\prime}$ on $J$ such that $d^u(\ta,\ta^{\prime})>me^{-\xi(f)T_1}$. Since every point of $J$ is at a distance of at most 1 from some $\tix^u_{i_j}$ we conclude that there must exist a pair of points $x^u_{i_j}$ and $x^u_{i_r}$ on $\cW^u(\tix_{\tau^*})$ such that 
\[
d^u(x^u_{i_j},x^u_{i_r})>me^{-\xi(f)T_1}-2.
\]  
By the Hölder continuity of the center bundle, we have that 
\[
d^u(x^u_{i_j},x^u_{i_r})\leq C(f)d^u(y^u_{i_j},y^u_{i_r})^{\theta(f)}\leq C(f)4^{\theta(f)}.
\]
We deduce that
\[
m\leq e^{\xi(f)T_1}\left(C(f)4^{\theta(f)}+2\right),
\]
with $T_1=T_1(\cL_1,\beta_1)$. This ends the proof.
\end{proof}

\begin{remark}
In our argument, the constant $m(\cL_1)$ appearing in the lemma above depends on our \emph{first} choice of Lusin set $\cL_1$ which we use to have uniform estimates allowing us to put all eight points of the main picture \ref{fig:EMscheme} in a larger Lusin set $\cL_2$, while in \cite{ALOS} the constant is uniform (depending only on the dynamical system).
\end{remark}

In the same setting as above, we need to determine another constant which compares the lengths of the unstable dynamical balls $J_i,\hat{J}_i$. 

\begin{lemma}
\label{lem:lengthcontrol}
There exists a constant $\Delta=\Delta(\cL_1)$ with the following property. Assume that $(\tix,\tix^u,\tilde{y},\tilde{y}^u)$ is a $(\beta_1,\ell)$-quadrilateral, with $\tix,\tilde{y}\in\cL_1$. Given a compact set $Q\subset\widetilde{\cW}^u_1(\tix)\cap\cL_1$ with $\tm(Q)>1-\eta_1$, given any finite cover $X=\cup_{i=1}^kJ_i$ of 
\[
Q\cap\left(\widetilde{\cW}^u_1(\tix)\setminus\widetilde{\cW}^u_{1/\beta_1}(\tix)\right),
\]
by unstable dynamical balls $J_i=J^u(\tix^u_i)$ such that for every $\tilde{z}\in\widetilde{\cW}^u_1(\tix)$
\[
\card\{i;\tilde{z}\in J_i\}\leq 2,
\]
defining $\tilde{y}^u_i=H^{cs}_{\tix,\tilde{y}}(\tix^u_i)$ and $\hat{J}_i\eqdef J^u(\tilde{y}^u_i)$, if $X=\cup_{i=1}^kJ_i$ and $Y=\cup_{i=1}^k\hat{J}_i$ the normalized probabilities 
\[
\nu\eqdef\frac{\tm^u_{\tix}|_X}{\tm^u_{\tix}(X)}\:\:\:\textrm{and}\:\:\:\:\hat{\nu}\eqdef\frac{\tm^u_{\tilde{y}}|_Y}{\tm^u_{\tilde{y}}(Y)}
\] 
satisfy
\[
\nu(J_i)\asymp_{\Delta}\hat{\nu}(\hat{J}_i),
\]
for every $i=1,...,k$.
\end{lemma}
\begin{proof}
At one hand, it follows from Corollary~\ref{cor:tamanhocomparavel} that $|J_i|\asymp_C|\hat{J}_i|$, for some constant $C=C(\cL_1,\beta)$. From Lemma~\ref{lem:densidadeuniforme} we may enlarge this constant (but still depending only on $\cL_1$) so that $\tm^u_{\tix}(J_i)\asymp_{C}\tm^u_{\tilde{y}}(\hat{J}_i)$.  	
 Lemma~\ref{lem:overlapcontrol} allows us to estimate 
\[
\left|\bigcup_{i=1}^k\hat{J}_i\right|\geq\frac{1}{m}\sum_{i=1}^k|\hat{J}_i|\geq\frac{1}{Cm}\sum_{i=1}^k|J_i|\geq\frac{1}{Cm}\left|\bigcup_{i=1}^k J_i\right|,
\]
and a similar reasoning gives an upper bound which shows that $|X|$ and $|Y|$ are comparable up to some constant $\hat{C}=\hat{C}(\cL_1,\beta)$. Using Lemma~\ref{lem:densidadeuniforme}, we may enlarge $\hat{C}$ (but still depending only on $\cL_1$) if necessary in order to have that $\tm^u_{\tix}(X)\asymp_{\hat{C}}\tm^u_{\tilde{y}}(Y)$. Putting all these estimates together, one deduces that the probabilities $\nu=\tm^u_{\tix}|_X/\tm^u_{\tix}(X)$ and $\hat{\nu}=\tm^u_{\tilde{y}}|_Y/\tm^u_{\tilde{y}}(Y)$ satisfy  
\[
\nu(J_i)\asymp_{\Delta}\hat{\nu}(\hat{J}_i),
\]
for every $i=1,...,k$ and for some constant $\Delta=\Delta(\cL_1)$, as claimed.
\end{proof}

\begin{remark}
	\label{rem:malin}
The argument above also gives, as a by product, that $\tm^u_{\tilde{y}}(Y)\geq \zeta>0$, for some uniform constant $\zeta=\zeta(\cL_1)$.
\end{remark}

We are now in position to give our choice of the constant $\delta_2$. Let $\zeta(\cL_1)$ be the constant given in Remark~\ref{rem:malin}. Since $\eta_2=\eta(\delta_2)\to 0$ as $\delta_2\to 0$ we may choose $\delta_2$ small enough so that 
\begin{equation}
\label{eq:deltadois}
\frac{\eta_2\Delta m}{\zeta(\cL_1)}<0.5\:\:\:
\end{equation}
Once $\delta_2$ is fixed by \eqref{eq:deltadois}, we choose a larger Lusin set $\cL_2\supset\cL_1$ with measure $\tm(\cL_2)>1-\delta_2$.
\subsubsection{Third step: final choice of constants and sets}
We take $\delta=\delta_2/10$ and $\cL\supset\cL_2$ a Lusin set with $\tm(\cL)>1-\delta$. Let $\beta=\beta_1$. To complete the constants and sets announced in Lemma~\ref{lem:matching} it only remains to define $T>1$. Let $T=T(\delta_1)$ be chosen large enough, so that, given any compact set $K$ with measure $\tm(K)>1-\delta_1$, the set of points $\tix\in\TTf$ for which \eqref{eq:rec} holds has measure larger than  $1-\delta_1$. Notice that the existence of $T$ depending only on $\delta_1$ is ensured by Lemma~\ref{lem:lema114}. Enlarging this choice, if necessary, we may further assume that $T\geq T(\cL_1,\beta_1)$, where $T(\cL_1,\beta_1)$ is the constant given by the synchronization Proposition~\ref{prop:sync}.

We claim that Lemma~\ref{lem:matching} holds with these choices.
\subsubsection{Fourth setp: finding the configurations}
Let $K\subset\TTf$ be a compact set with measure $\tm(K)>1-2\delta$. 

Let $K_i=(K\cap\cL_i)^{\circ}$ be the set of points which are $(\delta_1,T)$-recurrent to $K\cap\cL_i$, in the sense that equation \eqref{eq:rec} holds with $\delta=\delta_1$, $K=K\cap\cL_i$ and $T$. By Lemma~\ref{lem:lema114} we know that $\tm(K_i)>1-3\delta_i$. We also let $Q^u_i(\tix,\ell)$ denote the corresponding set, as defined in \eqref{eq.q}, with $K=K_i$. Notice that $Q^u_1(\tix,\ell)\subset Q^u_2(\tix,\ell)$.


    Then, Lemma~\ref{lem:claim64} gives $D_1, D_2\subset\N$ sets of positive integer such that 
    \[
    \ell\in D_i\implies \tm\big(K_i(\ell)\big)>1-\eta_i,
    \]
    and for every $L>T$ one has 
    \[
    \operatorname{card} D_i\cap [0,L]>(1-\eta_i)L.
    \]
Since $1-\eta_1>0.6$, we have that $D\eqdef D_1\cap D_2$ has positive density. In particular, $\operatorname{card}(D)=\infty$. Let $\ell\in D$ and consider $K(\ell)=K_1(\ell)\cap K_2(\ell)$. 
    
    Fix $A=\tif^{-\ell}(K(\ell))\cap\cL_1$. Notice that
    \[
    \tm(A)>1 - \eta_1 - \eta_2 - \delta_1 >1-3\eta_1,
    \]
    since $\eta_2<\eta_1$.
    
    By applying Proposition~\ref{prop:transversalidade} we find $A'\subset A$ such that $\tm(A')>1-2\sqrt{3\eta_1}$. Take a point $\tix_{-\ell}\in A'$. Then, there exists $\tilde{y}_{-\ell}\in A\cap \Sigma(\tix_{-\ell})$ such that $\alpha(\tix_{-\ell},\tilde{y}_{-\ell})>\beta$. Thus, we have $\tix,\tilde{y}\in K(\ell)$ and 
    \[
    \tm\left(Q^u_i(\tix,\ell)\right)>1-\eta_i.
    \]

    We now claim that there exists $\tix^u\in Q^u_1(\tix,\ell)\setminus\w^u_{\beta^{-1}}(\tix)$ such that, defining $\tilde{y}^u=H^{cs}_{\tix,\tilde{y}}(\tix^u)$, one has that 
    \[
    J(\tilde{y}^u)\cap Q^u_2(\tilde{y},\ell)\neq\emptyset.
    \]
    To prove this claim, first take a compact set $Q\subset Q^u(\tix,\ell)$ with $\mu^u_{\tix}(Q)>1-\eta_1$. We now take a finite cover of the compact set $Q\cap\left(\overline{\w^u_1(\tix)\setminus\w^u_{\beta^{-1}}(\tilde{y})}\right)$ by unstable dynamical balls $\{J_i=J(\tix^u_i)\}_{i=1}^k$, with $\tix^u_i\in Q$, for every $i=1,\dots,k$. We may assume, up to taking a subcover, that no more than two intervals $J_i$ overlap at the same point. We define $X=\cup_{i=1}^kJ_i$. Observe that 
    \[
    \frac{\tm^u_{\tix}(X\setminus Q)}{\tm^u_{\tix}(X)}\leq\frac{\eta_1}{1-\eta_1}<0.5,
    \]
    due to \eqref{eq:etaum}.

    For each $i$ consider $\tilde{y}^u_i=H^{cs}_{\tix,\tilde{y}}(\tix^u_i)$ and $\hat{J}_i\eqdef J(\tilde{y}^u_i)$. We define $Y=\cup_{i=1}^k\hat{J}_i$. From Lemma~\ref{lem:lengthcontrol}, if we define probabilities $\nu=\tm^u_{\tix}|_X/\tm^u_{\tix}(X)$ and $\hat{\nu}=\tm^u_{\tilde{y}}|_Y/\tm^u_{\tilde{y}}(Y)$ we have that 
    \[
    \nu(J_i)\asymp_{\Delta}\hat{\nu}(\hat{J}_i),
    \]
    for every $i=1,...,k$ and for some constant $\Delta=\Delta(\cL_1)$. By our choice of $\delta_2$ we have that
     \[
     \hat{\nu}(Y\setminus Q^u_2(\tilde{y},\ell))=\frac{\tm^u_{\tilde{y}}(Y\setminus Q^u_2(\tilde{y},\ell))}{\tm^u_{\tilde{y}}(Y)}<\frac{\eta_2}{\zeta(\cL_1)},
     \]
     where $\zeta(\cL_1)$ is given in Remark~\ref{rem:malin}. Thus, by \eqref{eq:deltadois} our claim now follows directly from Lemma~\ref{lem:eskinmirza}.
\subsubsection{Ending the proof}
    We define $\ta=\tix^u$ and $\tb\in J(\tilde{y}^u)\cap Q^u_2(\tilde{y},\ell)$ given by the claim. Thus we have that 
    \[
    \tif^{\tau_{\ta}}(\ta),\tif^{\tau_{\tb}}(\tb), \tif^t_{\tix}(\tix),\tif^{t_{\tilde{y}}}(\tilde{y})\in K_2,
    \]
    where $\tau_a,\tau_b,t_{\tix},t_{\tilde{y}}$ are the stopping times corresponding to the $Y$ configurations $Y(\tix,\ta,\eps,\ell)$ and $Y(\tilde{y},\tb,\eps,\ell)$. In particular, these four points are $(\delta_1,T_1)$-recurrent to $K$. By the synchronization of stopping times given in Proposition~\ref{prop:sync}, we have that $|\tau_{\ta}-\tau_{\tb}|<T_1$ and $|t_{\tix}-t_{\tilde{y}}|<T_1$. Let $\tau\eqdef\min\{\tau_{\ta},\tau_{\tb}\}$. Since $\delta_1<0.3$, for $\star=\ta,\tb$, the set 
    \[
    \{m\in[\tau,\tau+5T_1];\tif^m(\star)\in K\}
    \]
    has density larger than $0.7$ in the interval $[\tau,\tau+5T_1]$. Thus, there exists $m\in[\tau,\tau+5T_1]$ such that 
    \[
    \tif^m(\ta),\tif^m(\tb)\in K.
    \]
   With a similar reasoning one deduces the existence of an integer $\hat{m}$ such that $\tif^{\hat{m}}(\tix),\tif^{\hat{m}}(\tilde{y})\in K$. This completes the proof of Lemma~\ref{lem:matching}. \qed

\section{Invariance of leaf-wise quotient measures by affine maps}
\label{sec:proof-main-lemma2}

We are going to complete the proof of Theorem~\ref{teo:main.rigidez} in this section by demonstrating that there exists a positive measure set of points in $\TTf$ for which the associated leaf-wise quotient measure is invariant by sufficiently many small affine maps. Once we show this, we apply a result of Kalinin and Katok to conclude.  

\subsection{Small affine maps}

We say that a collection $\cA\subset\operatorname{Aff}(\R)$ of affine maps of the real line is a collection of $M$-\emph{small affine maps} if there exists a constant $M>1$ such that for every $n\in\N$ there exist $0<\eps<1/n$ and $\psi\in\cA$ such that
\begin{enumerate}
        \item $M^{-1} < \vert \psi'(0) \vert < M$; 
        \item $\varepsilon M^{-1} < \vert \psi(0) \vert < \varepsilon M$;
\end{enumerate}

The proof of the result below can be found in \cite{ALOS}, Section 8. See also \cite{BRH}, Section 7 or the original reference \cite{KalininKatok}.

\begin{lemma}[Kalinin--Katok]
    \label{lem:kksmall}
Assume that there exist $M>1$ and $\delta>0$ and a set $G\subset\TTf$ with $\tm(G)\geq\delta$ such that for every $\tix\in G$ there exists a collection $\cA(\tix)\subset\operatorname{Aff}(\R)$ of $M$-small affine maps satisfying
\[
\hat{\nu}^c_{\tix}\propto\psi_*\hat{\nu}^c_{\tix},\:\:\:\textrm{for every}\:\:\:\psi\in\cA(\tix).
\]
Then, $\mu$ is an SRB measure. 
\end{lemma}

By Lemma~\ref{lem:srbendo} we know that every SRB measure is absolutely continuous with respect to $\operatorname{Leb}_{\TT}$.

Thus, to prove Theorem~\ref{teo:main.rigidez} we need to show that the assumptions of Lemma~\ref{lem:kksmall} are satisfied. Consider $\delta>0,\beta>1$ and $T>1$ the constants given in Lemma~\ref{lem:matching}. Let also $\cL$ be the Lusin set with measure $\tm(\cL)>1-\delta$, also given by Lemma~\ref{lem:matching}.

\begin{lemma}
    \label{lem:main-lemma} 
There exists $M=M(\delta,\beta,T)>1$ such that, for every $\eps>0$ sufficiently small, there exists a measurable set $G(\eps,M)\subset\cL$, with $\tm\left(G(\eps,M)\right)\geq \delta$ such that, for every $\tix\in G(\eps,M)$, there exists:
\begin{itemize}
    \item $\tilde{y}\in W_{\loc}^c(\tix)$, with $\cR^c_{\tix}(\tilde{y}))\asymp_M\eps$, and
    \item a linear map $B:\R\to\R$, with slope $B'(0)\asymp_M1$,
\end{itemize}
such that $\hat{\nu}^c_{\tilde{y}}\propto B_*\hat{\nu}^c_{\tix}$. 
\end{lemma}

\subsubsection{Proof of Theorem~\ref{teo:main.rigidez}}
We claim that Lemma~\ref{lem:main-lemma} implies Theorem~\ref{teo:main.rigidez}. Indeed, define 
\[
G=\{\tix\in\TTf \; : \;\tix\in G(1/N,M)\:\textrm{for infinitely many}\:N\in\N\}.
\]
Notice that $\tm(G)\geq\delta$. We need to show that for every $\tix\in G$ there exists a collection of small affine maps which preserves, up to a factor, the measure $\hat{\nu}^c_{\tix}$. To see this, fix some $\eps$ and suppose $\tix\in G(\eps,M)$. Using Lemma~\ref{lem:basicmoves} we obtain an affine map $h^2_{\tix,\tilde{y}}:\R\to\R$ such that $\hat{\nu}^c_{\tilde{y}}\propto (h^2_{\tix,\tilde{y}})_*\hat{\nu}^c_{\tix}$, and $h^2_{\tix,\tilde{y}}(s)=\rho_{\tix}(\tilde{y})s+\cR^c_{\tix}(\tilde{y})$. Since $G(\eps,M)\subset\cL$, by Lemma~\ref{lem:main-lemma}, we may assume that 
\[
\rho^c_{\tix}(\tilde{y})\asymp_M 1.
\]
Thus, by defining $\psi=B\circ(h^2_{\tix,\tilde{y}})^{-1}$ we have that $\psi$ satisfies items (1) and (2) of the definition of small affine maps for some constant $\hat{M}=\hat{M}(M)$. This shows that for every $\tix\in G$ there exists a collection of $\hat{M}$-small affine maps preserving $\hat{\nu}^c_{\tix}$, up to a factor. Theorem~\ref{teo:main.rigidez} now follows from Lemma~\ref{lem:kksmall}.\qed
\subsection{The drift argument: proof of Lemma~\ref{lem:main-lemma}}
We now proceed to show how to deduce Lemma~\ref{lem:main-lemma} from Lemma~\ref{lem:matching}. The first step is to determine the constant $M=M(\delta,T,\beta)>1$. We first consider the movement of quotient measures along a ``$\cL$-good'' $Y$-configuration, with particular attention to the factor relating the measures. 

\begin{lemma}
    \label{lem:lemadom}
There exists $M=M(\cL,T)>1$ such that, given $\eps>0$ and $\ell\in\N$, if there exists $(\ell,T)$ $Y$-configuration $Y(\tix,\ta,\eps,\ell)$ satisfying 
\[
\tix,\ta,\ta_m,\tix_{\hat{m}}\in\cL,
\]
then there exist $C>1$, with $C\asymp_M1$, and linear map $B:\R\to\R$, with slope $B'(0)\asymp_M1$, such that 
\[
\hat{\nu}^c_{\ta_m}=CB_*\hat{\nu}^c_{\tix_{\hat{m}}}.
\]
\end{lemma}
\begin{proof}
It follows from the basic moves of leaf-wise quotient measures Lemma~\ref{lem:basicmoves} that $\hat{\nu}^c_{\ta_m}=CB_*\hat{\nu}^c_{\tix_{\hat{m}}}$, where $B$ is the linear map with slope
\[
B'(0)=\rho^c_{\tix}(\ta)\lambda^c_{\tix}(\hat{m})\left(\lambda^c_{\ta}(m)\right)^{-1},
\]
and the constant $C$ is given by
\[
C=\left(\hat{\nu}^c_{\tix_{\hat{m}}}(B^{-1}([-1,1]))\right)^{-1}.
\]
We claim that these quantities are uniformly bounded. Indeed, by the definition of the stopping times and from Lemma~\ref{lem:crescelyapunovcresce} we have that 
\[
\hat{\lambda}^c_{\tix}(t(\ell))\left(\hat{\lambda}^c_{\ta}(\tau(\ell))\right)^{-1}
\]
is uniformly bounded. Since the Lyapunov norm and the background norm are uniformly comparable inside the Lusin set, this implies that $B'(0)/\rho^c_{\tix}(\ta)$ is bounded by a constant depending on $\cL$ and $T$. Since $\ta,\tix\in\cL$, we also have that $\rho^c_{\tix}(\ta)$ is uniformly bounded (with constant depending on $\cL$). A similar reasoning applies to $C$.
\end{proof}

We can now fix $M$. We assume that the constant given by Lemma~\ref{lem:lemadom} is large enough so that it satisfies also the following: if $\tix\in\cL$, then for every $\tilde{a}\in\w^u_{\sigma^T}(\tix)$ one has $\rho^c_{\tix}(\ta)\asymp_M1$, where $\sigma=\max\{\lambda^u_{\tix};\tix\in\TTf\}$. With this choice, we continue to have $M=M(\cL,T)$. We are now in position to complete our proof.

\subsubsection{Proof of Lemma~\ref{lem:main-lemma}}
Let $M$ be the constant determined above and consider $G(\eps)\eqdef G(\eps,M^{3})$ the (measurable) set of points $\tix\in\cL$ such that for every $\tix\in G(\eps)$ one can find $\tilde{y}\in \tW^c_{\loc}(\tix)$ and a linear map $B:\R\to\R$ with slope $B'(0)\asymp_{M^3}1$ such that $\cR^c_{\tix}(\tilde{y})\asymp_{M^3}\eps$ and $\hat{\nu}^c_{\tilde{y}}\propto B_*\hat{\nu}^c_{\tix}$. Assume, by contradiction, that $\tm(G(\eps))<\delta$. Then, by the regularity of the measure, there must exist a compact set 
\[
K\subset\cL\cap\left(\TTf\setminus G(\eps)\right)
\] 
with measure $\tm(K)>1-2\delta.$ We find a contradiction by proving that $K\cap G(\eps)\neq\emptyset$.

For that, we fix some large $\ell\in D$ and consider a pair $Y(\tix,\tix^u,\eps,\ell)$ and $Y(\tilde{y},\tilde{y}^u,\eps,\ell)$ of $(\ell,T)$-$Y$-configurations which are in good matching relatively to $(K,\cL,T,\beta)$, as given by Lemma~\ref{lem:matching}. See Figure~\ref{fig:proof-main-lemma}.

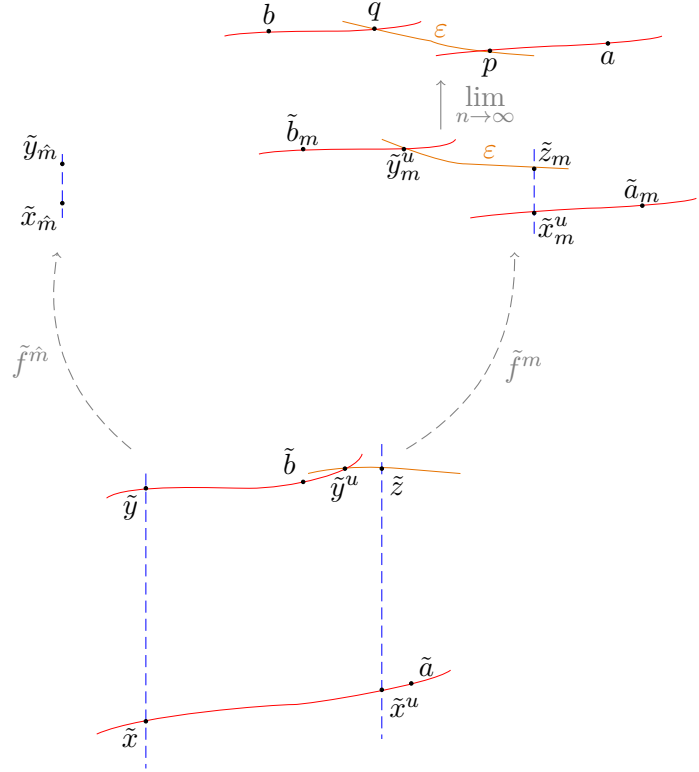
\begin{figure}[ht]
    \centering
\begin{tikzpicture}[y=1cm, x=1cm, yscale=0.65,xscale=0.65, inner sep=0pt, outer sep=0pt]
   
  \path[draw=blue,line cap=butt,line join=miter,line width=0.0cm,miter 
  limit=4.0,dash pattern=on 0.16cm off 0.08cm] (3.0, 6.0) -- (3.0, -0.0);
  \path[draw=blue,line cap=butt,line join=miter,line width=0.0cm,miter 
  limit=4.0,dash pattern=on 0.16cm off 0.08cm] (7.8, 6.6) -- (7.8, 0.6);
  \path[draw=orange!90!black,line cap=butt,line join=miter,line width=0.0cm] (6.3, 6.0).. 
  controls (6.3, 6.0) and (7.0, 6.2) .. (8.2, 6.1).. controls (9.4, 6.0) and 
  (9.4, 6.0) .. (9.4, 6.0);
  \path[draw=red,line cap=butt,line join=miter,line width=0.0cm] (2.0, 0.7).. 
  controls (2.9, 1.1) and (5.6, 1.3) .. (6.0, 1.3).. controls (6.3, 1.3) and 
  (8.8, 1.7) .. (9.2, 2.0);
  \path[draw=red,line cap=butt,line join=miter,line width=0.0cm,miter limit=4.0]
   (2.2, 5.5).. controls (2.6, 5.8) and (4.5, 5.7) .. (5.1, 5.7).. controls 
  (5.9, 5.7) and (7.3, 6.0) .. (7.4, 6.4);

  \draw[->, gray, dash pattern=on 0.16cm off 0.08cm] (8.3,6.5) to[bend right] node[midway, right=0.2cm]{$\tilde{f}^{m}$} (10.5,10.5);
  \draw[->, gray, dash pattern=on 0.16cm off 0.08cm] (2.7,6.5) to[bend left] node[midway, left=0.2cm]{$\tilde{f}^{\hat{m}}$} (1.2,10.5);
  
  \filldraw (3.0, 5.7) circle (1pt) node[below left=0.1cm]{$\tilde{y}$};
  \filldraw (3.0, 0.96) circle (1pt) node[below left=0.1cm]{$\tilde{x}$};
  \filldraw (7.05, 6.1) circle (1pt) node[below=0.1cm]{$\tilde{y}^u$};
  \filldraw (6.2, 5.83) circle (1pt) node[above left=0.1cm]{$\tb$};
  \filldraw (7.8, 1.6) circle (1pt) node[below right=0.1cm]{$\tilde{x}^u$};
  \filldraw (7.8, 6.1) circle (1pt) node[below right=0.1cm]{$\tilde{z}$};
  \filldraw (8.4, 1.73) circle (1pt) node[above right=0.1cm]{$\ta$};

  \path[draw=blue,line cap=butt,line join=miter,line width=0.0cm,miter 
  limit=4.0,dash pattern=on 0.16cm off 0.08cm] (10.9, 12.6) -- (10.9, 10.8);
  \path[draw=orange!90!black,line cap=butt,line join=miter,line width=0.0cm,miter 
  limit=4.0] (7.8, 12.8).. controls (7.8, 12.8) and (9.0, 12.3) .. (9.5, 12.3) 
  -- (11.6, 12.2);
  \path[draw=red,line cap=butt,line join=miter,line width=0.0cm,miter limit=4.0]
   (9.6, 11.2).. controls (10.2, 11.3) and (11.9, 11.4) .. (12.1, 11.4).. 
  controls (12.4, 11.4) and (14.0, 11.5) .. (14.2, 11.6);
  \path[draw=red,line cap=butt,line join=miter,line width=0.0cm,miter limit=4.0]
   (5.3, 12.5).. controls (5.7, 12.6) and (7.1, 12.6) .. (7.6, 12.6).. controls 
  (8.2, 12.6) and (9.2, 12.6) .. (9.3, 12.8);

  \filldraw (8.25, 12.6) circle (1pt) node[below=0.05cm]{$\tilde{y}^u_m$};
  \filldraw (6.2, 12.6) circle (1pt) node[above=0.05cm]{$\tb_m$};
  \filldraw (10.9, 11.3) circle (1pt) node[below right=0.05cm]{$\tilde{x}^u_m$};
  \filldraw (10.9, 12.2) circle (1pt) node[above right=0.05cm]{$\tilde{z}_m$};
  \filldraw (13.1, 11.45) circle (1pt) node[above=0.05cm]{$\ta_m$};

  \draw[orange!90!black] (10, 12.55) node{$\varepsilon$};

  \path[draw=blue,line cap=butt,line join=miter,line width=0.0cm,miter 
  limit=4.0,dash pattern=on 0.16cm off 0.08cm] (1.3, 12.5) -- (1.3, 11.2);

  \filldraw (1.3, 11.5) circle (1pt) node[below left=0.05cm]{$\tilde{x}_{\hat{m}}$};
  \filldraw (1.3, 12.3) circle (1pt) node[above left=0.05cm]{$\tilde{y}_{\hat{m}}$};
  
  \draw[->, gray] (9,13) to node[midway, right=0.2cm]{$\lim\limits_{n \to \infty}$} (9,14);

  \path[draw=orange!90!black,line cap=butt,line join=miter,line width=0.0cm,miter 
  limit=4.0] (7.0, 15.2).. controls (7.0, 15.2) and (8.2, 14.9) .. (8.8, 
  14.8).. controls (9.2, 14.6) and (10.9, 14.5) .. (10.9, 14.5);
  \path[draw=red,line cap=butt,line join=miter,line width=0.0cm,miter limit=4.0]
   (8.9, 14.5).. controls (9.5, 14.6) and (11.2, 14.7) .. (11.4, 14.7).. 
  controls (11.7, 14.7) and (13.3, 14.8) .. (13.5, 14.9);
  \path[draw=red,line cap=butt,line join=miter,line width=0.0cm,miter limit=4.0]
   (4.6, 14.9).. controls (5.0, 15.0) and (6.4, 15.0) .. (6.9, 15.0).. 
  controls (7.5, 15.0) and (8.5, 15.1) .. (8.6, 15.2);
  
  \filldraw (7.65, 15.05) circle (1pt) node[above =0.1cm]{$q$};
  \filldraw (5.5, 15.0) circle (1pt) node[above =0.1cm]{$b$};
  \filldraw (10.00, 14.6) circle (1pt) node[below=0.1cm]{$p$};
  \filldraw (12.4, 14.75) circle (1pt) node[below=0.1cm]{$a$};
  
  \draw[orange!90!black] (9, 15) node{$\varepsilon$};
  
\end{tikzpicture}
    \caption{We control the movements of quotient measures along the approximating configurations in good matching so that the limit point $b$ belongs to $K\cap G(\eps)$.}
    \label{fig:proof-main-lemma}
\end{figure}

Take $a$ to be a limit point of the sequence $\ta_{m}=\tif^m(\ta)$, as $\ell\to\infty$. Observe that, to simplify the notation, we dropped the dependence of the points in $\ell$, but the reader should keep in mind that for each $\ell$ we have a picture like Figure~\ref{fig:proof-main-lemma}, so that $\tix,\ta,\ta_{m}\in\cL\cap K$ as well as $\tilde{y},\tb,\tb_{m}\in\cL\cap K$ and $\tix_{\hat{m}},\tilde{y}_{\hat{m}}\in\cL\cap K$. We may extract subsequences so that all these sequences converge to points in $\cL\cap K$. In particular, let also $b$ be a limit point of $\tb_{m}$ as $\ell\to\infty$. Notice also that $d(\tix_{\hat{m}},\tilde{y}_{\hat{m}})\to 0$, so we can assume both sequences converge to a point $x$. 

By applying Lemma~\ref{lem:lemadom}, we obtain that $\hat{\nu}^c_{\ta_m}=CB_*\hat{\nu}^c_{\tix_{\hat{m}}}$ where the constant $C$ and the slope of the linear map $B$ are bounded independently of $\ell$ and $\eps$. Similarly we obtain that $\hat{\nu}^c_{\tb_m}=\hat{C}\hat{B}_*\hat{\nu}^c_{\tilde{y}_{\hat{m}}}$, where $\hat{C}$ and the slope of the linear map $\hat{B}$ are uniformly bounded. Up to further extraction of subsequences, we may assume that $C,\hat{C},B,\hat{B}$ converge as $\ell\to\infty$. 

On the other hand, continuity inside the Lusin set implies that $\hat{\nu}^c_{\tix},\hat{\nu}^c_{\tilde{y}}\to\hat{\nu}^c_x$. Then we have that
\[
\hat{\nu}^c_a=\alpha\gamma_*\hat{\nu}^c_b,
\]
for some constant $\alpha\asymp_{M^2}1$ and a linear map $\gamma$ with slope $\asymp_{M^2}1$. See Figure~\ref{fig:lim-main-lemma}.

\begin{figure}[ht]
    \centering
    \begin{tikzpicture}
  \path[draw=orange!90!black!80!black,line cap=butt,line join=miter,thick] (20.4, 6.6)..
   controls (20.6, 5.8) and (20.6, 4.0) .. (20.6, 4.0) -- (20.6, 4.0);
  \path[draw=orange!90!black,line cap=butt,line join=miter, thick] (18.4, 6.6)..
   controls (18.6, 5.8) and (18.6, 4.0) .. (18.6, 4.0) -- (18.6, 4.0);
  \path[draw=orange!90!black,line cap=butt,line join=miter, thick] (22.4, 6.6)..
   controls (22.6, 5.8) and (22.6, 4.0) .. (22.6, 4.0) -- (22.6, 4.0);
  \path[draw=red!80!black,line cap=butt,line join=miter,thick,miter limit=4.0]
   (17.7, 4.4).. controls (18.5, 4.6) and (20.8, 4.6) .. (21.1, 4.6).. controls 
  (21.4, 4.6) and (23.6, 4.7) .. (23.9, 4.8);
  \path[draw=red!80!black,line cap=butt,line join=miter,thick,miter limit=4.0]
   (17.5, 6.0).. controls (18.1, 6.1) and (20.2, 6.0) .. (21.1, 6.0).. controls 
  (21.9, 6.0) and (23.6, 6.1) .. (23.8, 6.2);
  
  \path[fill,line width=0.0cm,miter limit=4.0] (20.5, 6.0) circle (0.08cm) node[above right=0.1cm]{$p$};
  \path[fill,line width=0.0cm,miter limit=4.0] (18.5, 6.05) circle (0.08cm) node[above left=0.05cm]{$a$};
  \path[fill,line width=0.0cm,miter limit=4.0] (22.5, 6.05) circle (0.08cm) node[above right=0.05cm]{$w$};
  \path[fill,line width=0.0cm,miter limit=4.0] (20.6, 4.6) circle (0.08cm) node[below left=0.05cm]{$q$};
  \path[fill,line width=0.0cm,miter limit=4.0] (22.6, 4.67) circle (0.08cm) node[below right=0.1cm]{$b$};
  \path[fill,line width=0.0cm,miter limit=4.0] (18.6, 4.55) circle (0.08cm) node[below left=0.05cm]{$v$};
  
  \draw[orange!90!black!90!black] (20.7, 5.3) node{$\varepsilon$};
\end{tikzpicture}
    \caption{The points $a, b, p$ and $q$ to $\TT$ are obtained as the limit of $Y$-configurations in good matching. The points $v$ and $w$ are, respectively, the images of $p$ and $q$ under the unstable holonomy.}
    \label{fig:lim-main-lemma}
\end{figure}
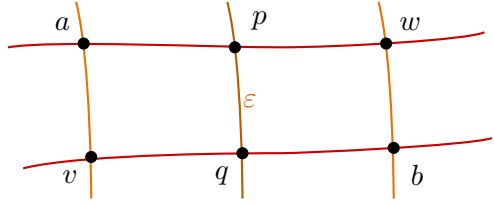

Consider the point $w\in\tW^c(b)$, which is the image of $a$ under the unstable holonomy. Since unstable holonomies are uniformly $C^1$ in restriction to center unstable manifolds, we have that $d^c(w,b)\asymp_M\eps$. Since $b\in\cL$, we have that $\cR^c_{b}(w)\asymp_Md^c(w,b)$. Finally, observe that $w\in\w^u_{\sigma^T}(a)$ and thus, by the basic move and the definition of $M$, we have that 
\[
\hat{\nu}^c_w\propto L_*\hat{\nu}^c_a,
\]
where $L$ is a linear map with slope $\asymp_M 1$. Combining these facts together we see that $\hat{\nu}^c_w\propto\hat{L}_*\hat{\nu}^c_b$, for some linear map $\hat{L}$ with slope $\asymp_{M^3} 1$ and $\cR^c_{b}(w)\asymp_{M^3}\eps$, which shows that $b\in K\cap G(\eps)$, a contradiction. \qed

\bibliographystyle{alpha}
\bibliography{Bib_CS}

\smallskip
\begin{flushleft}

{\scshape Marisa Cantarino}\\
School of Mathematical Sciences, Monash University,\\ Clayton, VIC 3800, Australia\\
email:  \texttt{marisa.cantarino@monash.edu}

	\vspace{0.2cm}
	
	{\scshape Bruno Santiago}\\
	Instituto de Matem\'atica e Estat\'istica, Universidade Federal Fluminense\\
	Rua Prof. Marcos Waldemar de Freitas Reis, S/N -- Bloco H, 4o Andar\\
	Campus do Gragoatá, Niterói, Rio de Janeiro 24210-201, Brasil\\
	email:  \texttt{brunosantiago@id.uff.br}

\end{flushleft}

\end{document}